\documentclass[11pt, a4paper]{amsart}

\usepackage{microtype}
\usepackage{times}
\AtBeginDocument{
  \DeclareSymbolFont{AMSb}{U}{msb}{m}{n}
  \DeclareSymbolFontAlphabet{\mathbb}{AMSb}
}

\usepackage[dvipsnames]{xcolor}
\usepackage{graphicx}
\usepackage{tikz}
\usetikzlibrary{positioning, shapes.geometric}
\definecolor{lightblue}{rgb}{0.38, 0.51, 0.71} % #6182B5
\definecolor{darkblue}{RGB}{17, 42, 60} % #112A3C
\definecolor{red}{RGB}{175, 49, 39} % #AF3127
\definecolor{yellow}{RGB}{250, 199, 100} % #FAC764
\definecolor{green}{RGB}{144, 169, 84} % #90A954

\newcommand{\linearSpline}[3]{
  \draw[black, thick] (#1, 0) -- (#2, 1) -- (#3, 0);
}
\newcommand{\linearSplineVert}[3]{
  \draw[black,thick,dashed] (#2, 0) -- (#2, 1);
}

\newcommand{\quadraticSpline}[4]{
  \draw[domain=#1:#2,smooth,thick] plot (\x, {(\x - #1)^2 / (#2 - #1) / (#3 - #1)});
  \draw[domain=#2:#3,smooth,thick] plot (\x, {(#3 - \x) / (#3 - #2) * (\x - #1) / (#3 - #1) + (\x - #2) / (#3 - #2) * (#4 - \x) / (#4 - #2)});
  \draw[domain=#3:#4,smooth,thick] plot (\x, {(#4 - \x)^2 / (#4 - #3) / (#4 - #2)});
}
\newcommand{\quadraticSplineVert}[4]{
  \draw[black,thick,dashed] (#2, 0) -- (#2, {(#2 - #1) / (#3 - #1)});
  \draw[black,thick,dashed] (#3, 0) -- (#3, {(#4 - #3) / (#4 - #2)});
}

\newcommand{\cubicSpline}[5]{
  \draw[domain=#1:#2,smooth,thick] plot (\x, {(\x-#1)^3/(#2-#1)/(#3-#1)/(#4-#1)});
  \draw[domain=#2:#3,smooth,thick] plot (\x, {(#3-\x)/(#3-#2)*(\x-#1)^2/(#3-#1)/(#4-#1)+(\x-#2)/(#3-#2)*(#4-\x)/(#4-#2)*(\x-#1)/(#4-#1)+(\x-#2)^2/(#3-#2)/(#4-#2)*(#5-\x)/(#5-#2)});
  \draw[domain=#3:#4,smooth,thick] plot (\x, {(#4-\x)^2/(#4-#3)/(#4-#2)*(\x-#1)/(#4-#1)+(#4-\x)/(#4-#3)*(\x-#2)/(#4-#2)*(#5-\x)/(#5-#2)+(\x-#3)/(#4-#3)*(#5-\x)^2/(#5-#3)/(#5-#2)});
  \draw[domain=#4:#5,smooth,thick] plot (\x, {(#5-\x)^3/(#5-#4)/(#5-#3)/(#5-#2)});
}
\newcommand{\cubicSplineVert}[5]{
  \draw[black,thick,dashed] (#2, 0) -- (#2, {(#2-#1)^2/(#3-#1)/(#4-#1)});
  \draw[black,thick,dashed] (#3, 0) -- (#3, {(#4-#3)/(#4-#2)*(#3-#1)/(#4-#1)+(#3-#2)/(#4-#2)*(#5-#3)/(#5-#2)});
  \draw[black,thick,dashed] (#4, 0) -- (#4, {(#5-#4)^2/(#5-#3)/(#5-#2)});
}

\newcommand{\hatfun}[6]{
  \draw[thick, black] (#1, #4) -- (#3, #4);
  \draw[thick, black] (#1, #5) -- (#3, #5);
  \draw[thick, black] (#1, #6) -- (#3, #6);

  \draw[thick, black] (#1, #4) -- (#1, #6);
  \draw[thick, black] (#2, #4) -- (#2, #6);
  \draw[thick, black] (#3, #4) -- (#3, #6);

  \filldraw[lightblue] (#1, #4) circle (2pt);
  \filldraw[lightblue] (#1, #5) circle (2pt);
  \filldraw[lightblue] (#1, #6) circle (2pt);
  \filldraw[lightblue] (#2, #4) circle (2pt);
  \filldraw[green] (#2, #5) circle (2pt);
  \filldraw[lightblue] (#2, #6) circle (2pt);
  \filldraw[lightblue] (#3, #4) circle (2pt);
  \filldraw[lightblue] (#3, #5) circle (2pt);
  \filldraw[lightblue] (#3, #6) circle (2pt);
}

\usepackage{caption}
\usepackage{subcaption}

\graphicspath{ {img/} {svg-inkscape/} }

\usepackage[a4paper, pdftex, left=2cm, top=2cm, right=2cm, bottom=2cm]{geometry}
\usepackage[final]{pdfpages}
\usepackage{changepage}
\usepackage[foot]{amsaddr}

\usepackage{url}
\usepackage{hyperref}
\hypersetup{
    breaklinks=true,
    bookmarksopen=true,
    pdftitle={Energy minimisation using overlapping tensor-product free-knot B-splines}, % Change here
    pdfauthor={Alexandre Magueresse, Santiago Badia}, % Change here
    pdfsubject={Energy minimisation using overlapping tensor-product free-knot B-splines}, % Change here
    pdfkeywords={Energy minimisation, Tensor-product B-splines, Free-knot B-splines, Mesh adaptivity}, % Change here
    colorlinks=true,
    linkcolor=blue,
    citecolor=blue,
    filecolor=blue,
    urlcolor=blue
}
\newcommand{\fig}[1]{Fig.~\ref{fig:#1}}

\newcommand{\alg}[1]{\hyperref[alg:#1]{Algorithm~\ref*{alg:#1}}}
\newcommand{\assum}[1]{\hyperref[assumption:#1]{Assumption~\ref*{assumption:#1}}}
\newcommand{\subassum}[2]{\hyperref[assumption:#1]{Assumption~\ref*{assumption:#1}.#2}}
\newcommand{\lem}[1]{\hyperref[lem:#1]{Lemma~\ref*{lem:#1}}}
\newcommand{\corol}[1]{\hyperref[corol:#1]{Corollary~\ref*{corol:#1}}}
\newcommand{\prop}[1]{\hyperref[prop:#1]{Proposition~\ref*{prop:#1}}}
\newcommand{\thm}[1]{\hyperref[thm:#1]{Theorem~\ref*{thm:#1}}}
\newcommand{\sect}[1]{\hyperref[sect:#1]{Section~\ref*{sect:#1}}}
\newcommand{\subsect}[1]{\hyperref[subsect:#1]{Subsection~\ref*{subsect:#1}}}
\newcommand{\app}[1]{\hyperref[app:#1]{Appendix~\ref*{app:#1}}}

\newcommand{\algopt}{{\color{purple} Minimisation Algorithm}}

\newcommand{\lemlocal}{{\color{purple} Lemma~3}} % Local convergence
\newcommand{\lemglobal}{{\color{purple} Lemma~5}} % Global convergence
\newcommand{\corolcea}{{\color{purple} Corollary~1}} % Global convergence
\newcommand{\lemregularity}{{\color{purple} Lemma~6}} % Energy regularity
\newcommand{\lemregularityreduced}{{\color{purple} Lemma~7}} % Reduced energy regularity

\newcommand{\assumspd}{{\color{purple} Assumption~1}} % SPD
\newcommand{\assumspace}{{\color{purple} Assumption~2}} % Compact and convex space
\newcommand{\assumlinboundreg}{{\color{purple} Assumption~3}} % Linear boundedness regularity
\newcommand{\assumenergyreg}{{\color{purple} Assumption~4a}} % Energy regularity
\newcommand{\assumnonlinboundreg}{{\color{purple} Assumption~4b}} % Nonlinear boundedness regularity
\newcommand{\assumdircon}{{\color{purple} Assumption~5}} % Directional convexity

\usepackage{csquotes}
\usepackage[backend=biber, defernumbers=false, maxbibnames=5, style=numeric-comp, sorting=none, isbn=false, bibencoding=utf8, safeinputenc, url=false, doi=true, giveninits=true]{biblatex}

\usepackage{float}
\usepackage{framed}
\usepackage{booktabs}
\usepackage{multirow}
\usepackage[inline]{enumitem}

\usepackage{amsthm}
\newtheorem*{remark}{Remark}

\newtheorem{lemma}{Lemma}
\newtheorem{corollary}{Corollary}

\newtheorem{definition}{Definition}

\usepackage{amsmath, amsfonts, amssymb, amscd}
\usepackage{mathtools}

\newcommand{\mb}[1]{\boldsymbol{#1}}
\newcommand{\mc}[1]{\mathcal{#1}}
\newcommand{\md}[1]{\mathbb{#1}}

\let\epsilon\varepsilon
\let\phi\varphi

\newcommand{\NN}{\mathbb{N}}
\newcommand{\RR}{\mathbb{R}}

\newcommand{\oo}[2]{\left(#1, #2\right)}

\newcommand{\cc}[2]{\left[#1, #2\right]}
\newcommand{\co}[2]{\left[#1, #2\right)}
\newcommand{\rg}[2]{\left\{#1 : #2\right\}}

\newcommand{\DX}{\nabla_{\md{X}}}

\newcommand{\abs}[1]{{\left\vert #1 \right\vert}}
\newcommand{\norm}[1]{{\vert\kern-0.25ex\vert #1 \vert\kern-0.25ex\vert}}
\newcommand{\inner}[2]{{\left\langle #1, #2 \right\rangle}}

\newcommand{\remBeg}{\triangleleft}
\newcommand{\remEnd}{\triangleright}

\newcommand{\red}[1]{#1^{\mathrm{red}}}

\newcommand{\dd}[1]{\ \mathrm{d} #1}

\usepackage{acronym}
\acrodef{pde}[PDE]{partial differential equation}
\acrodef{fem}[FEM]{finite element method}
\acrodef{iga}[IGA]{isogeometric analysis}
\acrodef{pinn}[PINN]{physics-informed neural network}
\acrodef{dof}[DOF]{degree of freedom}
\acrodefplural{dof}{degrees of freedom}
\acrodef{cg}[CG]{conjugate gradient}

\title[Energy minimisation using overlapping tensor-product free-knot B-splines]{Energy minimisation using overlapping \\ tensor-product free-knot B-splines}
\date{\today}

\author[A. Magueresse]{Alexandre Magueresse$^{\dagger\ast}$}
\address{$^\dagger$School of mathematics\\Monash university\\Clayton\\Victoria 3800\\Australia}
\email{alexandre.magueresse@monash.edu}
\author[S. Badia]{Santiago Badia$^{\dagger}$}
\email{santiago.badia@monash.edu}
\thanks{$^\ast$Corresponding author}

\begin{document}

\begin{abstract}
  Accurately solving \acp{pde} with localised features requires refined meshes that adapt to the solution. Traditional numerical methods, such as finite elements, are linear in nature and often ineffective for such problems, as the mesh is not tailored to the solution. Adaptive strategies, such as $h$- and $p$-refinement, improve efficiency by sequentially refining the mesh based on a posteriori error estimates. However, these methods are geometrically rigid---limited to specific refinement rules---and require solving the problem on a sequence of adaptive meshes, which can be computationally expensive. Moreover, the design of effective a posteriori error estimates is problem-dependent and non-trivial.
In this work, we study a specific nonlinear approximation scheme based on overlapping tensor-product free-knot B-spline patches, where knot positions act as nonlinear parameters controlling the geometry of the discretisation. We analyse the corresponding energy minimisation problem for linear, self-adjoint elliptic \acp{pde}, showing that, under a mild mesh size condition, the discrete energy satisfies the structural properties required for the local and global convergence of the constrained optimisation scheme developed in our companion work [Magueresse, Badia (2025, arXiv:2508.17687)]. This establishes a direct connection between the two analyses: the adaptive free-knot B-spline space considered here fits into the abstract framework, ensuring convergence of projected gradient descent for the joint optimisation of knot positions and coefficients. Numerical experiments illustrate the method's efficiency and its ability to capture localised features with significantly fewer degrees of freedom than standard finite element discretisations.
\end{abstract}

\maketitle

\section{Introduction}
\label{sect:introduction}
The numerical solution of \acfp{pde} modelling physical phenomena often exhibits highly localised features, such as sharp gradients, boundary or internal layers, and singularities. These features are typically confined to small regions of the physical domain, while the solution remains smooth elsewhere. In most practical applications, analytical solutions are unavailable, so numerical methods must be employed to obtain approximate solutions. A critical requirement in this context is that the numerical method can accurately capture and resolve such local structures without incurring prohibitively high computational costs.

The \ac{fem} has long been the method of choice for the discretisation of \acp{pde}, valued for its flexibility in handling complex geometries and boundary conditions, as well as its rigorous mathematical foundation \cite{ciarlet2002finite}. However, when the \ac{fem} is applied to problems with sharp local features, it often requires either highly refined meshes or high-order polynomial bases to achieve sufficient accuracy. In both cases, the number of \acp{dof} may grow exponentially with the dimension, making the method computationally expensive. This challenge, known as the curse of dimensionality, has motivated the development of adaptive techniques that concentrate computational effort where it is most needed.

Traditional numerical methods for \acp{pde}, such as the \ac{fem}, operate within a fundamentally linear framework: the numerical approximation of the solution is expressed as a linear combination of preselected basis functions, with the coefficients chosen to be optimal in some sense. While this structure provides powerful theoretical and computational tools, it lacks the flexibility to adapt the geometry of the approximation space. To address this limitation, adaptive strategies have been developed to refine the mesh or enrich the polynomial basis in response to the solution's local features.

The most common adaptive strategies are $h$-refinement (local mesh refinement) and $p$-refinement (local polynomial enrichment), both guided by a posteriori error estimates tailored to the problem \cite{babuvska1994p}. These methods employ a posteriori error estimators to direct the refinement process and have been shown to greatly enhance efficiency while maintaining accuracy \cite{ainsworth1997posteriori}. Nevertheless, both approaches ultimately increase the number of \acp{dof}, albeit in a more targeted manner. Moreover, because they are tied to a fixed mesh topology, they remain relatively rigid and become increasingly inefficient in high dimensions, where the curse of dimensionality and lack of geometric flexibility limit their effectiveness.

An alternative form of adaptivity, known as mesh relocation or $r$-adaptivity, addresses this limitation by redistributing the nodes of the mesh to align with features of the solution, rather than introducing new elements or increasing the polynomial degree. The redistribution is typically guided by a monitor function, which is often derived from a posteriori error indicators, but may also be informed by a priori estimates involving higher-order derivatives or curvature measures of the solution or its reconstruction \cite{huang2010adaptive}. Once the monitor function is defined, a mesh functional is constructed to encode desirable properties of the mesh, such as alignment with anisotropic features, equidistribution of error indicators, and avoidance of mesh tangling. The mesh is then obtained by minimising this functional, subject to geometric and topological constraints \cite{huang2005metric}.

Despite its promise, $r$-adaptivity presents significant challenges. Ensuring mesh regularity and preventing element inversion are non-trivial tasks, and the resulting approximation space depends nonlinearly on the relocation parameters. These issues necessitate novel algorithmic strategies and theoretical analyses. Nevertheless, $r$-methods have demonstrated their potential to capture localised phenomena with fewer \acp{dof} compared to purely $h$- or $p$-based approaches \cite{budd2009adaptivity}.

The idea of adaptivity has also gained traction in the context of \acp{pinn} \cite{raissi2019physics}, which approximate \ac{pde} solutions by training neural networks to minimise a residual-based loss combining the governing equations, boundary conditions, and available data. By operating on expressive, nonlinear function spaces, \acp{pinn} avoid meshing and allow both geometric and functional adaptation. However, they face significant challenges \cite{wang2021understanding}: costly numerical integration, difficulties in enforcing boundary conditions, and non-convex optimisation landscapes often prone to ill-posedness. Theoretical guarantees on convergence remain limited \cite{de2024numerical, beck2022full}, especially compared to classical Galerkin methods. The Deep Ritz Method \cite{weinan2018deep} exemplifies another branch of nonlinear approximation for variational problems, where neural networks parameterise trial functions minimising energy functionals directly.

In contrast, methods based on free-knot splines offer a middle ground between geometric adaptivity and numerical tractability. Free-knot spline are well established in one dimension and have been primarily applied to data fitting \cite{jupp1978approximation, schwetlick1995least, beliakov2004least, kovacs2019nonlinear}. They combine smoothness, local support, and adaptive knot placement to capture local features effectively, making them powerful nonlinear approximators in 1D. Extensions to two dimensions using tensor-product B-splines have been investigated for data fitting \cite{schutze2003bivariate, deng2004optimizing, zhang2016b, yeh2020efficient}, typically relying on heuristic adaptivity or constrained nonlinear optimisation to enforce a minimum mesh size. These approaches are restricted to a single tensor-product patch and remain limited by the curse of dimensionality. More recently, \cite{omella2024r} combined tensor-product free-knot linear splines with neural networks for \ac{pde} approximation, using a spline ansatz that interpolates the network at the knot locations. Their two-stage training---first fixing the knots while training the network, then adjusting the knots with the network---showed only modest benefits from the introduction of free knots. A related method is the shallow Ritz approach of \cite{cai2024efficient}, which employs a ReLU neural network with adaptive breakpoints to emulate the space of free-knot linear splines for 1D diffusion problems. Linear and nonlinear parameters are alternately updated via a damped block Newton method, achieving near-optimal knot placement. While effective in one dimension, the method is limited to linear free-knot splines and lacks a general convergence analysis.

\subsection*{Contributions}
We introduce a framework based on tensor-product free-knot spline patches, designed to fit within the abstract convergence theory of our companion paper \cite{companion}. We also proposed a second method that makes use of sums of possibly overlapping lower-dimensional patches to overcome the rigidity of standard tensor-product spline spaces (modifying one knot affects many basis functions globally). By optimising knot placement across overlapping patches, the method achieves geometric adaptivity while maintaining computational tractability, enabling local refinement without globally increasing the number of \acp{dof}. In particular, it can emulate non-truncated hierarchical B-spline methods \cite{giannelli2012thb}.

In the single-patch case, we prove a nonlinear Céa's lemma, guaranteeing convergence to a global minimiser under a directional convexity assumption and suitable initialisation. In the multi-patch setting, we establish local convergence results; however, global convergence is harder to ensure due to the difficulty of enforcing a uniform minimum mesh size across all patches.

Compared to \acp{pinn}, free-knot spline methods offer a more structured and theoretically grounded framework, preserving compatibility with classical tools from variational analysis, approximation theory, and numerical linear algebra. They also give direct access to gradients and allow exact enforcement of boundary conditions, making them a strong candidate for adaptive numerical methods with rigorous guarantees.

\clearpage

We present a series of numerical experiments assessing the approximation capabilities of free-knot spline spaces in solving elliptic \acp{pde}. The results highlight the potential of our approach to achieve significantly better accuracy than the standard \ac{fem} using fewer \acp{dof}, demonstrating the effectiveness of both geometric adaptivity and energy-based optimisation in resolving localised features with high precision. While our experiments focus on low-dimensional settings, the method extends naturally to higher-dimensional problems.

\section{Tensor-product free-knot B-splines}
\label{sect:patches}
The convergence guarantees established in our companion work \cite{companion} are conditional on the ability to satisfy key hypotheses, namely the Lipschitz continuity of the energy gradient, directional convexity of the energy in a neighbourhood of the global minimisers, and an initialisation within that neighbourhood. In this work, we construct two nonlinear approximation spaces that fall within this framework: one based on tensor-product free-knot B-splines, and another consisting of sums of such spline patches.

We introduce tensor-product free-knot spaces by taking tensor products of univariate free-knot B-spline spaces. To this end, we first revisit the construction of univariate B-splines, recalling both the Cox--De Boor recursion and the divided difference definition, and highlighting the distinct advantages of each. This sets the stage for the multivariate construction via tensor products. Finally, we extend the framework by considering sums of such tensor-product patches rather than a single one, which enhances nonlinear approximation properties while maintaining reasonable computational complexity, particularly in higher dimensions.

For all $n \geq 1$, define
$$\md{K}_{n} \doteq \left\{\mb{\tau} \in \RR^{n}, \mb{\tau}_{1} \leq \mb{\tau}_{2} \leq \ldots \leq \mb{\tau}_{n}\right\}$$
the set of non-decreasing sequences of $n$ numbers (knot vectors). While most references introduce B-splines via their order, we refer to them via their degree (order minus one).

\subsection{Cox--De Boor recursion}

The construction of B-splines through the Cox--De Boor recursion is useful for practical implementation and to derive elementary properties of B-splines. We follow the construction of \cite{deboor1978practical} and adapt the notations to refer to B-splines via their degree. Let $n \geq 1$ and $\mb{\tau} \in \md{K}_{n}$. For $a \leq b \in \NN$, let $\rg{a}{b}$ denote the set $\{a, \ldots, b\}$. For $i \in \rg{1}{n-1}$, the $i$-th B-spline of degree $0$ on the knot vector $\mb{\tau}$ is defined by
$$B_{i, 0}(\mb{\tau})(x) = \left\{\begin{array}{cl}
        1 & \text{if } \mb{\tau}_{i} \leq x < \mb{\tau}_{i+1}, \\
        0 & \text{otherwise}.
    \end{array}\right.$$
In other words, $B_{i, 0}(\mb{\tau})$ is the indicator function of the interval $\co{\mb{\tau}_{i}}{\mb{\tau}_{i+1}}$. B-splines of higher degree are defined via the Cox--De Boor recursion
$$B_{i, p}(\mb{\tau})(x) = \frac{x - \mb{\tau}_{i}}{\mb{\tau}_{i+p} - \mb{\tau}_{i}} B_{i, p-1}(\mb{\tau})(x) + \frac{\mb{\tau}_{i+p+1} - x}{\mb{\tau}_{i+p+1} - \mb{\tau}_{i+1}} B_{i+1, p-1}(\mb{\tau})(x),$$
for all $p \geq 1$ and $i \in \rg{1}{n-(p+1)}$. By convention, any term with a zero denominator is interpreted as zero. This recursive process continues up to the desired degree.

\subsection{Divided differences}

An alternative route to defining B-splines is through divided differences, which proves particularly useful for expressing and analysing their spatial and parametric derivatives. We follow the construction of \cite{schumaker2007spline} and adapt notations to refer to B-splines via their degree. Let $p \geq 0$ and $i \in \rg{1}{n-(p+1)}$. The $i$-th B-spline of degree $p$ can also be defined as
$$B_{i, p}(\mb{\tau})(x) = (\mb{\tau}_{i+p+1} - \mb{\tau}_{i}) [\mb{\tau}_{i}, \ldots, \mb{\tau}_{i+p+1}] (\bullet - x)_{+}^{p},$$
where $(\bullet - x)_{+}$ denotes the function $y \mapsto (y - x)_{+} = \max(0, y - x)$ and $[\mb{\tau}_{i}, \ldots, \mb{\tau}_{i+p+1}]$ denotes the divided difference operator of order $p+1$ over the points $\mb{\tau}_{i}, \ldots, \mb{\tau}_{i+p+1}$. We refer to \app{toolsbsplines} for more details regarding divided references. In this definition, the convention for $p = 0$ is $(x)_{+}^{0} = H(x)$, where $H$ is the Heaviside function.

\subsection{Properties of B-splines}
\label{subsect:spline-properties}

To clarify the notation and the mathematical nature of the objects involved, we note that $B_{i, p}(\mb{\tau}): \RR \to \RR$ is a function of the spatial variable given a fixed knot vector $\mb{\tau}$. In contrast, $B_{i, p}: \md{K}_{n} \to (\RR \to \RR)$ is a mapping that takes a knot vector as input and returns a function of the spatial variable. For degrees $p \in \{1, 2, 3\}$, we refer to $B_{i, p}$ respectively as linear, quadratic, or cubic B-splines.

Given $p \geq 0$ and $\mb{\tau} \in \md{K}_{n}$, we form the space
$$\md{S}_{n, p}(\mb{\tau}) \doteq \mathrm{Span}\left(B_{i, p}(\mb{\tau}), i \in \rg{1}{n-(p+1)}\right)$$
consisting of B-splines of degree $p$ associated with the knot vector $\mb{\tau}$. Since a knot vector of size $n$ generates exactly $n-(p+1)$ B-splines of degree $p$, it is necessary to have $n \geq p + 2$ for the space $\md{S}_{n, p}(\mb{\tau})$ to be non-empty. Examples of such spaces are depicted on \fig{splinePatches}, for linear and quadratic B-splines.

\begin{figure}
    \begin{subfigure}[t]{0.48\linewidth}
        \centering
        \begin{tikzpicture}[scale=2]
            \draw[black, thick, -stealth] (-2, 0) -- (+1.6, 0);

            \linearSpline{-1.8}{-1.0}{-0.5}
            \linearSpline{-1.0}{-0.5}{-0.1}
            \linearSpline{-0.5}{-0.1}{0.3}
            \linearSpline{-0.1}{0.3}{1.4}

            \foreach \x in {-1.8, -1.0, -0.5, -0.1, 0.3, 1.4}
            \filldraw[black] (\x, 0) circle (1pt);

            \foreach \x [count=\xk] in {-1.8, -1.0, -0.5, -0.1, 0.3, 1.4}
            \node at (\x, -0.35) [anchor=south] {$\mb{\tau}_{\xk}$};
        \end{tikzpicture}
        \caption{Basis for $\md{S}_{4, 1}(\mb{\tau})$.}
        \label{fig:linearSplinePatch}
    \end{subfigure}%
    \hfill%
    \begin{subfigure}[t]{0.48\linewidth}
        \centering
        \begin{tikzpicture}[scale=2]
            \draw[black, thick, -stealth] (-2, 0) -- (+1.6, 0);

            \quadraticSpline{-1.8}{-0.7}{-0.5}{-0.1}
            \quadraticSpline{-0.7}{-0.5}{-0.1}{1.1}
            \quadraticSpline{-0.5}{-0.1}{1.1}{1.4}

            \foreach \x in {-1.8, -0.7, -0.5, -0.1, 1.1, 1.4}
            \filldraw[black] (\x, 0) circle (1pt);

            \foreach \x [count=\xk] in {-1.8, -0.7, -0.5, -0.1, 1.1, 1.4}
            \node at (\x, -0.35) [anchor=south] {$\mb{\tau}_{\xk}$};
        \end{tikzpicture}
        \caption{Basis for $\md{S}_{3, 2}(\mb{\tau})$.}
        \label{fig:quadraticSplinePatch}
    \end{subfigure}%
    \caption{Some free-knot B-spline spaces.}
    \label{fig:splinePatches}
\end{figure}
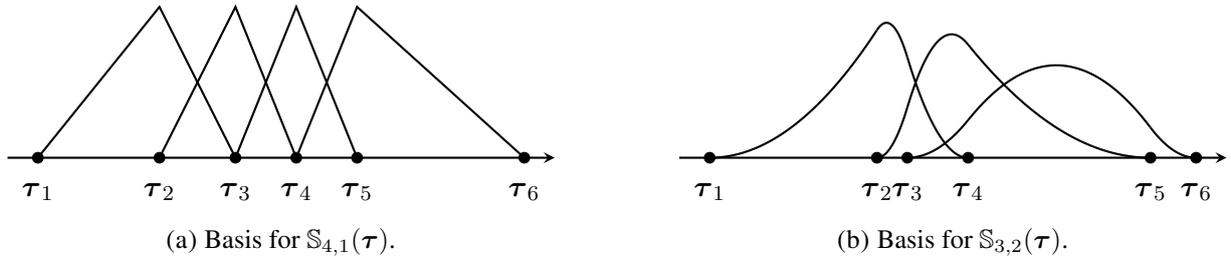

We briefly state some important properties of B-splines (see \cite{deboor1978practical} for their proof). For all $i \in \rg{1}{n - (p + 1)}$,
\begin{enumerate}
    \item $B_{i, p}(\mb{\tau})$ is piecewise polynomial of degree $p$.
    \item $B_{i, p}(\mb{\tau})$ is compactly supported in the interval $\cc{\mb{\tau}_{i}}{\mb{\tau}_{i+p+1}}$. Consequently, the only non-zero B-splines in $\cc{\mb{\tau}_{i}}{\mb{\tau}_{i+1}}$ are those $B_{j, p}(\mb{\tau})$ with indices $j \in \rg{i-(p+1)}{i} \cap \rg{1}{n-(p+1)}$.
    \item $B_{i, p}(\mb{\tau})$ takes values in $\cc{0}{1}$.
    \item For each knot $\mb{\tau}_{i + k}$ with $k \in \rg{0}{p+1}$, the B-spline $B_{i, p}(\mb{\tau})$ is $p - \mu(\mb{\tau}_{i+k})$ times continuously differentiable at $\mb{\tau}_{i+k}$, where $\mu(\mb{\tau}_{i+k}) \geq 1$ denotes the multiplicity of the knot $\mb{\tau}_{i+k}$ within the subsequence $(\mb{\tau}_{i}, \ldots, \mb{\tau}_{i+p+1})$.
    \item The collection $(B_{j, p}(\mb{\tau}))_{j \in \rg{1}{n-(p+1)}}$ is a partition of unity on the interval $\cc{\mb{\tau}_{p+1}}{\mb{\tau}_{n-p}}$.
\end{enumerate}

We briefly examine the conformity of $\md{S}_{n, p}(\mb{\tau})$ with Sobolev spaces. By properties (2) and (3), each B-spline $B_{i, p}(\mb{\tau})$ is compactly supported and bounded. With simple (non-repeating) knots, property (4) ensures this implies that each $B_{i, p}(\mb{\tau})$ belongs to $C^{p-1}(\RR)$ for all $i \in \rg{1}{n - (p+1)}$. Differentiating the Cox--De Boor recursion and proceeding by induction, one shows that the derivatives of $B_{i, p}(\mb{\tau})$ are combinations of lower-degree B-splines scaled by inverse knot spacings. Since the knot spacings are non-zero by assumption, it follows that $B_{i, p}(\mb{\tau}) \in H^{p}(\RR)$ and thus $\md{S}_{n, p}(\mb{\tau}) \subset H^{p}(\RR)$.

\subsection{Tensor-product patch}

We now extend the one-dimensional construction to higher dimensions via a tensor-product approach. In multiple dimensions, tensor-product B-spline functions are defined over rectangular regions determined by the cartesian product of univariate knot vectors. These regions, where the associated basis functions are non-zero, are commonly referred to as patches. They form the local supports of the multivariate spline basis and serve as the elementary units for constructing flexible, piecewise polynomial representations with free-knot placement.

Let $d \geq 1$ denote the dimension, and let $\mb{p}, \mb{n} \in \NN^{d}$ represent the degrees and knot vector sizes along each axis, respectively. We define the set of multivariate knot vectors and \acp{dof} as
$$\md{K}_{\mb{n}} \doteq \prod_{t = 1}^{d} \md{K}_{\mb{n}_{t}}, \qquad \md{W}_{\mb{n}} \doteq \bigotimes_{t = 1}^{d} \RR^{\mb{n}_{t} - (\mb{p}_{t} + 1)}.$$
For each $\mb{T} = (\mb{T}_{1}, \ldots, \mb{T}_{d}) \in \md{K}_{\mb{n}}^{d}$, we define the corresponding tensor-product B-spline space as
$$\md{S}_{\mb{n}, \mb{p}}(\mb{T}) \doteq \bigotimes_{t = 1}^{d} \md{S}_{\mb{n}, \mb{p}}(\mb{T}_{t}) = \mathrm{Span}\left(B_{\mb{i}, \mb{p}}(\mb{T}), \mb{i}_{t} \in \rg{1}{\mb{n}_{t} - (\mb{p}_{t} + 1)}, t \in \rg{1}{d}\right),$$
where the multivariate basis functions are given by
$$B_{\mb{i}, \mb{p}}(\mb{T})(\mb{x}) \doteq \prod_{t = 1}^{d} B_{\mb{i}_{t}, \mb{p}_{t}}(\mb{T}_{t})(\mb{x}_{t}).$$
Each such patch $\md{S}_{\mb{n}, \mb{p}}(\mb{T})$ is a space of piecewise polynomials with global continuity of order $\min \mb{p} - 1$ in $\RR^{d}$.

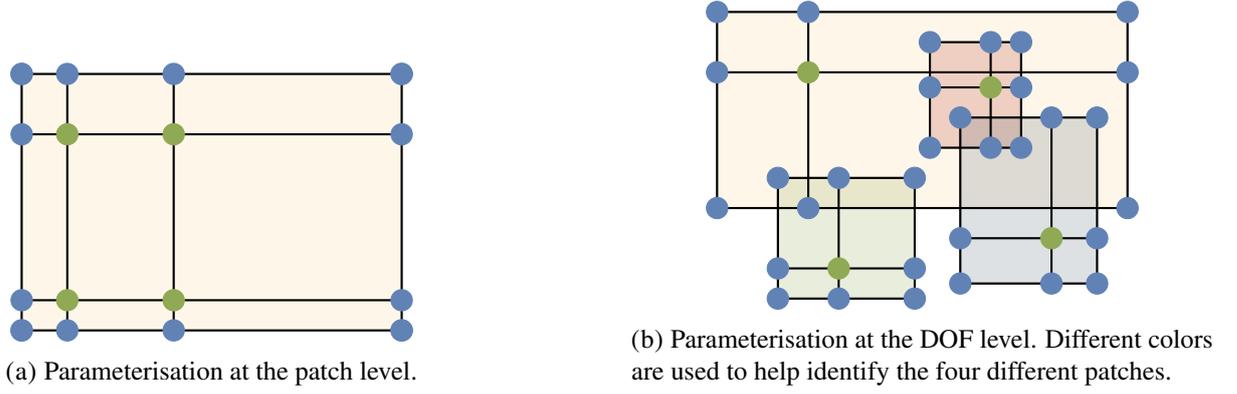
\begin{figure}
    \begin{subfigure}[b]{0.45\linewidth}
        \centering
        \begin{tikzpicture}[scale=2]
            \def \xa {-1}
            \def \xb {-0.7}
            \def \xc {0}
            \def \xd {1.5}

            \def \ya {-1}
            \def \yb {-0.8}
            \def \yc {0.3}
            \def \yd {0.7}

            \draw[fill=yellow!70, fill opacity=0.2] (\xa, \ya) rectangle (\xd, \yd);

            \foreach \x in {\xa, \xb, \xc, \xd}{
                    \draw[thick, black] (\x, \ya) -- (\x, \yd);
                }

            \foreach \y in {\ya, \yb, \yc, \yd}{
                    \draw[thick, black] (\xa, \y) -- (\xd, \y);
                }

            \foreach \x [count=\col] in {\xa, \xd}{
                    \foreach \y [count=\row] in {\ya, \yb, \yc, \yd}{
                            \filldraw[lightblue] (\x, \y) circle (2pt);
                        }
                }

            \foreach \x [count=\col] in {\xa, \xb, \xc, \xd}{
                    \foreach \y [count=\row] in {\ya, \yd}{
                            \filldraw[lightblue] (\x, \y) circle (2pt);
                        }
                }

            \foreach \x [count=\col] in {\xb, \xc} {
                    \foreach \y [count=\row] in {\yb, \yc} {
                            \filldraw[green] (\x, \y) circle (2pt);
                        }
                }
        \end{tikzpicture}
        \caption{Parameterisation at the patch level.}
        \label{fig:parameterisation_2D_axis}
    \end{subfigure}%
    \hfill%
    \begin{subfigure}[b]{0.45\linewidth}
        \centering
        \begin{tikzpicture}[scale=2]
            \draw[fill=yellow!70, fill opacity=0.2] (-0.8, 0.2) rectangle (1.9, 1.5);
            \draw[fill=green, fill opacity=0.2] (-0.4, -0.4) rectangle (0.5, 0.4);
            \draw[fill=red, fill opacity=0.2] (0.6, 0.6) rectangle (1.2, 1.3);
            \draw[fill=darkblue!70, fill opacity=0.2] (0.8, -0.3) rectangle (1.7, 0.8);

            \hatfun{-0.8}{-0.2}{1.9}{0.2}{1.1}{1.5}
            \hatfun{-0.4}{0}{0.5}{-0.4}{-0.2}{0.4}
            \hatfun{0.6}{1}{1.2}{0.6}{1}{1.3}
            \hatfun{0.8}{1.4}{1.7}{-0.3}{0}{0.8}
        \end{tikzpicture}
        \caption{Parameterisation at the \ac{dof} level. Different colors are used to help identify the four different patches.}
        \label{fig:parameterisation_2D_dof}
    \end{subfigure}
    \caption{Two parameterisations of free-knot spaces containing $4$ \acp{dof} in dimension $2$. Green dots correspond to the points directly controlled by the \acp{dof}, while blue points indicate the boundary of the supports of the corresponding basis functions.}
\end{figure}

In particular, the space $\md{S}_{\mb{n}, \mb{1}}(\mb{T})$ contains standard first-order Lagrange finite element spaces defined on axis-aligned, non-uniform grids over $d$-dimensional cubes. \fig{parameterisation_2D_axis} shows an example in dimension two when $\mb{n} = (2, 2)$.

\subsection{Sums of tensor-product patches}
\label{subsect:stp-patches}

While tensor-product spline spaces $\md{S}_{\mb{n}, \mb{p}}$ are convenient to work with numerically and already emulate $r$-adaptive tensor-product free-knot B-splines, they are inherently rigid from a mesh adaptation perspective: modifying a single knot along one axis influences a large number of basis functions. To gain greater local control and flexibility, we now consider function spaces formed by sums of tensor-product patches.

This construction resembles non-truncated hierarchical B-spline methods that build adaptive spaces by stacking tensor-product patches across levels \cite{giannelli2012thb}. Our approach can be seen as a flexible counterpart, where the patches are free to move.

Let $r \geq 1$ denote a number of tensor-product patches. For each patch $s \in \rg{1}{r}$, let $\mb{P}_{s} \in \NN^{d}$ denote the degrees of patch number $s$ along each axis and $\mb{N}_{s} \in \NN^{d}$ denote the number of knots for patch $s$ along each axis. Arranging $(\mb{N}_{1}, \ldots, \mb{N}_{r})$ into $\mb{N} \in \NN^{r \times d}$ and similarly for $\mb{P}$, we define the corresponding sets of knot vectors and \acp{dof} as
$$\md{K}_{\mb{N}} \doteq \prod_{s = 1}^{r} \md{K}_{\mb{N}_{s}}, \qquad \md{W}_{\mb{N}, \mb{P}} \doteq \prod_{s = 1}^{r} \md{W}_{\mb{N}_{s}, \mb{P}_{s}}.$$
Given $\mc{T} \in \md{K}_{\mb{N}}$, we define the corresponding space of B-splines as a sum of tensor-product patches
$$\md{S}_{\mb{N}, \mb{P}}(\mc{T}) \doteq \bigoplus_{s = 1}^{r} \md{S}_{\mb{N}_{s}, \mb{P}_{s}}(\mc{T}_{s}).$$
This construction enables local adaptivity by allowing each patch to refine its knot vector independently, without affecting the continuity or structure of adjacent patches. As a result, resolution can be concentrated in regions where the solution exhibits sharp features or high variation, while coarser mesh sizes suffice elsewhere.

Finally, we consider the union of these sums of tensor-product patches over all possible knot vectors. Formally, for given $r \geq 1$ and degree and knot count arrays $\mb{P}, \mb{N} \in \NN^{r \times d}$, we define the approximation space
$$\md{S}_{\mb{N}, \mb{P}} \doteq \left\{\md{S}_{\mb{N}, \mb{P}}(\mc{T}), \mc{T} \in \md{K}_{\mb{N}}\right\}.$$
This space encompasses all spline functions arising from varying knot configurations within the prescribed patch structure. More concretely, $\md{S}_{\mb{N}, \mb{P}}$ is the image of the realisation map $\mc{R}_{\mb{N}, \mb{P}}: \md{W}_{\mb{N}, \mb{P}} \times \md{K}_{\mb{N}} \to (\RR^{d} \to \RR)$ defined by
$$\mc{R}_{\mb{N}, \mb{P}}(\mc{W}, \mc{T}) = \sum_{s = 1}^{r} \sum_{\mb{i} = \mb{1}}^{\mb{N}_{s} - (\mb{P}_{s} + 1)} \mc{W}_{s, \mb{i}} B_{\mb{i}, \mb{N}_{s}}(\mc{T}_{s}).$$
We refer to $\mc{W} \in \md{W}_{\mb{N}, \mb{P}}$ as the linear parameters (degrees of freedom) and $\mc{T} \in \md{K}_{\mb{N}}$ as nonlinear (mesh) parameters. The dimensions of $\md{W}_{\mb{N}, \mb{P}}$ and $\md{K}_{\mb{N}}$ are
$$\dim \md{W}_{\mb{N}, \mb{P}} = \sum_{r = 1}^{r} \prod_{t = 1}^{d} (\mb{N}_{r, t} - (\mb{P}_{r, t} + 1)), \qquad \dim \md{K}_{\mb{N}} = \sum_{r = 1}^{r} \sum_{t = 1}^{d} \mb{N}_{r, t}.$$
The resulting space $\md{S}_{\mb{N}, \mb{P}}$ has global continuity of order $\min \mb{P} - 1$ in $\RR^{d}$, so it is contained in the Sobolev space $H^{k}(\Omega)$ for all $k \in \rg{0}{\min \mb{P}}$. In particular, $\md{S}_{\mb{N}, \mb{P}}$ is $H^{1}$-conforming whenever $\min \mb{P} \geq 1$.

For a fixed number of basis functions, most adaptivity can be achieved when the corresponding basis functions are as decoupled as possible. In the extreme case, this corresponds to parameterising each basis function independently by using patches that contain exactly one basis function; that is, when $\mb{n} = \mb{p} + \mb{2}$. An example of such a space is represented on \fig{parameterisation_2D_dof} for linear splines. This parameterisation at the \ac{dof} level offers significant flexibility, as the supports of the basis functions become fully independent. It is also clear that the space parameterised at the \ac{dof} level strictly contains the space parameterised at the patch level.

This level of independence is reminiscent of radial basis function methods \cite{buhmann2000radial, fornberg2015primer}, where each basis function is centred and shaped individually. Such methods also offer highly localised and flexible approximation spaces, though they typically lack the compact support and tensor-product structure present here.

From a numerical standpoint, the \ac{dof}-level parameterisation is unwieldy because it lacks any structured overlap among basis functions. When adapting the mesh to minimise an energy functional, the support of every basis function may change at each iteration. This makes evaluating bilinear forms very expensive, since one must account for the intersections between the supports of all pairs of basis functions. Similar issues arise in radial basis function methods, which likewise require evaluating dense matrices unless special localisation, truncation or sparsification techniques are used \cite{schaback2007practical}. Moreover, as the mesh parameters vary, the pattern of which basis functions overlap changes unpredictably. In other words, the sparsity pattern of the matrices representing these bilinear forms may change at every iteration during mesh optimisation.

In contrast, the connectivity pattern within a single patch is independent of the knot vector. This provides a structured overlap that can be exploited for efficient numerical computations. For instance, in one dimension, each basis function $B_{i, p}$ only overlaps with those $B_{j, p}$ where $j \in \{i-p, i+p\}$. As a result, the matrix representation of a one-dimensional bilinear form has a banded structure with bandwidth $2 p + 1$, enabling efficient assembly and storage.

We explore several choices of space architecture ($r$, $\mb{N}$) in our experiments, ranging from a single patch with many basis functions ($r = 1$, $\mb{N} = (\mb{n},)$, $\mb{n} \gg 1$, most coupled parameterisation), to many patches each containing a single basis function ($r \gg 1$, $\mb{N} = (\mb{1}, \ldots, \mb{1})$, most decoupled parameterisation).

\section{Numerical analysis of tensor-product free-knot B-spline spaces}
\label{sect:analysis}
In this section, we specialise the abstract framework developed in our companion work \cite{companion} to the concrete setting of overlapping tensor-product free-knot B-spline patches. References to the algorithm, assumptions, and lemmas from our companion work are given in \textcolor{purple}{red}, following their original numbering.

We fix an approximation space $V = \md{S}_{\mb{N}, \mb{P}}$ for some number of patches $r \geq 1$, with knot numbers and degrees $\mb{N}, \mb{P} \in \NN^{r \times d}$. Following the notation of our companion work, we write $\md{W} = \md{W}_{\mb{N}, \mb{P}}$ and $\md{X} = \md{K}_{\mb{N}, \mb{P}}$. The basis functions $B_{\mb{i}, \mb{N}_{s}}$ are denoted by $(\mb{\phi}_{i})$ for a suitable enumeration. Finally, the realisation map is simply written $\mc{R}$ instead of $\mc{R}_{\mb{N}, \mb{P}}$.

Our goal is to minimise an energy functional $\mc{J}: U \to \RR$ in $V$, assuming $V \subset U$, of the form
$$\mc{J}(u) \doteq \frac{1}{2} a(u, u) - \ell(u).$$
Here $a: U \times U \to \RR$ is a symmetric, coercive and continuous bilinear form, $\ell: U \to \RR$ is a continuous linear form and $U$ is a suitable Hilbert space. Let $\mc{K} \doteq \mc{J} \circ \mc{R}: \md{W} \times \md{X} \to \RR$ denote the discrete energy functional.

In this section, we begin by formalising the tensor-product structure assumed for the approximation space $U$, as well as for the associated linear and bilinear forms. This is not a restrictive assumption, as any multilinear form can be approximated arbitrarily well by a finite sum of tensor-product terms, with the number of terms controlling the approximation rank \cite{hackbusch2012tensor}. This separability simplifies the regularity analysis of the discrete energy by reducing it to the one-dimensional setting.

Next, we study the B-spline basis through the B-spline functor $B_p: \md{K}_{p+2} \to (\RR \to \RR)$, deriving explicit bounds and continuity properties of B-splines and their spatial, parametric, and mixed derivatives. Under a uniform minimum mesh size assumption, we prove that B-spline functions are uniformly bounded and Hölder continuous with respect to the knot vectors, achieving Lipschitz continuity when the spline degree is sufficiently high. These regularity results ensure that the discrete energy functional satisfies the required smoothness and continuity assumptions (\assumenergyreg{}).

We then invoke classical results on the invertibility and conditioning of discretisation matrices to show that a uniform lower bound on the global mesh size guarantees coercivity of the quadratic forms, establishing the uniform positive-definiteness condition (\assumspd{}).

To enforce the minimum mesh size constraint, we characterise the nonlinear parameter space as a compact and convex set (\assumspace{}). Maintaining this constraint throughout optimisation reduces to projecting onto this set, which can be formulated as a quadratic program with quadratic inequality constraints.

Finally, we briefly discuss the approximation properties of tensor-product free-knot B-splines and their connection to hierarchical spline constructions, highlighting the expressiveness and adaptivity of the proposed discretisation space.

\subsection{Tensor-product structure}
We recall the definition of tensor-product domain, functional space, linear and bilinear forms.

\begin{definition}[Tensor-product domain]
    We say that $\Omega \subset \RR^{d}$ is a \emph{tensor-product domain} if there exist bounded intervals $(\Omega_{t})_{1 \leq t \leq d}$ such that $\Omega = \prod_{t = 1}^{d} \Omega_{t}$.
\end{definition}

\begin{definition}[Tensor-product functional space]
    Assuming $\Omega$ is a tensor-product domain, we say that $U$ is a \emph{tensor-product functional space} if there exist functional spaces $(U_{t})_{1 \leq t \leq d}$, with $U_{t}$ a set of functions from $\Omega_{t}$ to $\RR$, such that $\bigotimes_{t = 1}^{d} U_{t} \subset U$.
\end{definition}

\begin{definition}[Separable linear form]
    Assuming $U$ is a tensor-product functional space, we say that a linear form $\ell: U \to \RR$ is \emph{separable} if there exist $r \geq 1$ and linear forms $(\ell_{s, t})_{s \in \rg{1}{r}, t \in \rg{1}{d}}$ with $\ell_{s, t}: U_{t} \to \RR$, such that $\ell = \sum_{s = 1}^{r} \bigotimes_{t = 1}^{d} \ell_{s, t}$; that is for all for all $u = \bigotimes_{t = 1}^{d} u_{t}$ with $u_{t} \in U_{t}$, $\ell$ decomposes as
    $$\ell(u) = \sum_{s = 1}^{r} \prod_{t = 1}^{d} \ell_{s, t}(u_{t}).$$
\end{definition}

\begin{definition}[Separable bilinear form]
    Assuming $U$ is a tensor-product functional space, we say that a bilinear form $a: U \times U \to \RR$ is \emph{separable} if there exist $r \geq 1$ and bilinear forms $(a_{s, t})_{s \in \rg{1}{r}, t \in \rg{1}{d}}$ with $a_{s, t}: U_{t} \times U_{t} \to \RR$, such that $a = \sum_{s = 1}^{r} \bigotimes_{t = 1}^{d} a_{s, t}$; that is, for all $u = \bigotimes_{t = 1}^{d} u_{t}$ and $v = \bigotimes_{t = 1}^{d} v_{t}$ with $u_{t}, v_{t} \in U_{t}$, $a$ decomposes as
    $$a(u, v) = \sum_{s = 1}^{r} \prod_{t = 1}^{d} a_{s, t}(u_{t}, v_{t}).$$
\end{definition}

In this work, we assume that the linear and bilinear forms involved in the energy functional are separable. This representation is not a restrictive assumption, as any sufficiently regular form can be approximated to arbitrary accuracy by a finite-rank tensor-product decomposition, where the rank corresponds to the number of tensor-product terms \cite{hackbusch2012tensor}.

Now consider degrees $\mb{p} \in \NN^{d}$ and the corresponding free-knot basis function $B_{\mb{p}}$. By the product rule for differentiation, it is clear that the gradient and Hessian of $\ell \circ B_{\mb{p}}$ with respect to the nonlinear parameters are separable, with terms involving $\ell_{s, t} \circ B_{\mb{p}_{t}}$, $\nabla (\ell_{s, t} \circ B_{\mb{p}_{t}})$ and $\nabla^{2} (\ell_{s, t} \circ B_{\mb{p}_{t}})$.

Similarly, for other degrees $\mb{q} \in \NN^{d}$ and the corresponding free-knot basis function $B_{\mb{q}}$, it is clear that the gradient and Hessian of $a \circ (B_{\mb{p}}, B_{\mb{q}})$ with respect to the nonlinear parameters is separable, with terms involving $a_{s, t} \circ (B_{\mb{p}_{t}}, B_{\mb{q}_{t}})$, $\nabla (a_{s, t} \circ (B_{\mb{p}_{t}}, B_{\mb{q}_{t}}))$ and $\nabla^{2} (a_{s, t} \circ (B_{\mb{p}_{t}}, B_{\mb{q}_{t}}))$.

Altogether, this shows that the differentiability of the discrete energy, as well as the boundedness and continuity of the energy and its gradient reduce to analysis of the one-dimensional components $\ell_{s, t} \circ B_{\mb{p}_{t}}$ and $a_{s, t} \circ (B_{\mb{p}_{t}}, B_{\mb{q}_{t}})$.

\subsection{The B-spline functor and its derivatives}

The analysis of a nonlinear discretisation using free-knot B-splines is more intricate than that of standard linear discretisations because the B-spline basis functions are defined in the physical space. Still, the definition of B-splines with divided differences makes it clear that $B_{i, p}(\mb{\tau})$ only depends on the $p+2$ knots $(\mb{\tau}_{i}, \ldots, \mb{\tau}_{i+p+1})$, for all $\mb{\tau} \in \md{K}_{n}$, $p \geq 0$ and $i \in \rg{1}{n - (p + 1)}$. As suggested by the notations in \sect{patches}, we see $B_{i, p}$ as a map from $\md{K}_{n}$ to functions from $\RR$ to $\RR$.

For convenience, given $p \geq 0$, we introduce the \emph{B-spline functor} $B_{p}: \md{K}_{p+2} \to (\RR \to \RR)$ defined by
$$B_{p}(\mb{\tau})(x) \doteq (\mb{\tau}_{p+2} - \mb{\tau}_{1}) [\mb{\tau}_{1}, \ldots, \mb{\tau}_{p+2}] (\bullet - x)_{+}^{p},$$
for all $\mb{\tau} \in \md{K}_{p+2}$, using the notations of \sect{patches}. This way, all B-splines of degree $p$ can be written in terms of $B_{p}$. More precisely,
$$B_{i, p}(\mb{\tau}) = B_{p}(\mb{\tau}_{i}, \ldots, \mb{\tau}_{i+p+1}).$$
This reduces the study of all B-splines of a given degree to a single B-spline functor. The linear, quadratic and cubic B-spline functors are depicted on \fig{splineFunctors}.

For clarity, we use the letters $\mb{\tau}, \mb{\sigma} \in \md{K}_{p+2}$ when studying B-spline functors, and $\mb{\xi}, \mb{\eta} \in \md{X}$ when working with the full nonlinear parameter space.

\begin{figure}
    \centering
    \begin{subfigure}[t]{0.31\linewidth}
        \centering
        \begin{tikzpicture}[scale=1.8]
            \draw[black, thick, -stealth] (-0.2, 0) -- (2.2, 0);
            \draw[black, thick, -stealth] (-0.1, -0.1) -- (-0.1, 1.2);
            \node at (-0.2, 0) [anchor=east] {$0$};
            \node at (-0.2, 1) [anchor=east] {$1$};

            \linearSpline{0.0}{1.6}{2.0}
            \linearSplineVert{0.0}{1.6}{2.0}

            \foreach \x in {0.0, 1.6, 2.0}
            \filldraw[black] (\x, 0) circle (1pt);

            \foreach \x [count=\xk] in {0.0, 1.6, 2.0}
            \node at (\x, -0.35) [anchor=south] {$\mb{\tau}_{\xk}$};
        \end{tikzpicture}
        \caption{Linear, $B_{1}(\mb{\tau})$.}
        \label{fig:linearSplineFunctor}
    \end{subfigure}%
    \hfill%
    \begin{subfigure}[t]{0.31\linewidth}
        \centering
        \begin{tikzpicture}[scale=1.8]
            \draw[black, thick, -stealth] (-0.2, 0) -- (2.2, 0);
            \draw[black, thick, -stealth] (-0.1, -0.1) -- (-0.1, 1.2);
            \node at (-0.2, 0) [anchor=east] {$0$};
            \node at (-0.2, 1) [anchor=east] {$1$};

            \quadraticSpline{0.0}{0.8}{1.5}{2.0}
            \quadraticSplineVert{0.0}{0.8}{1.5}{2.0}

            \foreach \x in {0.0, 0.8, 1.5, 2.0}
            \filldraw[black] (\x, 0) circle (1pt);

            \foreach \x [count=\xk] in {0.0, 0.8, 1.5, 2.0}
            \node at (\x, -0.35) [anchor=south] {$\mb{\tau}_{\xk}$};
        \end{tikzpicture}
        \caption{Quadratic, $B_{2}(\mb{\tau})$.}
        \label{fig:quadraticSplineFunctor}
    \end{subfigure}%
    \hfill%
    \begin{subfigure}[t]{0.31\linewidth}
        \centering
        \begin{tikzpicture}[scale=1.8]
            \draw[black, thick, -stealth] (-0.2, 0) -- (2.2, 0);
            \draw[black, thick, -stealth] (-0.1, -0.1) -- (-0.1, 1.2);
            \node at (-0.2, 0) [anchor=east] {$0$};
            \node at (-0.2, 1) [anchor=east] {$1$};

            \cubicSpline{0.0}{1.0}{1.3}{1.6}{2.0}
            \cubicSplineVert{0.0}{1.0}{1.3}{1.6}{2.0}

            \foreach \x in {0.0, 1.0, 1.3, 1.6, 2.0}
            \filldraw[black] (\x, 0) circle (1pt);

            \foreach \x [count=\xk] in {0.0, 1.0, 1.3, 1.6, 2.0}
            \node at (\x, -0.35) [anchor=south] {$\mb{\tau}_{\xk}$};
        \end{tikzpicture}
        \caption{Cubic, $B_{3}(\mb{\tau})$.}
        \label{fig:cubicSplineFunctor}
    \end{subfigure}%
    \caption{Some B-spline functors.}
    \label{fig:splineFunctors}
\end{figure}
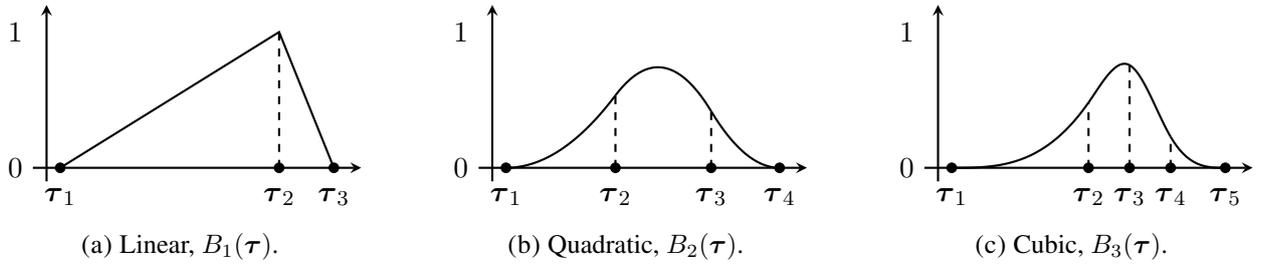

We introduce some notations to work with knot vectors. Given $\mb{\tau} \in \md{K}_{p}$, we write $\mb{\tau} \remBeg$ and $\mb{\tau} \remEnd$ the elements in $\md{K}_{p-1}$ obtained by \emph{removing the first or last element} from $\mb{\tau}$, respectively; that is,
$$\mb{\tau} \remBeg \doteq (\mb{\tau}_{2}, \ldots, \mb{\tau}_{p}), \qquad \mb{\tau} \remEnd \doteq (\mb{\tau}_{1}, \ldots, \mb{\tau}_{p-1}).$$
Given $x \in \cc{\mb{\tau}_{1}}{\mb{\tau}_{p}}$, we also write $\mb{\tau} \oplus x$ the element in $\md{K}_{p+1}$ obtained by \emph{inserting} $x$ into $\mb{\tau}$ while \emph{keeping the order},
$$\mb{\tau} \oplus x \doteq (\mb{\tau}_{1}, \ldots, \mb{\tau}_{k}, x, \mb{\tau}_{k+1}, \ldots, \mb{\tau}_{p}),$$
where $\mb{\tau}_{k} \leq x \leq \mb{\tau}_{k+1}$. We also define the \emph{width} of $\mb{\tau}$ as $\abs{\mb{\tau}} \doteq \mb{\tau}_{p} - \mb{\tau}_{1}$.

Observe that the B-spline functor can be written $B_{p}(\mb{\tau}) = \abs{\mb{\tau}} N_{p}(\mb{\tau})$, where $N_{p}: \md{K}_{p+2} \to (\RR \to \RR)$ is the \emph{normalised B-spline functor} defined by
$$N_{p}(\mb{\tau})(x) \doteq [\mb{\tau}_{1}, \ldots, \mb{\tau}_{p+2}] (\bullet - x)_{+}^{p}.$$
This alternative normalisation is useful to work with B-spline derivatives (spatial or parametric).

With these notations, we can state the differentiability of the B-spline functor and give expressions for its spatial and parametric derivatives.

\begin{lemma}
    \label{lem:BsplineDerivatives}
    Let $p \geq 0$ and $\mb{\tau} \in \md{K}_{p+2}$.

    \begin{itemize}
        \item $B_{p}(\mb{\tau})$ is differentiable in $H^{p-1}(\RR)$, with
              $$\partial_{x} B_{0}(\mb{\tau}) = \delta_{\mb{\tau}_{1}} - \delta_{\mb{\tau}_{2}},$$
              where $\delta_{z} \in H^{-1}(\RR)$ is the Dirac distribution at $z \in \RR$, and for all $p \geq 1$,
              $$\partial_{x} B_{p}(\mb{\tau}) = -p \left(N_{p-1}(\mb{\tau} \remBeg) - N_{p-1}(\mb{\tau} \remEnd)\right).$$
        \item For all $i \in \rg{1}{p+2}$, the map $B_{p}: \md{K}_{p+2} \to H^{p-1}(\RR)$ is differentiable with respect to its $i$-th input in the topology induced by the $H^{p-1}(\RR)$ norm, and
              $$\partial_{i} B_{p}(\mb{\tau}) = N_{p}([\mb{\tau} \oplus \mb{\tau}_{i}] \remBeg) \delta_{i \neq 1} - N_{p}([\mb{\tau} \oplus \mb{\tau}_{i}] \remEnd) \delta_{i \neq p + 2}.$$
    \end{itemize}
\end{lemma}

\begin{proof}
    See \app{BsplineDerivatives}.
\end{proof}

It is important to distinguish between the two notions of differentiability involved in \lem{BsplineDerivatives}. Spatial differentiability refers to the regularity of the B-spline $B_p(\mb{\tau})$ as a function from $\RR$ to $\RR$; specifically, $\partial_x B_p(\mb{\tau})$ belongs to $H^{p-1}(\RR)$ for all $\mb{\tau} \in \md{K}_{p+2}$. In contrast, parametric differentiability concerns the smoothness of $B_p$ as a map from $\md{K}_{p+2}$ to $H^{p-1}(\RR)$, which encodes how the shape of the spline depends on its knot vector.

From now on, we view $\partial_{x} B_{p}$, $\partial_{i} B_{p}$ and compositions thereof as maps from $\md{K}_{p+2}$ to appropriate Hilbert space; that is,
$$\partial_{i_{1}} \cdots \partial_{i_{\alpha}} \partial_{x}^{\beta} B_{p}: \left|\begin{array}{rcl}
        \md{K}_{p+2} & \longrightarrow & H^{p-\alpha-\beta}(\RR)                                                                             \\
        \mb{\tau}    & \longmapsto     & (x \mapsto \partial_{i_{1}} \cdots \partial_{i_{\alpha}} \partial_{x}^{\beta} B_{p}(\mb{\tau})(x)),
    \end{array}\right.$$
for all $\alpha \geq 0$, $\beta \geq 0$ and $i_{1}, \ldots, i_{\alpha} \in \rg{1}{p+2}$.

\subsection{Differentiability of the discrete energy}

We turn to the study of the differentiability of one-dimensional linear and bilinear forms composed with free-knot B-splines. We highlight the interplay between the regularity of the data, the degree of the B-spline and the differentiability of the discrete energy functional.

A key insight is that the differentiability of the energy does not require the B-spline functor itself to be differentiable in the ambient space $U$. In particular, spatial integration and parametric differentiation generally do not commute, and this subtlety must be treated carefully.

For instance, with linear B-splines, the shape function derivatives with respect to knot positions are piecewise constant: they belong to $L^{2}$ but not to $H^{1}$. As a result, parametric differentiation cannot be exchanged with spatial integration, but the energy is still smooth under the existence of a minimum mesh size. A detailed derivation of this example can be found in \cite{thesis}.

Since the convergence of our \algopt{} requires uniformly continuous gradients, we focus on the continuously differentiable case, rather than the more general setting of merely pointwise differentiable compositions.

\begin{lemma}
    \label{lem:derivativeLinearForm}
    Let $I \subset \RR$, $p \geq 0$, $f \in H^{-(p-1)}(I)$. Then the function $\mb{\tau} \mapsto \int_{I} f B_{p}(\mb{\tau}) \dd{x}$ is differentiable and for all $i \in \rg{1}{p+2}$ and all $\mb{\tau} \in \md{K}_{p+2}$, it holds
    $$\partial_{i} \int_{I} f B_{p}(\mb{\tau}) \dd{x} = \int_{I} f \partial_{i} B_{p}(\mb{\tau}) \dd{x}.$$
\end{lemma}

\begin{proof}
    See \app{derivativeLinearForm}.
\end{proof}

\begin{lemma}
    \label{lem:derivativeBilinearForm}
    Let $I \subset \RR$, $k \geq 0$, $p \geq k$, $q \geq k$, $r = \min(p-k-1, q-k)$, $f \in H^{-r}(I)$. If $(p, q) \neq (k, k)$, then for all $\mb{\sigma} \in \md{K}_{q+2}$, the function $\mb{\tau} \mapsto \int_{I} f (\partial_{x}^{k} B_{p}(\mb{\tau})) (\partial_{x}^{k} B_{q}(\mb{\sigma})) \dd{x}$ is differentiable and for all $i \in \{1, \ldots, p+2\}$, $\mb{\tau} \in \md{K}_{p+2}$, it holds
    $$\partial_{i} \int_{I} f (\partial_{x}^{k} B_{p}(\mb{\tau})) (\partial_{x}^{k} B_{q}(\mb{\sigma})) \dd{x} = \int_{I} f (\partial_{i} \partial_{x}^{k} B_{p}(\mb{\tau})) (\partial_{x}^{k} B_{q}(\mb{\sigma})) \dd{x}.$$
\end{lemma}

\begin{proof}
    See \app{derivativeBilinearForm}.
\end{proof}

In other words, \lem{derivativeLinearForm} and \lem{derivativeBilinearForm} give minimal conditions under which spatial integration and parametric differentiation commute. A notable corner case in \lem{derivativeBilinearForm} is the case $(p, q) = (k, k)$. This exclusion is deliberate: differentiability can fail when the knot vectors are allowed to cross. Intuitively, this failure arises because the $k$-th derivatives of both basis functions are piecewise constant and discontinuous only in this case. Knot crossings then cause abrupt changes in the integration domains. In all other cases, at least one derivative is smoother, which smooths out these changes and preserves differentiability. For example, suppose $I = \cc{-1}{+1}$, $\mb{\tau} = (x, 0, 1)$ and $\mb{\sigma} = (y, 0, 1)$ with $-1 \leq x, y < 0$. Then we compute
$$a_{1}(B_{1}(\mb{\sigma}), B_{1}(\mb{\tau})) = 1 - \frac{\max(x, y)}{x y}.$$
This coefficient is continuous in $\co{-1}{0}^{2}$ but not differentiable on the hyperplane $x = y$ due to the maximum function. This type of non-differentiability does not occur when using a single B-spline patch, since the knot vectors remain ordered and do not cross. To keep the presentation easier to follow, we defer the analysis of lowest-degree B-splines to \cite{thesis}.

\lem{derivativeLinearForm} and \lem{derivativeBilinearForm} specify regularity conditions on the data under which linear and bilinear forms are differentiable when evaluated at B-splines of given degrees. In practical applications, however, the data regularity is fixed, and one must instead determine the minimal B-spline degrees that ensure differentiability of the resulting discrete functionals. For the function approximation problem and the Poisson equation, the practical implications are as follows.

From \lem{derivativeLinearForm}:
\begin{itemize}
    \item if $f \in L^{2}(I)$, then $\mb{\tau} \mapsto \int_{I} f B_{p}(\mb{\tau}) \dd{x}$ is differentiable when $p \geq 1$;
    \item if $f \in H^{-1}(I)$, then $\mb{\tau} \mapsto \int_{I} f B_{p}(\mb{\tau}) \dd{x}$ is differentiable when $p \geq 2$.
\end{itemize}
From \lem{derivativeBilinearForm}:
\begin{itemize}
    \item if $k = 0$ and $f \in L^{\infty}(I)$, then $(\mb{\sigma}, \mb{\tau}) \mapsto \int_{I} f B_{p}(\mb{\sigma}) B_{q}(\mb{\tau}) \dd{x}$ is differentiable when $p, q \geq 1$;
    \item if $k = 1$ and $f \in L^{\infty}(I)$, then $(\mb{\sigma}, \mb{\tau}) \mapsto \int_{I} f (\partial_{x} B_{p}(\mb{\sigma})) (\partial_{x} B_{q}(\mb{\tau})) \dd{x}$ is differentiable when $p, q \geq 2$.
\end{itemize}

Here again, the case of a single patch of lowest degree is described in \cite{thesis}. The generalisation of \lem{derivativeLinearForm} and \lem{derivativeBilinearForm} to the $d$-dimensional case follows from the tensor-product structure of the basis functions, of the physical domain, and the separability of the linear and bilinear forms.

\subsection{Uniform boundedness and regularity}
\label{subsect:boundedness-regularity}

We now turn to showing that B-splines satisfy the uniform boundedness and regularity conditions required by \assumlinboundreg{}, under a minimum mesh size condition. For any degree $p \geq 0$ and knot vector $\mb{\tau} \in \md{K}_{p+2}$, we define the \emph{minimum mesh size}
$$h(\mb{\tau}) \doteq \min_{i \in \rg{1}{p+1}} \mb{\tau}_{i+1} - \mb{\tau}_{i}.$$

\begin{lemma}
    \label{lem:boundedness}
    There exist constants $0 < C_{p} \leq (p + 1)^{-1/2}$ such that
    \begin{itemize}
        \item For all $p \geq 0$ and $\mb{\mb{\tau}} \in \md{K}_{p+2}$,
              $$\norm{B_{p}(\mb{\tau})}_{L^{2}(\RR)} \leq C_{p} \abs{\mb{\tau}}^{1/2}.$$
        \item For all $p \geq 1$ and $\mb{\mb{\tau}} \in \md{K}_{p+2}$,
              $$\norm{\partial_{x} B_{p}(\mb{\tau})}_{L^{2}(\RR)} \leq \sqrt{2 p} C_{p-1} h(\mb{\tau})^{-1/2}.$$
        \item For all $p \geq 1$, $i \in \rg{1}{p+2}$ and $\mb{\tau} \in \md{K}_{p+2}$,
              $$\norm{\partial_{i} B_{p}(\mb{\tau})}_{L^{2}(\RR)} \leq \sqrt{\frac{2}{p}} C_{p} h(\mb{\tau})^{-1/2}.$$
        \item For all $p \geq 2$, $i \in \rg{1}{p+2}$ and $\mb{\tau} \in \md{K}_{p+2}$,
              $$\norm{\partial_{i} \partial_{x} B_{p}(\mb{\tau})}_{L^{2}(\RR)} \leq \frac{2 \sqrt{2}}{p - 1} C_{p - 1} h(\mb{\tau})^{-3/2}.$$
    \end{itemize}
\end{lemma}

\begin{proof}
    See \app{boundedness}.
\end{proof}

\lem{boundedness} establishes the boundedness of B-splines and their spatial and parametric derivatives under the minimal mesh size assumption. In particular, it implies the boundedness of one-dimensional discrete linear and bilinear forms whenever $h(\mb{\tau}) > 0$. To ensure uniform boundedness of the full $d$-dimensional discrete forms with respect to $\mb{\xi} \in \md{X}$, it suffices to impose a uniform lower bound $h_{\min} > 0$ such that $h(\mb{\xi}_{s, t}) \geq h_{\min}$ for all patches $s \in \rg{1}{r}$ and dimensions $t \in \rg{1}{d}$.

We next examine the continuity properties of one-dimensional B-splines and their derivatives with respect to variations in the knot vector. Specifically, we establish Hölder continuity under a minimum mesh size condition. Given a collection of knot vectors $\mb{\tau}_{k} \in \md{K}_{p_{k} + 2}$, $p_{k} \geq 0$, $k \in \rg{1}{n}$, we define the joint minimum mesh size
$$h(\mb{\tau}_{1}, \ldots, \mb{\tau}_{n}) \doteq \min_{k \in \rg{1}{n}} h(\mb{\tau}_{k}).$$

\begin{lemma}
    \label{lem:holder}
    The constants $C_{p}$ are the same as in \lem{boundedness}.
    \begin{itemize}
        \item For all $\mb{\sigma}, \mb{\tau} \in \md{K}_{2}$,
              $$\norm{B_{0}(\mb{\sigma}) - B_{0}(\mb{\tau})}_{L^{2}(\RR)} \leq 2 C_{0} \norm{\mb{\sigma} - \mb{\tau}}_{2}^{1/2}$$
        \item For all $p \geq 1$ and $\mb{\sigma}, \mb{\tau} \in \md{K}_{p+2}$,
              $$\norm{B_{p}(\mb{\sigma}) - B_{p}(\mb{\tau})}_{L^{2}(\RR)} \leq \sqrt{\frac{2 (p + 1)}{p}} C_{p} h^{-1/2}(\mb{\sigma}, \mb{\tau}) \norm{\mb{\sigma} - \mb{\tau}}_{2}$$
        \item For all $\mb{\sigma}, \mb{\tau} \in \md{K}_{3}$,
              $$\norm{\partial_{x} B_{1}(\mb{\sigma}) - \partial_{x} B_{1}(\mb{\tau})}_{L^{2}(\RR)} \leq 4 C_{0} h^{-1}(\mb{\sigma}, \mb{\tau}) \norm{\mb{\sigma} - \mb{\tau}}_{2}^{1/2}$$
        \item For all $p \geq 1$ and $\mb{\sigma}, \mb{\tau} \in \md{K}_{p+2}$,
              $$\norm{\partial_{x} B_{p}(\mb{\sigma}) - \partial_{x} B_{p}(\mb{\tau})}_{L^{2}(\RR)} \leq 2 \sqrt{\frac{p + 1}{p - 1}} C_{p-1} h^{-3/2}(\mb{\sigma}, \mb{\tau}) \norm{\mb{\sigma} - \mb{\tau}}_{2}$$
        \item For all $i \in \rg{1}{3}$ and $\mb{\sigma}, \mb{\tau} \in \md{K}_{3}$,
              $$\norm{\partial_{i} B_{1}(\mb{\sigma}) - \partial_{i} B_{1}(\mb{\tau})}_{L^{2}(\RR)} \leq 5 C_{1} h^{-1}(\mb{\sigma}, \mb{\tau}) \norm{\mb{\sigma} - \mb{\tau}}_{2}^{1/2}$$
        \item For all $p \geq 1$, $i \in \rg{1}{p+2}$ and $\mb{\sigma}, \mb{\tau} \in \md{K}_{p+2}$,
              $$\norm{\partial_{i} B_{p}(\mb{\sigma}) - \partial_{i} B_{p}(\mb{\tau})}_{L^{2}(\RR)} \leq \frac{2}{p} \sqrt{\frac{2 p + 7}{p - 1}} C_{p} h^{-3/2}(\mb{\sigma}, \mb{\tau}) \norm{\mb{\sigma} - \mb{\tau}}_{2}$$
        \item For all $i \in \rg{1}{4}$ and $\mb{\sigma}, \mb{\tau} \in \md{K}_{4}$,
              $$\norm{\partial_{i} \partial_{x} B_{2}(\mb{\sigma}) - \partial_{i} \partial_{x} B_{2}(\mb{\tau})}_{L^{2}(\RR)} \leq 12 C_{1} h^{-2}(\mb{\sigma}, \mb{\tau}) \norm{\mb{\sigma} - \mb{\tau}}_{2}^{1/2}$$
        \item For all $p \geq 2$, $i \in \rg{1}{p+2}$ and $\mb{\sigma}, \mb{\tau} \in \md{K}_{p+2}$,
              $$\norm{\partial_{i} \partial_{x} B_{p}(\mb{\sigma}) - \partial_{i} \partial_{x} B_{p}(\mb{\tau})}_{L^{2}(\RR)} \leq \frac{4}{p - 1} \sqrt{\frac{p + 7}{p - 2}} C_{p-1} h^{-5/2}(\mb{\sigma}, \mb{\tau}) \norm{\mb{\sigma} - \mb{\tau}}_{2}$$
    \end{itemize}
\end{lemma}

\begin{proof}
    See \app{holder}
\end{proof}

In other words, \lem{holder} states that B-splines of sufficiently high degree, along with their spatial and parametric derivatives, are Lipschitz continuous in the $L^{2}$-norm, while B-splines of lower degree only enjoy Hölder continuity.

\begin{remark}
    In \lem{boundedness} and \lem{holder}, we carefully track boundedness and continuity constants rather than relying on implicit or generic bounds. While rough estimates often suffice for exposition, we provide explicit constants to support our convergence analysis and enable principled step size selection (see \lemlocal{} and \lemglobal{} in our companion work).
\end{remark}

We now connect \lem{boundedness} and \lem{holder} to the framework of our companion work. For $k \in \{0, 1\}$ and polynomial degrees $p \geq k + 1$, we have shown that the B-spline functor $B_{p}$ is differentiable as a map into $H^{k}(\RR)$. Under the minimum mesh size assumption, the parametric derivatives $\partial_{i} B_{p}$ are also uniformly bounded and Lipschitz continuous in $H^{k}(\RR)$. This verifies \assumnonlinboundreg{}, and therefore $\DX \mc{K}$ is automatically Lipschitz continuous by a result of our companion work (\lemregularity{}).

The case $p = k$ requires a special treatment since $B_{p}$ is not differentiable in $H^{p}(\RR)$. However, when all knot vectors are derived from a common knot vector and the data is sufficiently smooth, the discrete energy remains differentiable. In \cite{thesis}, we demonstrate that $\DX \mc{K}$ is Hölder continuous for $k \in \{0, 1\}$.

\subsection{Uniform coercivity}
\label{subsect:coercivity}

The global convergence of our \algopt{} hinges on the uniform coercivity of the energy; that is, on a uniform lower bound on the discrete bilinear form. For all $\mb{\xi} \in \md{X}$, let $\mb{A}(\mb{\xi}) \in \RR^{\dim \md{W} \times \dim \md{W}}$ denote the matrix with entries
$$\mb{A}(\mb{\xi})_{ij} = a(\mb{\phi}_{i}(\mb{\xi}), \mb{\phi}_{j}(\mb{\xi})).$$
We recall standard eigenvalue bounds for the mass and stiffness matrices on a single patch; see \cite{ern2006evaluation} for a comprehensive treatment of the condition number in linear systems arising in finite element approximations.
\begin{itemize}
    \item For the mass matrix of the function approximation problem, a standard scaling argument shows
          $$\lambda_{\min}(\mb{A}(\mb{\xi})) \gtrsim h_{\min}^{d}(\mb{\xi}), \qquad \lambda_{\max}(\mb{A}(\mb{\xi})) \lesssim h_{\max}^{d}(\mb{\xi}),$$
          where the constants involved in these bounds are independent of $\mb{\xi}$. This means the conditioning of $\mb{A}(\mb{\xi})$ scales as $O(h_{\max}^{d}(\mb{\xi}) / h_{\min}^{d}(\mb{\xi}))$.
    \item For the stiffness matrix of the Poisson problem, the Poincaré inequality and a discrete inverse inequality combined with the eigenvalue bounds for the mass matrix above show
          $$\lambda_{\min}(\mb{A}(\mb{\xi})) \gtrsim h_{\min}^{d}(\mb{p}), \qquad \lambda_{\max}(\mb{A}(\mb{\xi})) \lesssim h_{\max}^{d}(\mb{p}) h_{\min}^{-2}(\mb{\xi}),$$
          where the constants involved in these bounds are independent of $\mb{\xi}$. In particular, the conditioning of $\mb{A}(\mb{\xi})$ scales as $O(h_{\max}^{d}(\mb{\xi}) / h_{\min}^{d + 2}(\mb{\xi}))$.
\end{itemize}

\subsection{Constraints on the nonlinear parameters}
\label{subsect:constraints}

We constrain the nonlinear parameter space $\md{X}$ to ensure the conformity $V \subset U$ and the structural assumptions of our framework. These constraints include the required trace properties of the function space, a minimum mesh size, and prevent degenerate cases where basis functions lie entirely outside the physical domain, hence contributing nothing to the numerical ansatz.

Note that for all $n \geq 0$, the set $\md{K}_{n}$ is convex, and that the set $\md{X} = \md{K}_{\mb{N}}$ as defined in \subsect{stp-patches} (without additional constraints that we describe below) is convex as a product of convex sets.

Let $I = \cc{x_{-}}{x_{+}}$ be a bounded interval and consider a one-dimensional patch $\md{S}_{n, p}(\mb{\tau})$ for some $\mb{\tau} \in \md{K}_{n+2}$.

\subsubsection{Zero trace}

To enforce the zero-trace condition, it is enough to require $\mb{\tau}_{1} \geq x_{-}$ and $\mb{\tau}_{n+2} \leq x_{+}$. These constraints are linear inequalities, so enforcing them preserves the convexity of the parameter space.

\subsubsection{Preventing zero contributions}

Since each B-spline $B_{i, p}(\mb{\tau})$ is compactly supported, this restriction effectively discards any basis function whose support lies entirely outside $I$. To avoid this, one must ensure that the support of every B-spline has a non-empty intersection with $I$. This is guaranteed if $\mb{\tau}_{p+2} > x_{-}$ and $\mb{\tau}_{n-(p+1)} < x_{+}$. Under these conditions, the restricted space $\md{S}_{n, p}(\mb{\tau})|_{I}$ retains the full set of shape functions and thus its approximation power over $I$. We convert these strict inequalities into large ones and instead enforce $\mb{\tau}_{p+1} \geq x_{-}$ and $\mb{\tau}_{n-p} \leq x_{+}$. The zero-trace condition automatically satisfies this constraint. Here again these constraints are linear inequalities so they preserve the convexity of the parameter space.

\subsubsection{Minimum mesh size}

Two types of minimum mesh size conditions are needed: one within each one-dimensional patch to ensure boundedness and regularity of the parameterisation, and one across patches to guarantee uniform coercivity.

Within a single patch $\md{S}_{n, p}(\mb{\tau})$, the minimum mesh size condition is straightforward: ensure that $\mb{\tau}_{i+1} \geq \mb{\tau}_{i} + h_{\min}$ for all $i \in \rg{1}{n+1}$. These constraints preserve the convexity of the parameter space.

Across patches, minimum mesh size constraints can be formulated in various ways. A practical choice consists in treating each axis separately and imposing a minimum mesh size over the union of all patches along that direction. In that case, the union of patches can be regarded as a single larger patch with a positive minimum mesh size. The spectral estimates of the previous subsection then apply to the global stiffness matrix, ensuring a uniform lower bound on its eigenvalues. This guarantees \assumspd{}, which in turn implies the Hölder continuity of the reduced energy gradients (via \lemregularityreduced{}).

Mathematically, this constraint reads as follows: for any two one-dimensional patches $\md{S}_{n, p}(\mb{\sigma})$ and $\md{S}_{m, q}(\mb{\tau})$ along the same dimension, with $\mb{\sigma} \in \md{K}_{n+2}$ and $\mb{\tau} \in \md{K}_{m+2}$, we impose $\abs{\mb{\sigma}_{i} - \mb{\tau}_{j}} \geq h_{\min}$ for all $i \in \rg{1}{n+2}$ and $j \in \rg{1}{m+2}$.

These constraints are non-convex, but they can be rewritten as disjunctions of linear inequalities: $\mb{\sigma}_{i} - \mb{\tau}_{j} \geq h_{\min}$ or $\mb{\tau}_{j} - \mb{\sigma}_{i} \geq h_{\min}$. As a result, the parameter space splits into a finite union of disjoint convex sets, and \assumspace{} cannot be verified. We describe the complications associated with this non-convexity and a strategy to go around it.

A first complication is that the projection defining the mirror descent step may be set-valued, as there can be multiple global minimisers over a non-convex set. A choice must be taken algorithmically, typically guided by the implementation or initialisation of the projection algorithm.

More significantly, \lemlocal{} no longer applies so the energy may not decrease at every step. Indeed, our proof relies on the inequality
$$\inner{-\gamma \DX \mc{K}(\mb{w}, \mb{\xi}) - \nabla \psi(\mb{\xi}_{+}) + \nabla \psi(\mb{\xi})}{\mb{\xi} - \mb{\xi}_{+}}_{\md{X}} \leq 0,$$
which only holds when the line connecting $\mb{\xi}$ to $\mb{\xi}_{+}$ contains an open neighbourhood of $\mb{\xi}_{+}$. In other words, the energy may increase if $\mb{\xi}_{+}$ lies on the boundary of another convex component than that of $\mb{\xi}$.

One way to address this is to rely on the convergence guarantees of our abstract framework along segments of the trajectory that stay within a single convex component. Since our \algopt{} retains the best solution found, and gradients vanish near an interior minimum, the iterates eventually remain in the corresponding component. At that point, the local and global convergence results apply, and the total number of iterations corresponds to those spent within that component.

\subsubsection{Convexity and projection}
Taking all constraints above into account, the nonlinear parameter space is convex when using a single patch or in the multi-patch setting after the parameter iterates remain within a single convex component. In that case, the projection involved in the mirror descent step, with respect to the Euclidean distance, reduces to a linearly constrained quadratic program. This can be solved using standard quadratic programming techniques \cite[Chapter 16]{nocedal1999numerical}.

In the multi-patch case, when the parameters can jump across convex components, the projection involves minimising a quadratic function subject to quadratic inequalities. These arise by rewriting $\abs{\mb{\sigma}_{i} - \mb{\tau}_{j}} \geq h_{\min}$ as $(\mb{\sigma}_{i} - \mb{\tau}_{j})^{2} \geq h_{\min}^{2}$. Computing the projection classifies as a convex, quadratically constrained quadratic program. It can be efficiently solved using an interior point method \cite[Chapter 6]{nesterov1994interior}.

\subsubsection{Compactness}
The nonlinear parameter space must be compact to guarantee the existence of global minimisers of the discrete energy, and to ensure that these minimisers lie within the parameter space (\assumspace{}). Since all constraints above are stated with large inequalities, the constrained parameter space is closed. To guarantee compactness, it must also be bounded. If zero-trace conformity is enforced, then boundedness follows automatically. Otherwise, we bound the parameter space by requiring $\mb{\tau}_{1} \geq y_{-}$ and $\mb{\tau}_{n+2} \leq y_{+}$ for some $y_{-} \leq x_{-}$ and $y_{+} \geq x_{+}$. In our implementation, we choose $y_{-} = x_{-} - (x_{+} - x_{-})$ and $y_{+} = x_{+} + (x_{+} - x_{-})$.

\subsection{Initialisation of the nonlinear parameters}
\label{subsect:initialisation}

The initial patch configuration is designed to emulate the space of uniform B-splines, so that the approximation space contains the solution of a standard discretisation on a uniform mesh. By initialising the minimisation algorithm with the solution of this uniform space and retaining the solution with the lowest energy, we obtain a guaranteed lower bound on the error and ensure a minimal convergence rate.

\subsubsection{Function approximation problem}

In the function approximation setting with a single patch, the knots are evenly spaced so that the first and last $p$ knots lie outside the domain, and the knots number $p+1$ and $n-p$ coincide with the domain boundaries.

In the multi-patch case, adjacent patches of degree $p$ are made to overlap over $p+1$ knots to emulate a uniform discretisation.

\subsubsection{Poisson problem}

The key difference in the Poisson setting is that all knots must lie within the domain. To emulate uniform B-splines, repeated knots are used instead of knots outside the domain. More precisely, using B-splines of degree $p \geq 1$, the first and last $p-1$ knots are repeated. These repeated knots are fixed (non-trainable) to preserve the differentiability of the energy functional.

In the multi-patch case, only the patches touching the domain boundary at initialisation have repeated knots; the other ones simply overlap with their neighbours at initialisation and all their knots are free.
\subsection{Directional convexity}

Directional convexity of the discrete energy is generally difficult to prove. We verify it explicitly in two simple cases in \cite{thesis}.

\subsection{Summary of the analysis}

The patchwise minimum mesh size condition ensures the boundedness and regularity of the parameterisation (\assumlinboundreg{}) and its parametric derivatives. Moreover, B-splines of degree one more than the highest order involved in the linear and bilinear forms are parametrically differentiable in $U$, satisfying \assumnonlinboundreg{}. This ensures the energy gradient is Hölder in the parameters. For B-splines of the same degree as the variational forms, we have shown the regularity of the energy gradient independently, under additional $H^{1}$ regularity of the data and no knot crossing.

Since the nonlinear parameter space is compact and convex (\assumspace{}), at least after an exploration phase, all the hypotheses guaranteeing local convergence of our \algopt{} hold (\lemlocal{}).

\begin{corollary}[See \lemlocal{}]
    Under a global minimum mesh size condition, and B-spline degree exceeding the \ac{pde} order (or equal if single patch), our \algopt{} converges to a local minimum of the discrete energy.
\end{corollary}

Regarding global convergence, uniform coercivity is ensured under the assumption of a global minimum mesh size. The directional convexity condition (\assumdircon{}) is generally challenging to verify analytically, but we will verify it numerically on some examples in our numerical experiments.

\begin{corollary}[See \lemglobal{} and \corolcea{}]
    Under a global minimum mesh size condition, B-spline degree exceeding the \ac{pde} order (or equal if single patch), and assuming directional convexity, our \algopt{} converges to a global minimum of the discrete energy. With the same notations as in our companion work, these assumptions yield a nonlinear analogue of Céa's lemma: the $n$-th iterate $\mb{\xi}_{n}$ of our \algopt{} satisfies
    $$\norm{\red{\mc{R}}(\mb{\xi}_{n}) - u^{\star}}_{a}^{2} \leq \inf_{v \in V} \norm{v - u^{\star}}_{a}^{2} + \frac{2 \red{L} \delta_{\psi}^{\star}(\mb{\xi}_{0})}{n}.$$
    In other words, $\red{\mc{R}}(\mb{\xi}_{n})$ approaches an optimal approximation of $u^{\star}$ in $V$ at a rate $O(n^{-1/2})$.
\end{corollary}

\subsection{Approximation properties}

We conclude this section by summarising the approximation properties of free-knot splines and how they motivate our choice of approximation space.

We recall some results from linear and nonlinear approximation theory in dimension $1$. For $k, n \geq 1$, let $\tilde{S}_{k, n}(\Omega)$ denote the space of piecewise polynomials of degree $k - 1$ with $n - 1$ free knots in $\Omega$ ($n + 1$ knots in total). Let then $S_{k, n}(\Omega) = \tilde{S}_{k, n}(\Omega) \cap C^{k-2}(\Omega)$ denote the space of splines of degree $k - 1$ with $n + 1$ knots.

It is well-known that for all $u \in H^{2}(\Omega)$, the piecewise linear interpolant of $u$ on a uniform partition of $\Omega$ into $n$ pieces converges to $u$ at a rate $n^{-2}$ in the $L^{2}$-norm and $n^{-1}$ in the $H^{1}$-seminorm, so
$$\inf_{v \in S_{2, n}(\Omega)} \norm{u - v}_{L^{2}(\Omega)} \lesssim n^{-2} \norm{u''}_{L^{2}(\Omega)}, \qquad \inf_{v \in S_{2, n}(\Omega)} \abs{u - v}_{H^{1}(\Omega)} \lesssim n^{-1} \norm{u''}_{L^{2}(\Omega)},$$
where the constants involved depend only on $\Omega$.

These estimates require $u \in H^{2}(\Omega)$. In contrast, adaptive knot placement allows the same convergence rates for far rougher functions. Specifically, for $u$ in the Besov space $B^{2,2/5}(\Omega)$, which is much larger than $H^2$ and barely embedded in $L^2$, the same $L^2$ convergence rate holds. Similar results apply in the $H^1$-seminorm using $B^{1,2/3}$ regularity. These results, originating in nonlinear approximation theory \cite{petrushev1988direct}, illustrate the benefit of mesh adaptivity.

In higher dimensions, the theoretical understanding of free-knot splines is more limited. However, our construction includes as a subset the space of tensor-product B-splines with uniform knots, whose approximation power is well studied in the \ac{iga} community \cite[Section 3]{bazilevs2006isogeometric}, \cite[Appendix 3.B]{cottrell2009isogeometric}. In particular, if $u \in H^s(\Omega)$ with $s \leq p$, the standard estimate
$$\inf_{v \in V} \abs{u - v}_{H^{r}(\Omega)} \lesssim h^{s - r} \abs{u}_{H^{s}(\Omega)}$$
holds for all $0 \leq r \leq s \leq p$, where $h$ is the maximum element size.

By placing overlapping patches appropriately, our construction can replicate classical hierarchical B-spline spaces \cite{giannelli2012thb} (both truncated and untruncated) as well as sparse grid layouts \cite{bungartz2004sparse}. It enables not only patch-wise adaptivity, but also cross-patch adaptivity, where overlapping patches can work together to resolve local features, effectively distributing resolution across patch boundaries. This combines the strengths of tensor-product and hierarchical methods with greater flexibility in both geometry and refinement.

\section{Numerical experiments}
\label{sect:experiments}
We begin this section with a series of numerical experiments that illustrate and validate the key features of our variational framework. Starting from simple exploratory cases, we progressively investigate more complex one- and two-dimensional problems in function approximation and Poisson equations.

\subsection{Note on implementation}

All the bilinear forms considered in this work only involve piecewise polynomials and are therefore computed exactly using Gaussian quadratures of appropriate order. The linear forms, involving the data, are approximated with adaptive quadratures using the Simpson rule, with a tolerance of $10^{-12}$. We impose a patchwise minimum mesh size of $10^{-6}$ for all experiments.

The linear and bilinear forms are assembled into one-dimensional matrices and vectors at the patch level, significantly reducing storage compared to assembling full multidimensional matrices. This approach exploits sum factorisation \cite{antolin2015efficient}, combining one-dimensional operators via the tensor-product structure to efficiently apply multidimensional operators without forming large global matrices. Gradients with respect to nonlinear parameters are computed through matrix-vector products using analytical derivatives of the basis functions, enabling straightforward and scalable optimisation across varying problem sizes.

Our implementation is written in Julia using a custom codebase. Constraints are imposed via the optimisation modelling package \texttt{JuMP.jl}, with the interior point optimiser \texttt{Ipopt.jl} used as the nonlinear solver.

\subsection{Weak formulations}
We consider two weak formulations in our experiments. The first one is the simple function approximation problem: given $f \in L^{2}(\Omega)$, find $u \in L^{2}(\Omega)$ such that $(a, v) = \ell(v)$ for all $v \in L^{2}(\Omega)$, where
$$a(u, v) \doteq \int_{\Omega} u v \dd{\Omega}, \qquad \ell(v) \doteq \int_{\Omega} f v \dd{\Omega}.$$
This formulation serves primarily to test the approximation properties of free-knot B-splines and to verify the correctness of our optimisation algorithm.

We also consider the Poisson problem with homogeneous boundary conditions: given $f \in H^{-1}(\Omega)$, find $u \in H^{1}_{0}(\Omega)$ such that $a(u, v) = \ell(v)$ for all $v \in H^{1}_{0}(\Omega)$, where
$$a(u, v) \doteq \int_{\Omega} \nabla u \cdot \nabla v \dd{\Omega}, \qquad \ell(v) \doteq \int_{\Omega} f v \dd{\Omega}.$$
This second problem is the main focus of our study, as it is representative of self-adjoint second-order elliptic differential operators.

\subsection{Exploratory case studies}

In \cite{thesis}, we perform a series of exploratory experiments addressing several structural aspects of the discrete energy functional. We first examine the sharpness of the assumptions required for differentiability. We then investigate the possible degeneracy of global minimisers, including scenarios in which the minimising set comprises multiple non-isolated components. Further experiments assess how constraints in parameter space affect admissibility and the resulting solution structure. Finally, we study the smoothness of the energy landscape and the geometry of basins of attraction, with particular attention to how their size and shape evolve as the regularity of the continuous solution varies.

\subsection{Optimisation procedure}

We now give some detail regarding the implementation of the optimisation loop, which is common for the remaining experiments of this work. Our goal is not to eliminate optimisation error entirely, but rather to demonstrate that competitive accuracy can be achieved with minimal computational effort. Therefore, we rely on the algorithm proposed in our companion work, but we make several adjustments to improve computational efficiency.

At each iteration, the linear system definining the best linear parameters is solved using \ac{cg}, with a tolerance of $10^{-12}$. To reduce computational cost during training, the linear system is solved to full precision only every $25$ iterations; in between, the number of \ac{cg} iterations is capped at $100$.

We use the ADAM optimiser \cite{kingma2014adam} followed by projection on the feasible parameter set in place of mirror descent, as it usually converges faster to a stationary point. We choose the following hyperparameters: $\beta_{1} = 0.9$ (first momentum), $\beta_{2} = 0.99$ (second momentum), $\epsilon = 10^{-8}$ (offset to avoid divisions by zero). We use a custom learning rate scheduler of the form $\eta(t) = (1 - \exp(-t/50))\eta$, where $t$ is the iteration count and $\eta$ is the asymptotic learning rate. This warm-up strategy starts with a near-zero learning rate and gradually increases toward $\eta$, helping to stabilise early iterations where gradients can be large. In our case, this prevents abrupt mesh movements at initialisation, which would otherwise result in jagged or unstable mesh trajectories.

Each optimisation is performed for up to $1000$ iterations, with early stopping triggered if the total displacement of the nonlinear parameters between two consecutive iterations is smaller than $10^{-6}$ in 1D and $10^{-4}$ in 2D. Since experiments with more unknowns take longer to converge, we increase the maximum number of iterations to $3000$ when the number of \acp{dof} exceeds $1000$. We select the best-performing constant step size from $\{10^{-1}, 5 \times 10^{-2}, 2 \times 10^{-2}, 10^{-2}, 5 \times 10^{-3}, 2 \times 10^{-3}, 10^{-3}, 5 \times 10^{-4}\}$ in each run. In practice, we observe that the optimal learning rate is typically $10^{-2}$ for moderate numbers of mesh parameters, and dropping to $2 \times 10^{-3}$ as the number of parameters grows.

\subsection{1D experiments}

We compare the approximation properties of various spaces defined by different B-spline degrees (ranging from $1$ to $5$) and varying numbers of free nonlinear parameters ($1$ to $5$, $10$, $20$, $50$, $100$, $200$ and $500$). In each case, we initialise the knots configuration uniformly over the domain, so that after solving the initial linear system, the resulting model coincides with the standard B-spline approximation on a uniform mesh.

\subsubsection{Function approximation}

We first look at a function approximation problem. We choose to approximate a function with a sharp transition around several points in the domain $\Omega = \cc{-1}{+1}$, expecting that our optimisation method will be able to concentrate the mesh around them. Specifically, we approximate
$$f(x) = \sin(2 x + 0.4) \operatorname{sign}_{\epsilon}(\sin(2 x + 0.4)), \qquad \operatorname{sign}_{\epsilon}(x) = \frac{x}{\sqrt{x^{2} + \epsilon}},$$
where $\operatorname{sign}_{\epsilon}: \RR \to \oo{-1}{+1}$ is a smooth approximation of the sign function when $\epsilon \to 0$. We arbitrarily set $\epsilon = 0.01$. This function is shown in \fig{1-approx-pointwise}.

Since the target function lies in $L^{2}(\Omega)$, our analysis shows that the discrete energy functional is continuously differentiable for all B-spline degrees $p \geq 1$. The gradient of the energy is Hölder continuous with respect to the nonlinear parameters when $p = 1$, and becomes Lipschitz continuous for all $p \geq 1$. For $p = 0$, since the target function is in fact in $C^{\infty}(\Omega)$, and has bounded derivatives, it is also Lipschitz and the energy is differentiable and Lipschitz in the mesh parameters.

\fig{1-approx-convergence} presents the error in the energy norm as a function of the number of linear \acp{dof}, comparing uniform and adapted meshes for B-spline degrees $0$ through $5$. As expected, uniform meshes exhibit the standard convergence rate of $O(h^{p+1})$, that is $O(n^{-(p+1)})$ in terms of the number of \acp{dof}. In contrast, free-knot B-splines achieve substantial improvements, up to two orders of magnitude, by concentrating resolution in regions of reduced regularity. The benefits of adaptivity, however, depend on the number of free parameters: when too few are available, the mesh lacks sufficient flexibility to significantly improve the approximation; when too many are used, the uniform mesh is already fine enough to resolve the target features. The most notable gains occur for a moderate number of free parameters, where adaptive refinement can be effectively targeted.

\begin{figure}
    \centering
    \includegraphics[width=0.45\linewidth]{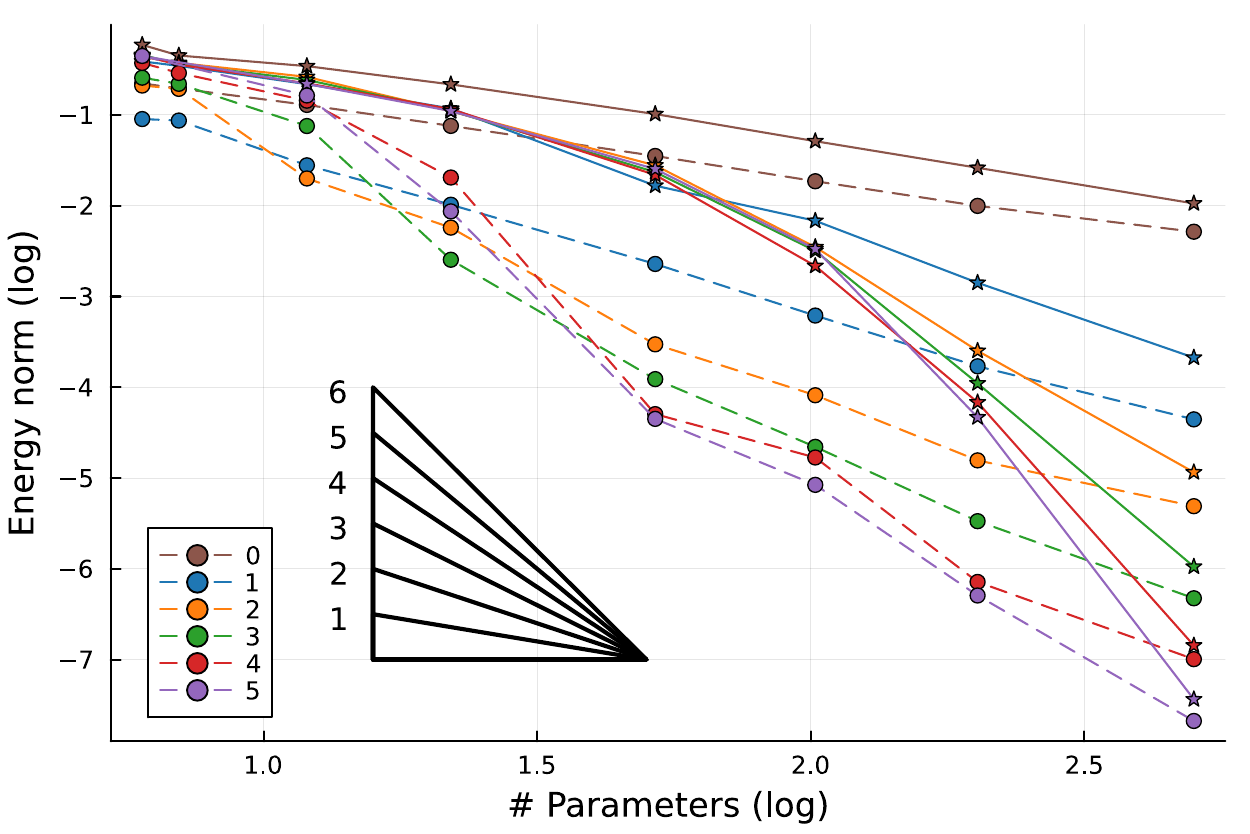}
    \caption{Convergence plot for the approximation problem in 1D with B-splines of degrees $0$ to $5$, in the energy norm. Solid lines represent uniform meshes while dashed lines correspond to adapted meshes.}
    \label{fig:1-approx-convergence}
\end{figure}

We now present a detailed example illustrating the behaviour of the optimisation process. \fig{1-approx-knots} shows the evolution of the knot positions over the course of training, in the case of the linear B-spline with $23$ free shape functions ($24$ free knots). As optimisation progresses, the knots gradually concentrate in regions where the target function exhibits sharp transitions, reflecting the adaptive refinement of the mesh. Interestingly, the outmost knots progressively move outside the domain, and the rightmost knot even reaches the artificial boundary $x = -3$, after which it is projected back to $x = -3$ in the remainder of the training. Additionally, the trajectories of some interior knots suggest potential collisions, which are prevented by our projection mechanism enforcing a minimum mesh size. In several instances, knots appear to approach one another closely, effectively attempting to collide, before being repelled and redistributed automatically. Finally, \fig{1-approx-pointwise} compares the initial solution, based on the uniform mesh, with the final solution obtained from the adapted mesh. The adapted solution clearly captures fine-scale features and localised variations that are missed by the uniform approximation.

\begin{figure}
    \centering
    \subfloat[\label{fig:1-approx-pointwise}]{
        \includegraphics[width=0.45\linewidth]{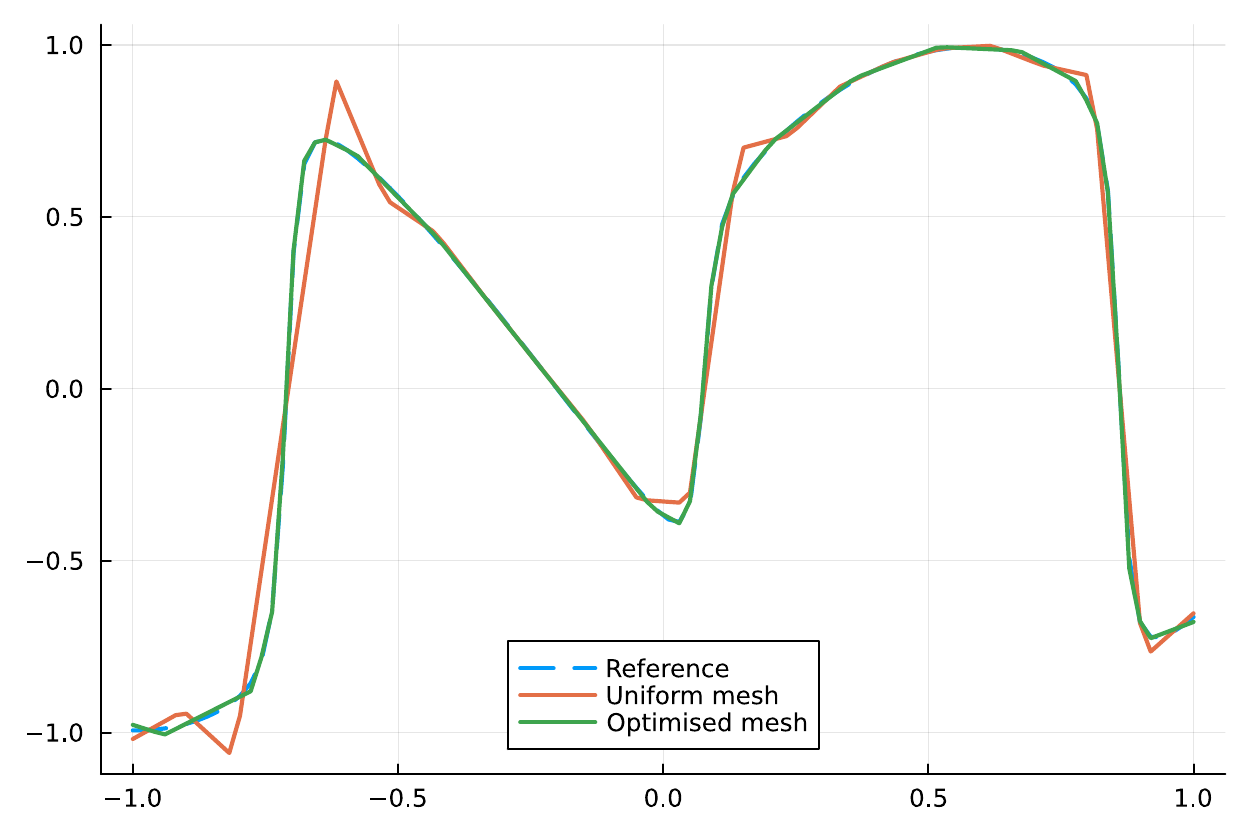}
    }%
    \hfill%
    \subfloat[\label{fig:1-approx-knots}]{
        \includegraphics[width=0.45\linewidth]{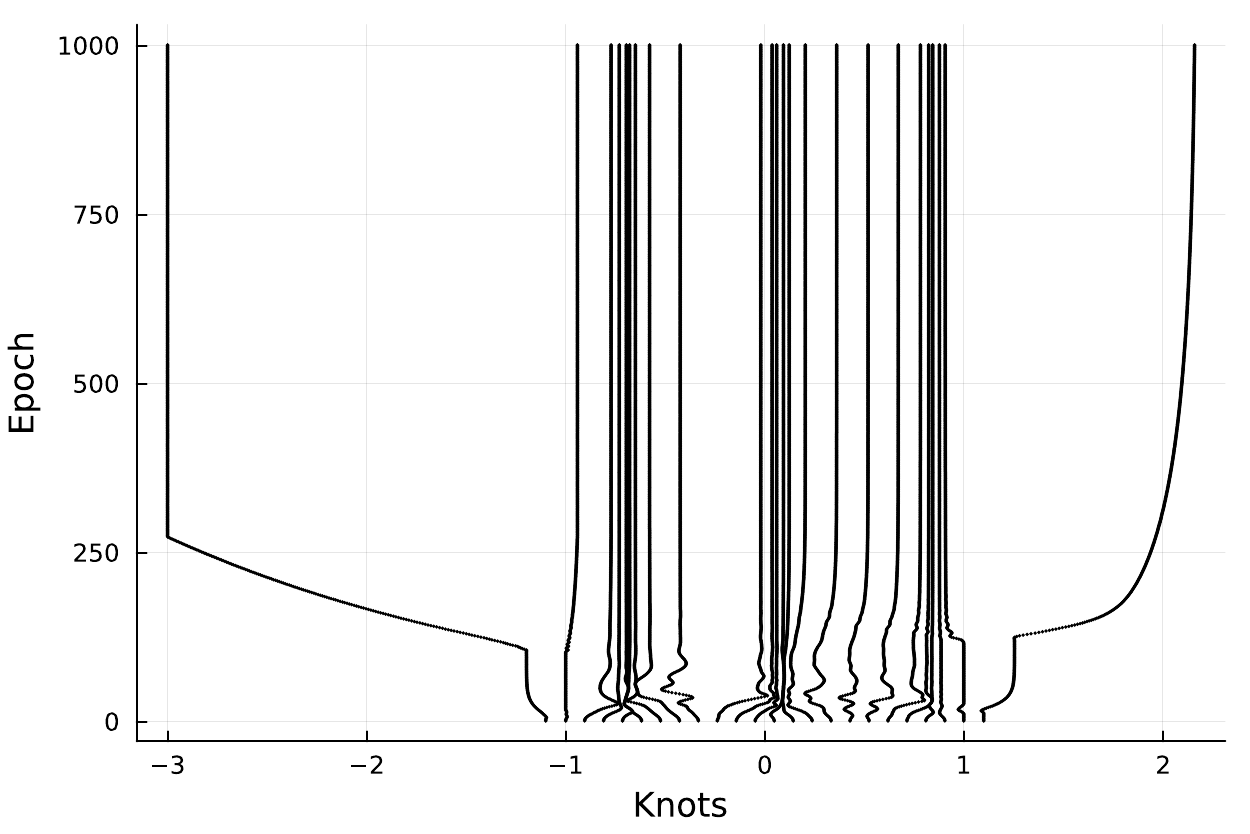}
    }%
    \caption{Function approximation problem with a linear free-knot B-spline and $12$ free knots. (\protect\subref{fig:1-approx-pointwise}) Pointwise error for the solutions on the uniform and optimised meshes. (\protect\subref{fig:1-approx-knots}) Evolution of the nonlinear parameters during the optimisation process.}
    \label{fig:1-approx}
\end{figure}

\subsubsection{Poisson problem}
We now turn to solving a one-dimensional Poisson problem. Using the method of manufactured solution, we select the right-hand side so that the exact solution is
$$u^{\star}(x) = (x^{2} - 1) \tanh(100 \sin(x - 0.3)),$$
which satisfies homogeneous Dirichlet boundary conditions at $x = \pm 1$. This function is smooth but exhibits sharp localised transitions due to the steep gradient of the hyperbolic tangent term, creating a problem that is difficult to resolve on uniform meshes.

In this example, the source term lies in $L^{2}(\Omega)$, so our analysis shows that the discrete energy functional is continuously differentiable for all B-spline degrees $p \geq 2$. The gradient of the energy is Hölder continuous with respect to the nonlinear parameters when $p = 2$, and becomes Lipschitz continuous for all $p \geq 3$. For the case $p = 1$, since the source term in fact belongs to $C^{\infty}(\overline{\Omega})$ and has bounded derivatives, we infer that the source is Lipschitz, so the energy is differentiable and Lipschitz in the mesh parameters.

\begin{figure}
    \centering
    \includegraphics[width=0.45\linewidth]{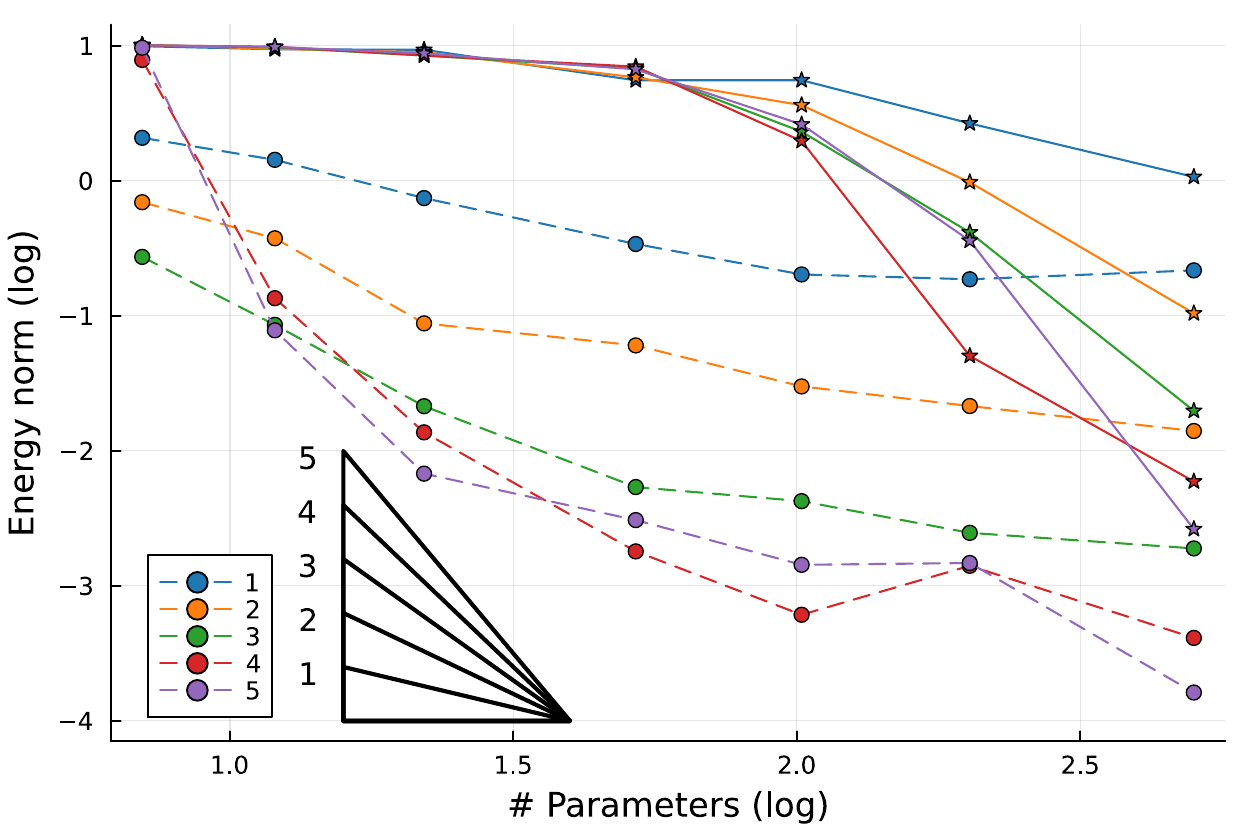}
    \caption{Convergence rate for the Poisson problem in 1D with B-splines of degrees $1$ to $5$. Dashed lines correspond to uniform meshes, while solid lines represent adapted meshes.}
    \label{fig:1_poisson_convergence}
\end{figure}

\fig{1_poisson_convergence} shows the convergence of the numerical solution in the energy norm as a function of the number of linear \acp{dof}, for B-spline degrees ranging from $1$ to $5$. For sufficiently fine discretisations, uniform B-splines exhibit the expected convergence rate of $O(h^{p})$ in the energy norm ($H^{1}$ semi-norm), which corresponds to a rate of $O(n^{-p})$ with respect to the number of \acp{dof}. In contrast, free-knot B-splines achieve substantial error reduction, up to three orders of magnitude, for moderate numbers of \acp{dof}, by concentrating resolution in regions of lower regularity. However, as the number of \acp{dof} increases, the associated optimisation problem becomes more difficult, and accurate numerical integration becomes increasingly critical. Some fluctuations in the error curves suggest that convergence has not been fully attained in these cases, despite efforts to improve stability by tuning the learning rate.

\subsection{2D experiments}

We now turn to two-dimensional experiments. We construct free-knot tensor-product B-spline spaces as sums of $1$, $4$, or $9$ patches, initially arranged in $1 \times 1$, $2 \times 2$, or $3 \times 3$ layouts, respectively. For each polynomial degree $p \in \rg{1}{5}$ (and also degree $0$ for the function approximation problem), we consider the following configurations:
\begin{itemize}
    \item For all three layouts ($1$, $4$, and $9$ patches), we consider grids of $k \times k$ cells per patch, for $k \in \rg{p}{5}$.
    \item For $1$ patch, we additionally test finer resolutions with $10 \times 10$, $20 \times 20$, and $50 \times 50$ cells per patch.
    \item For $4$ patches, we include configurations with $7 \times 7$, $10 \times 10$ and $15 \times 15$ cells per patch.
    \item For $9$ patches, we also test $10 \times 10$ cells per patch.
\end{itemize}

All knot vectors are initialised to emulate a uniform B-spline space, as described in \subsect{initialisation}.

\subsubsection{Function approximation}

We begin with a two-dimensional function approximation problem. Since the target function must be expressed as a sum of tensor-product functions, we consider a truncated version of an exponential composed with a function of $x$ and $y$. Specifically, we approximate the function $(x, y) \mapsto \exp(-\alpha(g(x) - h(y))^{2})$, where $\alpha = 5$, $g(x) = \cos(\pi x + 1)$ and $h(y) = 3 y$ by
$$f(x, y) = \sum_{k = 0}^{N} \frac{(2 \alpha)^{k}}{k!} \exp(-\alpha g(x)^{2}) g(x)^{k} \exp(-\alpha h(y)^{2}) h(y)^{k},$$
where we take $N = 10$ in this experiment. This function exhibits anisotropic features that are well suited for testing the approximation power of our free-knot B-spline spaces, as can be seen on \fig{2-approx-solution}.

The results are summarised in convergence plots in \fig{2-approx-convergence}, where the vertical axis shows the energy norm of the approximation error and the horizontal axis represents the number of linear \acp{dof}. Each polynomial degree is shown in a different color for visual clarity.

\begin{figure}
    \centering
    \subfloat[$1$ patch\label{fig:2-approx-convergence-1}]{
        \includegraphics[width=0.45\linewidth]{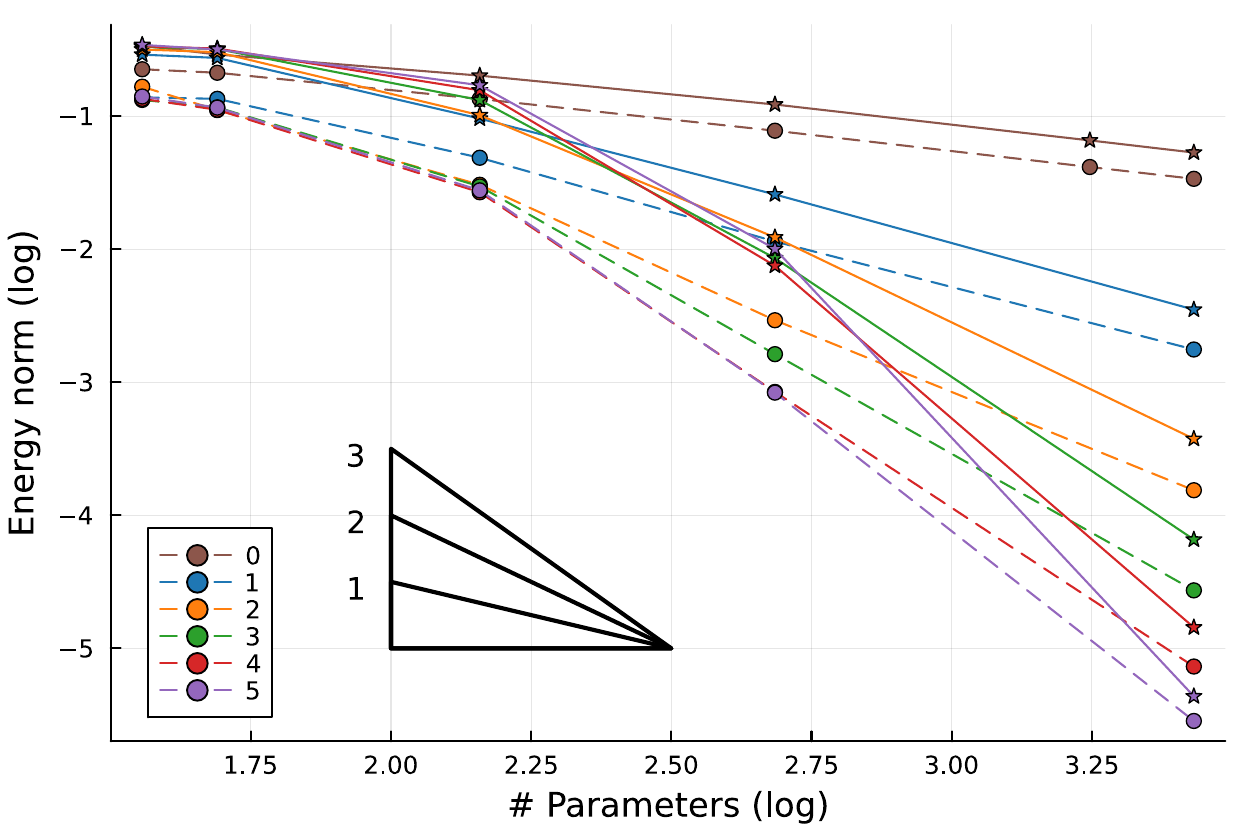}
    }%
    \hfill%
    \subfloat[$4$ patches\label{fig:2-approx-convergence-2}]{
        \includegraphics[width=0.45\linewidth]{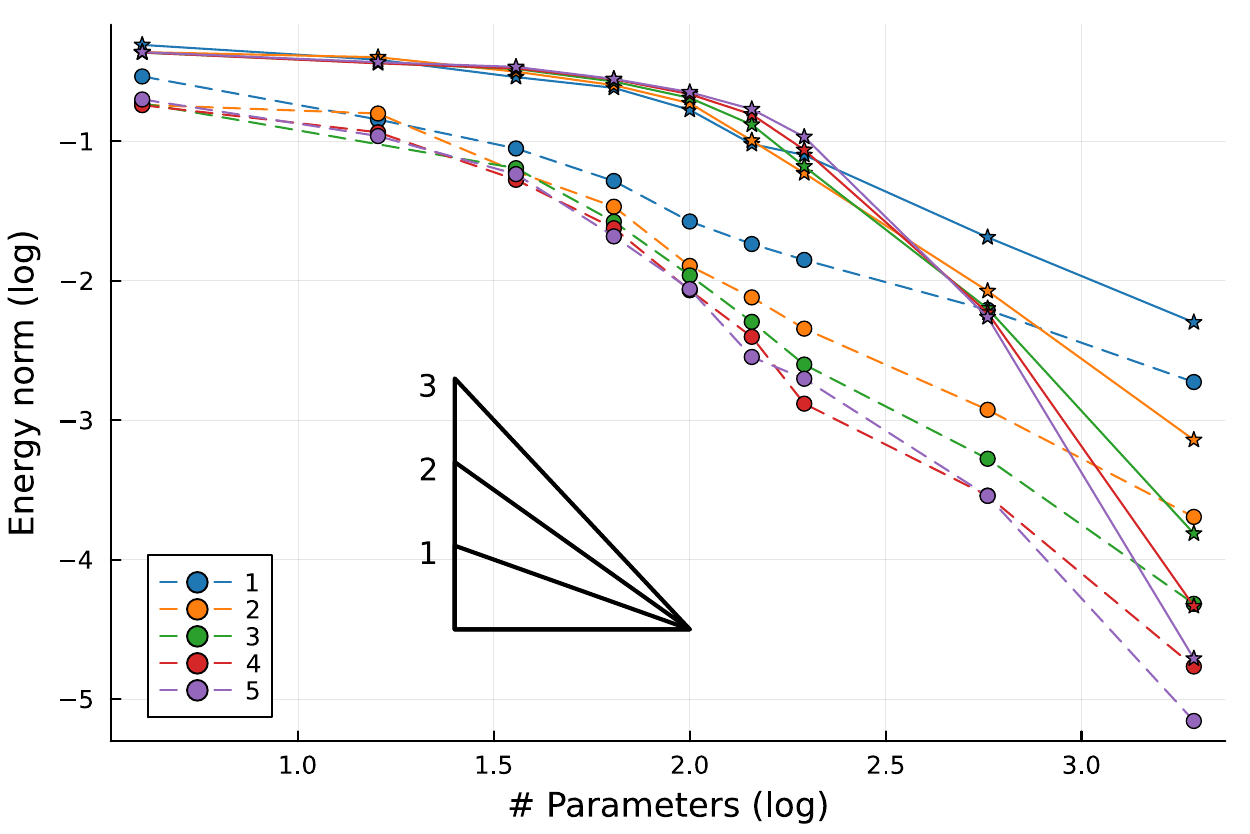}
    }%
    \hfill%
    \subfloat[$9$ patches\label{fig:2-approx-convergence-3}]{
        \includegraphics[width=0.45\linewidth]{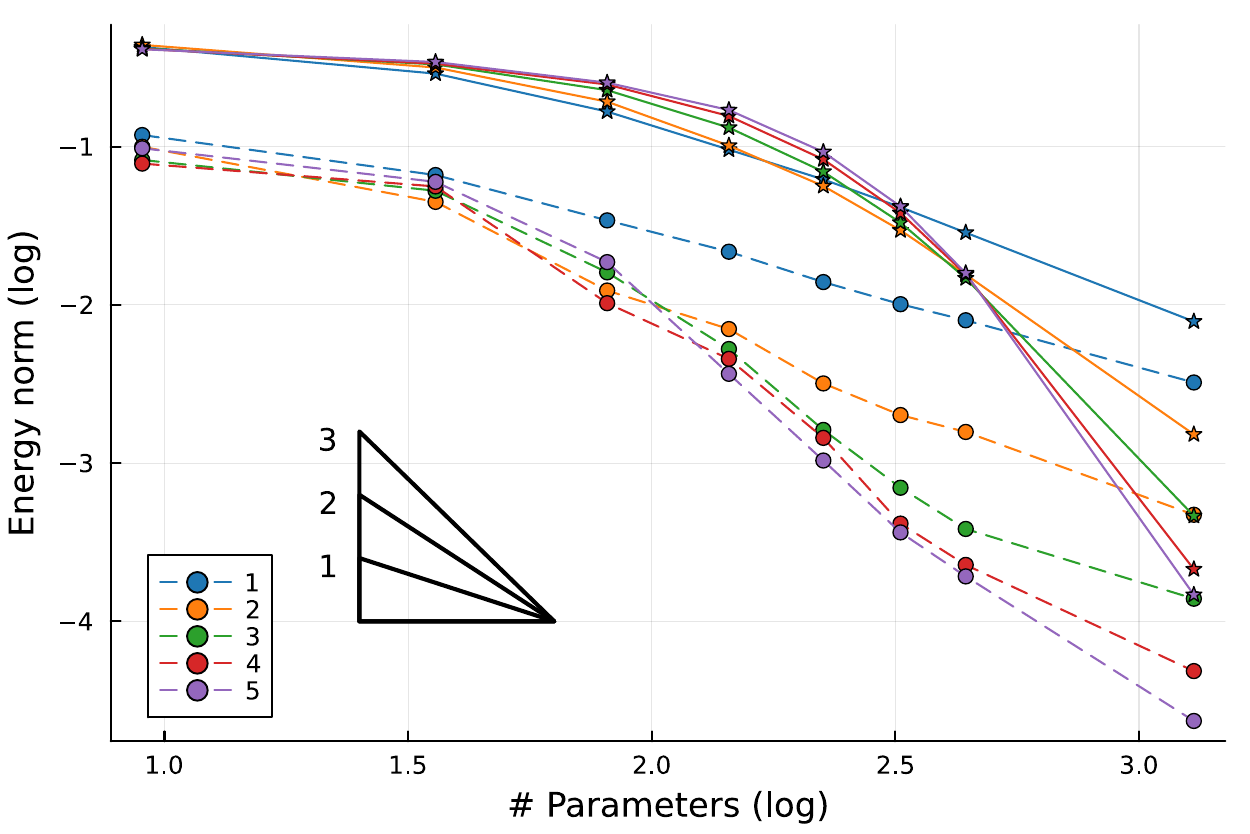}
    }%
    \caption{Convergence plots for the approximation problem in 2D with $1$, $4$ or $9$ B-spline patches of degrees $0$ to $5$, in the energy norm. Solid lines represent uniform meshes while dashed lines correspond to adapted meshes.}
    \label{fig:2-approx-convergence}
\end{figure}

In two dimensions, uniform B-splines of degree $p$ are known to approximate functions in $H^{p+1}$ at a rate of $O(h^{p+1})$, in the $L^{2}$ norm, where $h$ is the mesh size \cite[Section 3]{bazilevs2006isogeometric}. On an $n \times n$ grid, this translates to a rate $O(n^{-(p+1)/2})$ in terms of number of \acp{dof}. The corresponding rates for $p \in \{0, 5\}$ are depicted in each convergence plot.

\fig{2-approx-convergence-1} shows that for a single patch, free-knot B-splines improve accuracy by roughly half to one order of magnitude compared to their uniform counterparts, though the gains diminish as the number of \acp{dof} increases. This is expected: as the mesh becomes finer, it captures most features of the target function, leaving limited room for further improvement via adaptivity.

This behaviour reflects a more general trend: when the number of mesh parameters is too small, the mesh lacks the flexibility to meaningfully adapt, and the resulting improvement is negligible. Conversely, when the number of \acp{dof} is already high, the uniform approximation is sufficiently accurate, and further adaptation yields little benefit. The most significant gains from free-knot B-splines are observed in the intermediate regime, where the mesh is coarse enough that adaptivity matters, but flexible enough to respond to the local features of the target function.

This observation motivated the use of more patches to obtain more substantial gains for moderate numbers of \acp{dof}. \fig{2-approx-convergence-2} and \fig{2-approx-convergence-3} confirm that using multiple patches yields more significant improvements, with gains ranging from one to two orders of magnitude.

To complement the convergence results, \fig{2-approx-comparison} displays the target function, the solution at initialisation in the space composed of $4$ patches of quadratic B-splines on a $7 \times 7$ grid, and the solution after optimisation of the nonlinear parameters. The adapted solution clearly captures the sharp ridge aligned with the curve $\cos(\pi x + 1) = 3y$, which is blurred in the uniform case. This illustrates the core advantage of mesh adaptivity: concentrating resolution where it is most needed to resolve localised features without increasing the global model complexity.

\begin{figure}
    \centering
    \subfloat[\label{fig:2-approx-solution}]{
        \includegraphics[width=0.31\linewidth]{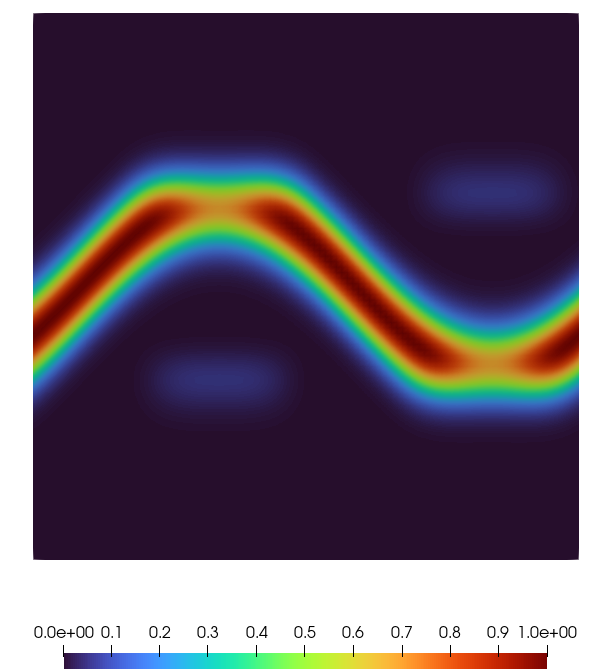}
    }%
    \hfill%
    \subfloat[\label{fig:2-approx-uniform-solution}]{
        \includegraphics[width=0.31\linewidth]{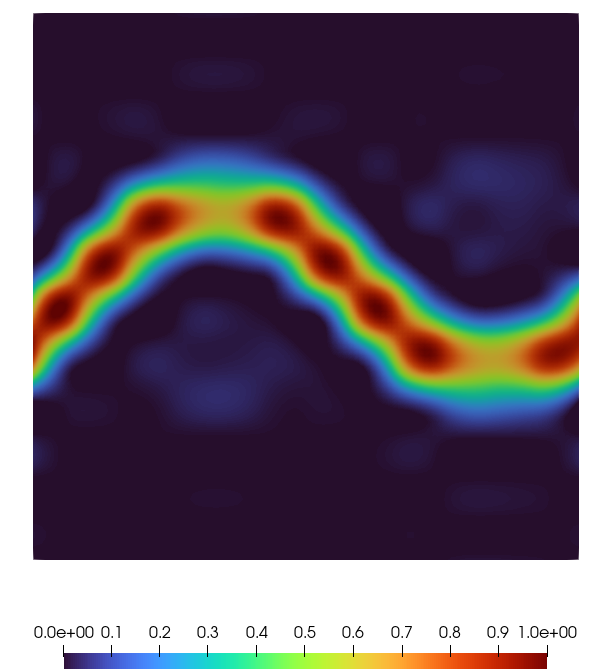}
    }%
    \hfill%
    \subfloat[\label{fig:2-approx-adapted-solution}]{
        \includegraphics[width=0.31\linewidth]{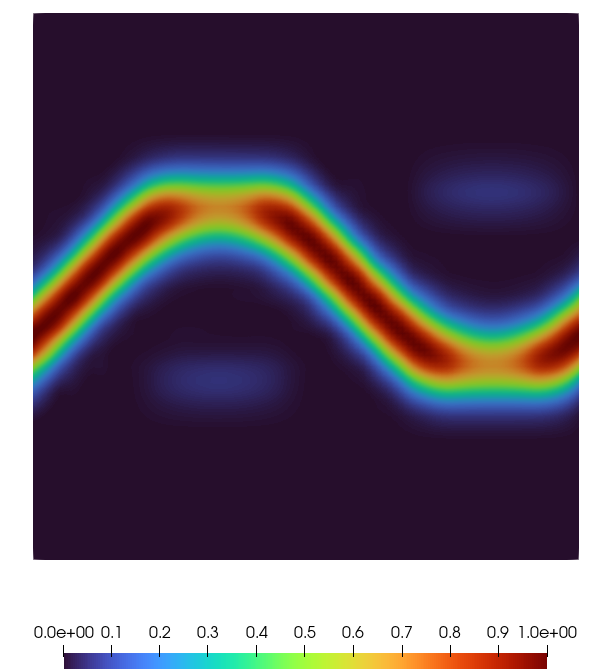}
    }%
    \caption{(\subref{fig:2-approx-solution}) Target solution for the function approximation problem in 2D. (\subref{fig:2-approx-uniform-solution}) Solution the space composed of $4$ patches of quadratic B-splines on $7 \times 7$ grids, at initialisation. (\subref{fig:2-approx-adapted-solution}) Solution in that same space after optimisation of the nonlinear parameters.}
    \label{fig:2-approx-comparison}
\end{figure}

\subsubsection{First Poisson problem}

We solve a Poisson problem on $\cc{-1}{+1}^{2}$ with homogeneous Dirichlet boundary conditions, whose exact solution $u$ is defined as follows
\begin{align*}
    u(x, y)     & = u_{0}(x, y) v(x, y),                          \\
    u_{0}(x, y) & = (1 - x^{2}) (1 - y^{2}),                      \\
    v(x, y)     & = (1 - \tanh(20(x-0.3))) (1 - \tanh(50(y+0.3))) \\
                & + \tanh(50(x+0.7)) (1 - \tanh(20(y-0.6))).
\end{align*}
This function exhibits sharp transitions along the hyperplanes $x = 0.3$, $x = -0.7$, $y = -0.3$ and $y = 0.6$, as well as extended regions where it is nearly zero. These features make it a compelling test case for examining mesh movement and adaptivity. The function is visualised in \fig{2-poisson-tanh-solution}.

\begin{figure}
    \centering
    \subfloat[$1$ patch\label{fig:2-poisson-tanh-convergence-1}]{
        \includegraphics[width=0.45\linewidth]{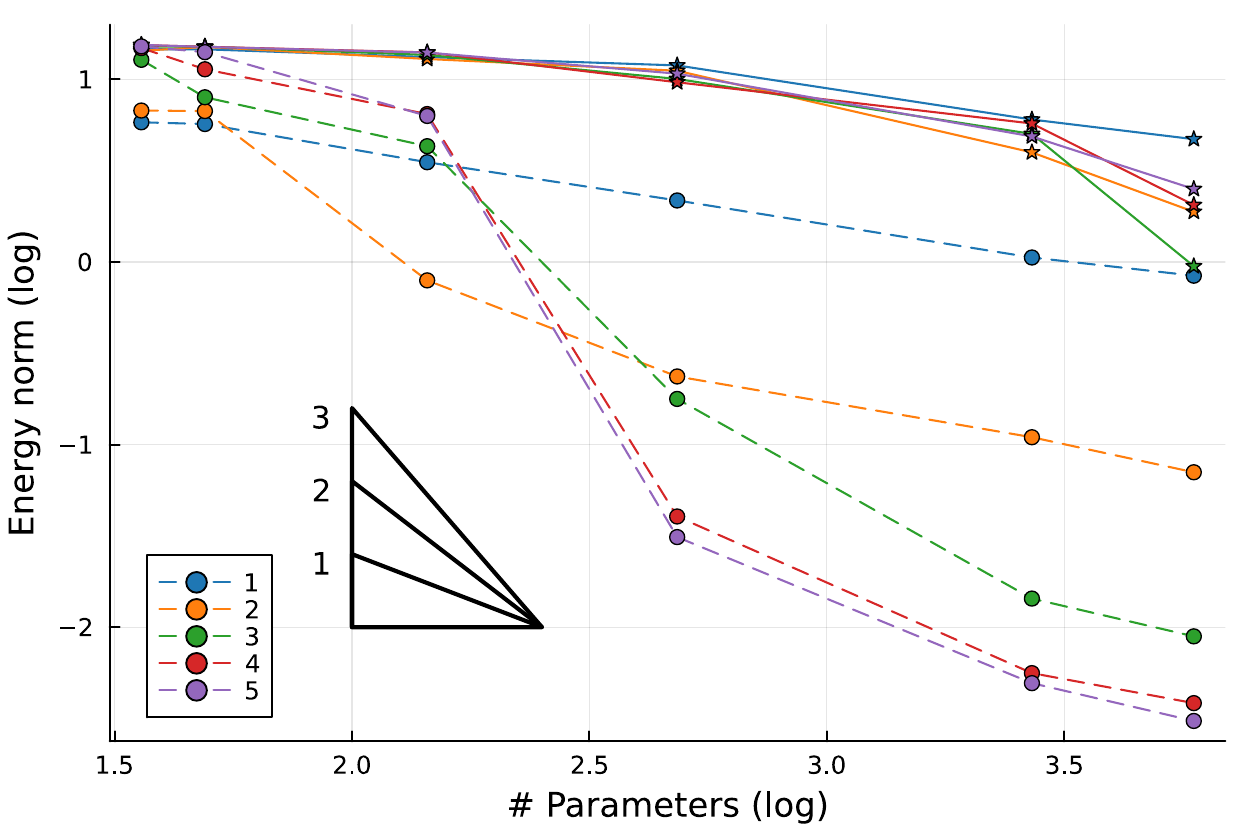}
    }%
    \hfill%
    \subfloat[$4$ patches\label{fig:2-poisson-tanh-convergence-2}]{
        \includegraphics[width=0.45\linewidth]{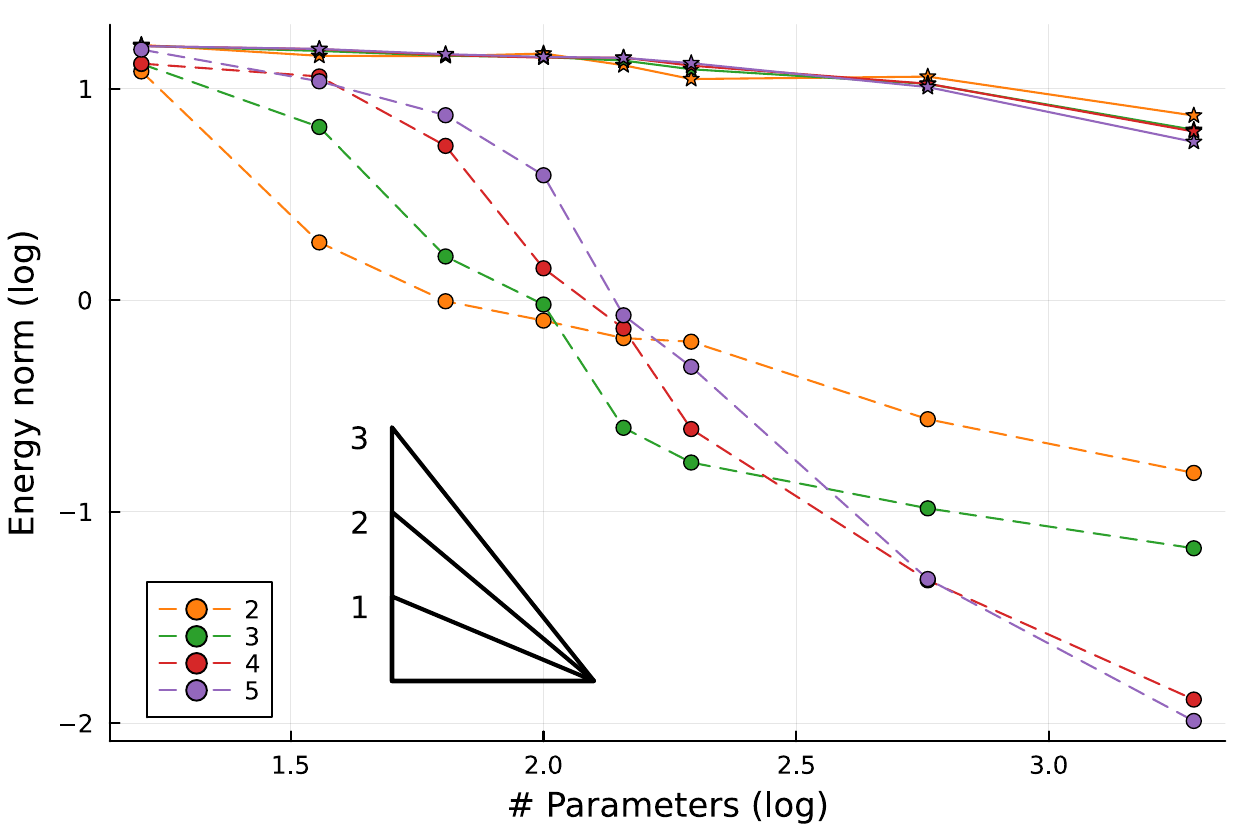}
    }%
    \hfill%
    \subfloat[$9$ patches\label{fig:2-poisson-tanh-convergence-3}]{
        \includegraphics[width=0.45\linewidth]{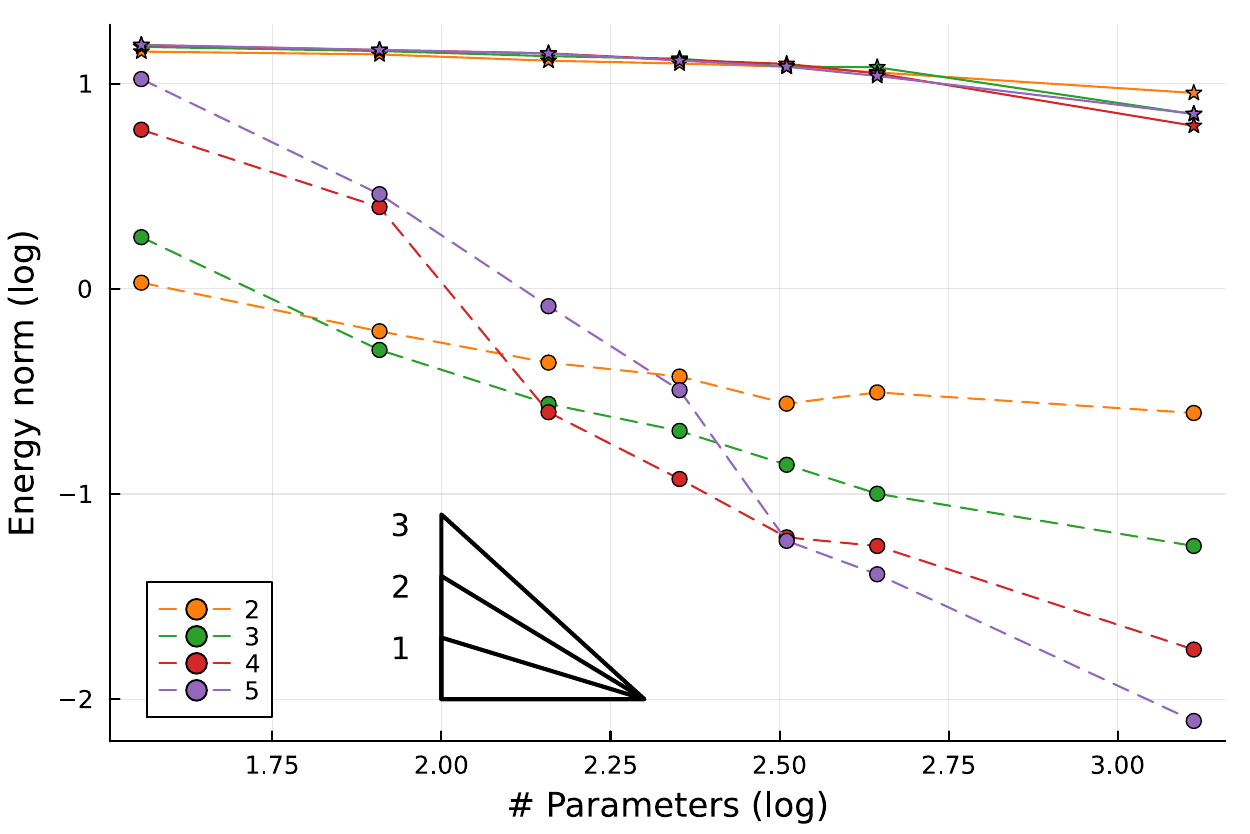}
    }%
    \caption{Convergence plots for the first Poisson problem in 2D with $1$, $4$ or $9$ B-splines patches of degrees $1$ to $5$, in the energy norm. Solid lines represent uniform meshes while coloured circles correspond to adapted meshes.}
    \label{fig:2-poisson-tanh-convergence}
\end{figure}

\fig{2-poisson-tanh-convergence} presents convergence results in the energy norm for B-spline degrees $1$ through $5$, across various patch configurations.  In two dimensions, uniform B-splines of degree $p$ approximate functions in $H^{p+1}$ at a rate of $O(h^{p})$, in the $H^{1}$ norm \cite[Section 3]{bazilevs2006isogeometric}, which corresponds to $O(n^{-p/2})$ on an $n \times n$ grid. These expected rates are shown in the convergence plots for $p \in \{1, 5\}$.

In this example, uniform B-splines struggle to approximate the target function, with the error decreasing very slowly as the mesh is refined. This indicates that, for all degrees, the range of \acp{dof} explored in our experiments remains within the pre-asymptotic regime for uniform meshes.

In contrast, free-knot B-splines achieve substantial error reduction, with improvements ranging from one to three orders of magnitude, even for relatively large numbers of \acp{dof}. The same qualitative behavior observed in earlier examples holds here: with too few free parameters, gains are minimal, while with a moderate number, significant improvements are possible. Based on \fig{2-poisson-tanh-convergence-1}, the transition out of the pre-asymptotic regime for uniform B-splines appears to occur near the upper end of the explored range, around $6000$ \acp{dof}.

\begin{figure}
    \centering
    \subfloat[\label{fig:2-poisson-tanh-solution}]{
        \includegraphics[width=0.31\linewidth]{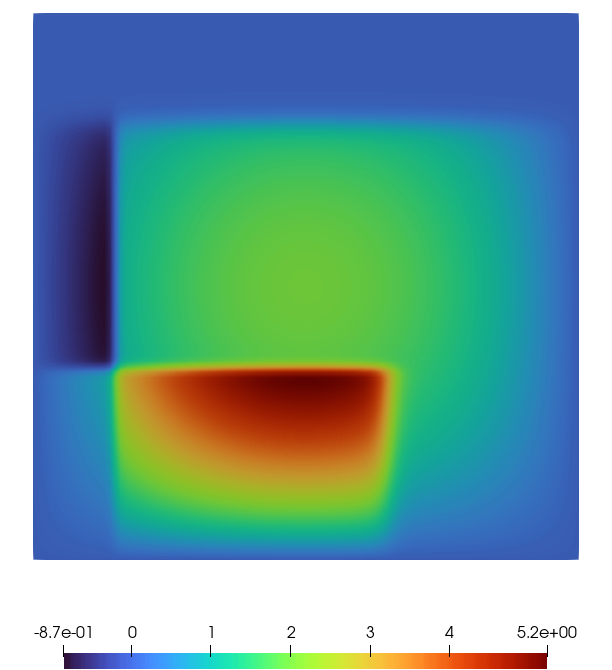}
    }%
    \hspace{0.2\linewidth}%
    \subfloat[\label{fig:2-poisson-tanh-patches}]{
        \includegraphics[width=0.31\linewidth]{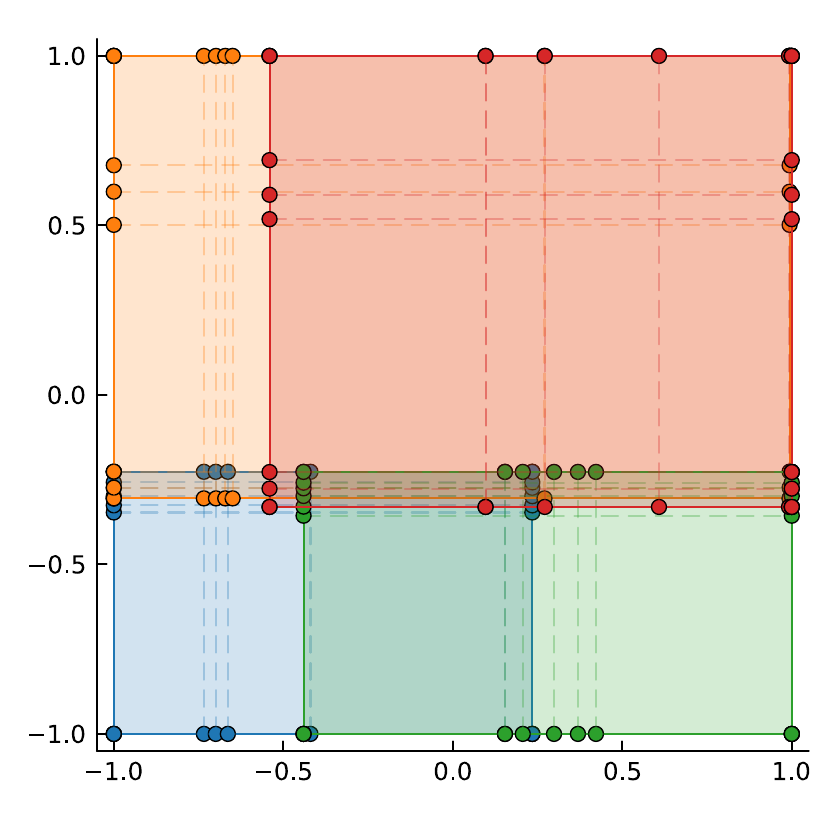}
    }%
    \caption{(\subref{fig:2-poisson-tanh-solution}) Exact solution for the second Poisson problem problem in 2D. (\subref{fig:2-poisson-tanh-patches}) Distribution of the patches after optimisation, for 4 patches of quadratic B-splines of size $6 \times 6$.}
    \label{fig:2-poisson-tanh-comparison}
\end{figure}

To give an example of mesh adaptation in 2D, \fig{2-poisson-tanh-patches} displays the patch layout after optimisation for a configuration with $4$ patches of quadratic B-splines, each defined over a $6 \times 6$ free-knot grid. The mesh deformation highlights the optimiser's ability to concentrate resolution around regions with sharp transitions. Notably, this concentration is achieved either by compressing knots within a single patch (e.g., the blue and orange patches near $x = -0.7$) or by overlapping multiple patches near critical regions (e.g., the green and red patches near $y = -0.3$). The point $(-0.7, 0.3)$, where sharp features intersect, is jointly refined by the blue, orange, and red patches, enabling resolution of the near-singularity in both directions.

\begin{figure}
    \centering
    \subfloat[Uniform mesh\label{fig:2-poisson-tanh-uniform-error}]{
        \includegraphics[width=0.31\linewidth]{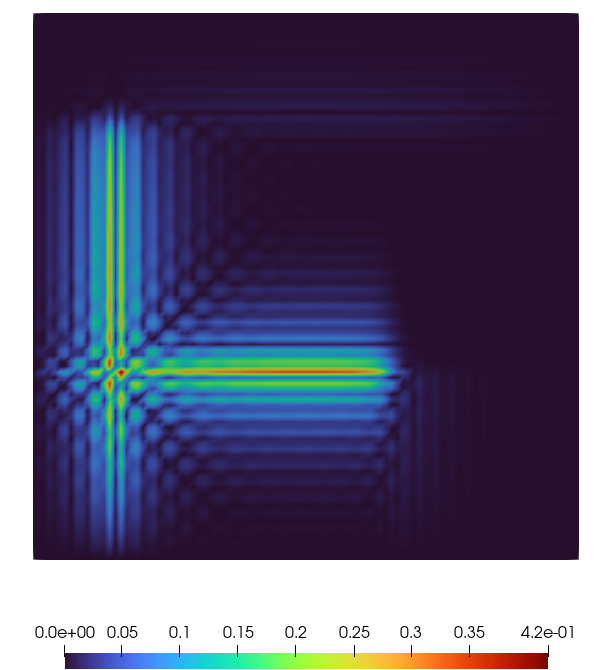}
    }%
    \hspace{0.2\linewidth}%
    \subfloat[Adapted mesh\label{fig:2-poisson-tanh-adapted-error}]{
        \includegraphics[width=0.31\linewidth]{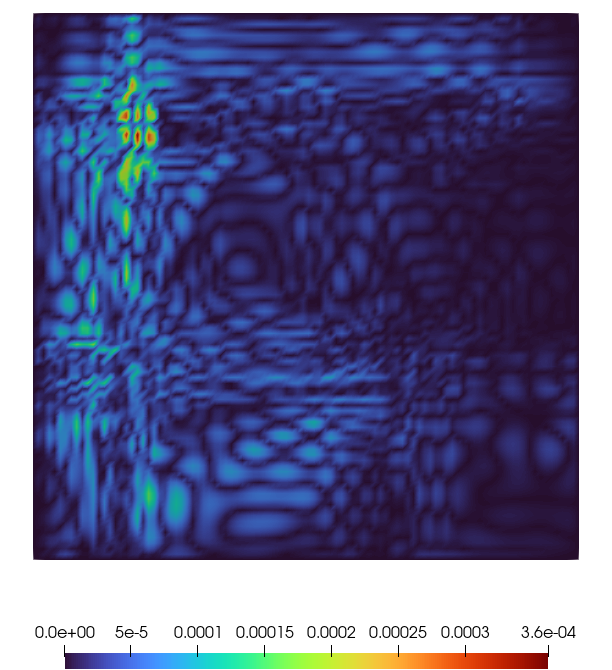}
    }%
    \caption{Pointwise error for the first Poisson problem with $9$ patches of B-splines of order $5$ on $10 \times 10$ grids, with a uniform mesh (\subref{fig:2-poisson-tanh-uniform-error}) or adapted mesh (\subref{fig:2-poisson-tanh-adapted-error}).}
    \label{fig:2-poisson-tanh-error}
\end{figure}

Finally, \fig{2-poisson-tanh-error} shows the pointwise error before and after mesh adaptation for the case with $9$ B-spline patches of order $5$, each defined on a $10 \times 10$ grid. In the uniform case, the error is highly concentrated around a single point, whereas after adaptation, the error is almost evenly distributed across the domain. This illustrates a key mechanism of energy-based optimisation: it implicitly drives the mesh to equidistribute the error, balancing local contributions to the global objective.

This behavior aligns with the principles of mesh adaptation via mesh movement \cite{budd2009adaptivity}, where the mesh is explicitly evolved to equidistribute a monitor function related to the solution error. In their formulation, this leads to solving a Monge-Ampère-type optimal transport problem, which constructs a mesh that equidistributes some error estimate. Although our approach is based on energy minimisation rather than explicitly prescribed mesh movement, the resulting adapted meshes exhibit similar features to those produced in the moving mesh framework. This highlights a strong conceptual link between our optimisation-based method and classical approaches to adaptivity using dynamic grids, both aiming to concentrate resolution where it is most needed, while enforcing geometric constraints to avoid degenerate meshes.

\subsubsection{Second Poisson problem}

We finally consider a Poisson problem on $\Omega = [-1, 1]^2$ with homogeneous Dirichlet boundary conditions, whose source term is chosen so that the exact solution is given by
$$u^{\star}(x, y) = (x^{2} - 1) (y^{2} - 1) \exp(-3 (x + 0.3)^{2} - (y - 0.5)^{2}) \cos(x + y).$$
This function exhibits a sharp peak concentrated around the point $(x, y) = (0.8, 0)$, and a sharp transition from that point to the boundary $x = 1$, making it a demanding test for uniform discretisations.

\begin{figure}
    \centering
    \subfloat[$1$ patch\label{fig:2-poisson-peak-convergence-1}]{
        \includegraphics[width=0.45\linewidth]{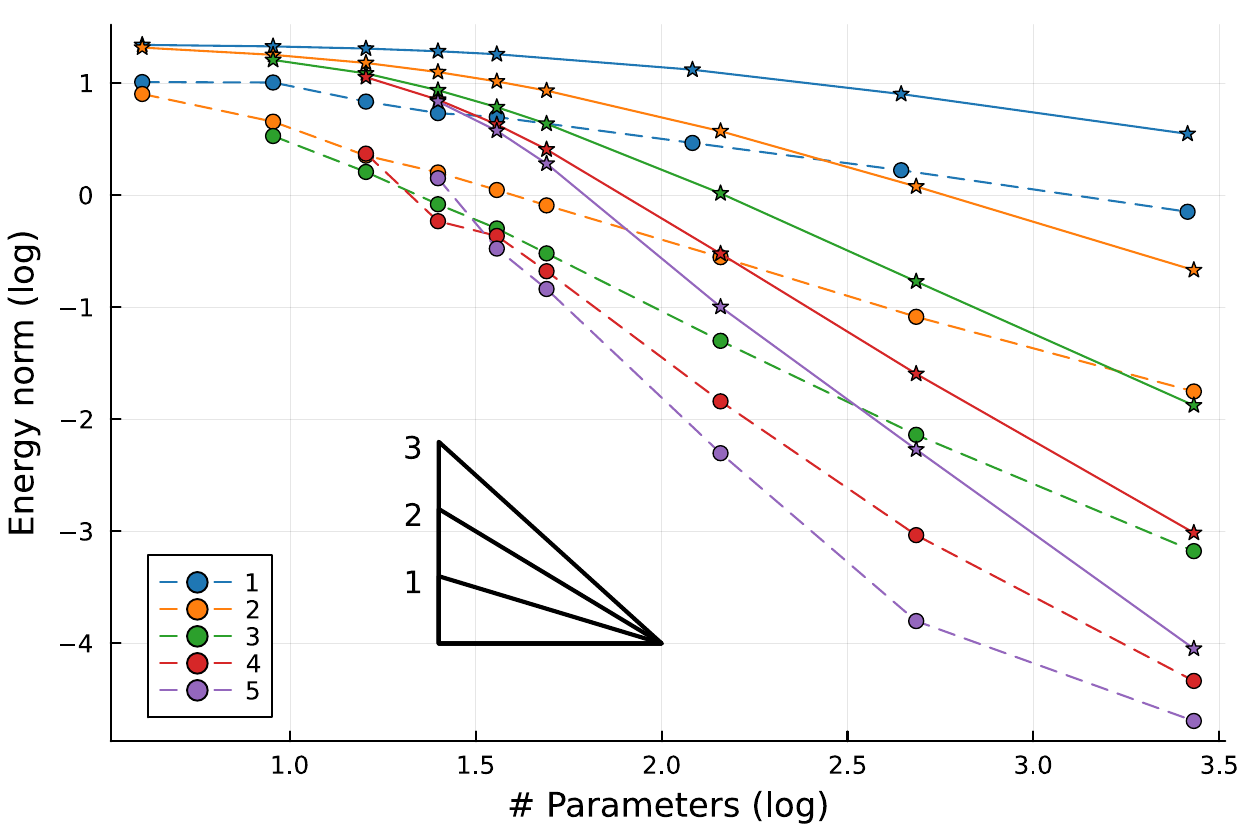}
    }%
    \hfill%
    \subfloat[$4$ patches\label{fig:2-poisson-peak-convergence-2}]{
        \includegraphics[width=0.45\linewidth]{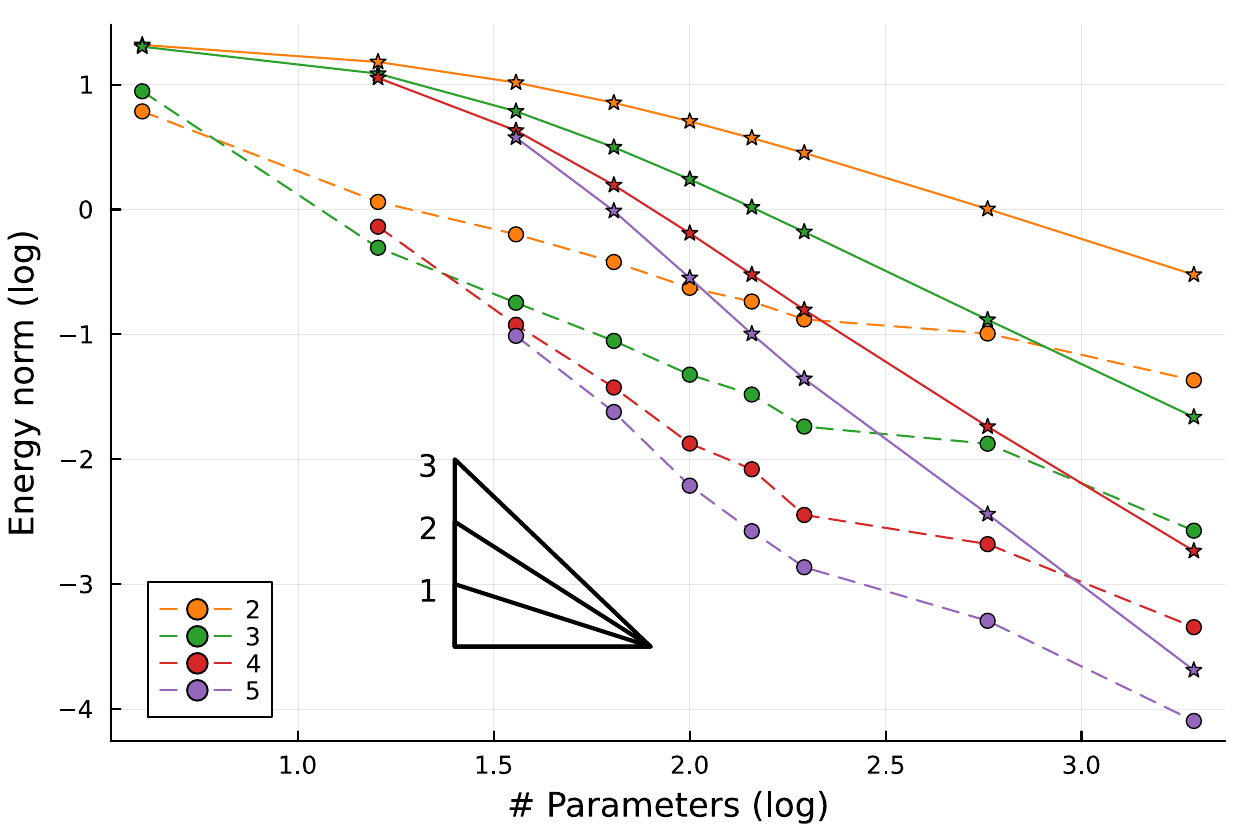}
    }%
    \hfill%
    \subfloat[$9$ patches\label{fig:2-poisson-peak-convergence-3}]{
        \includegraphics[width=0.45\linewidth]{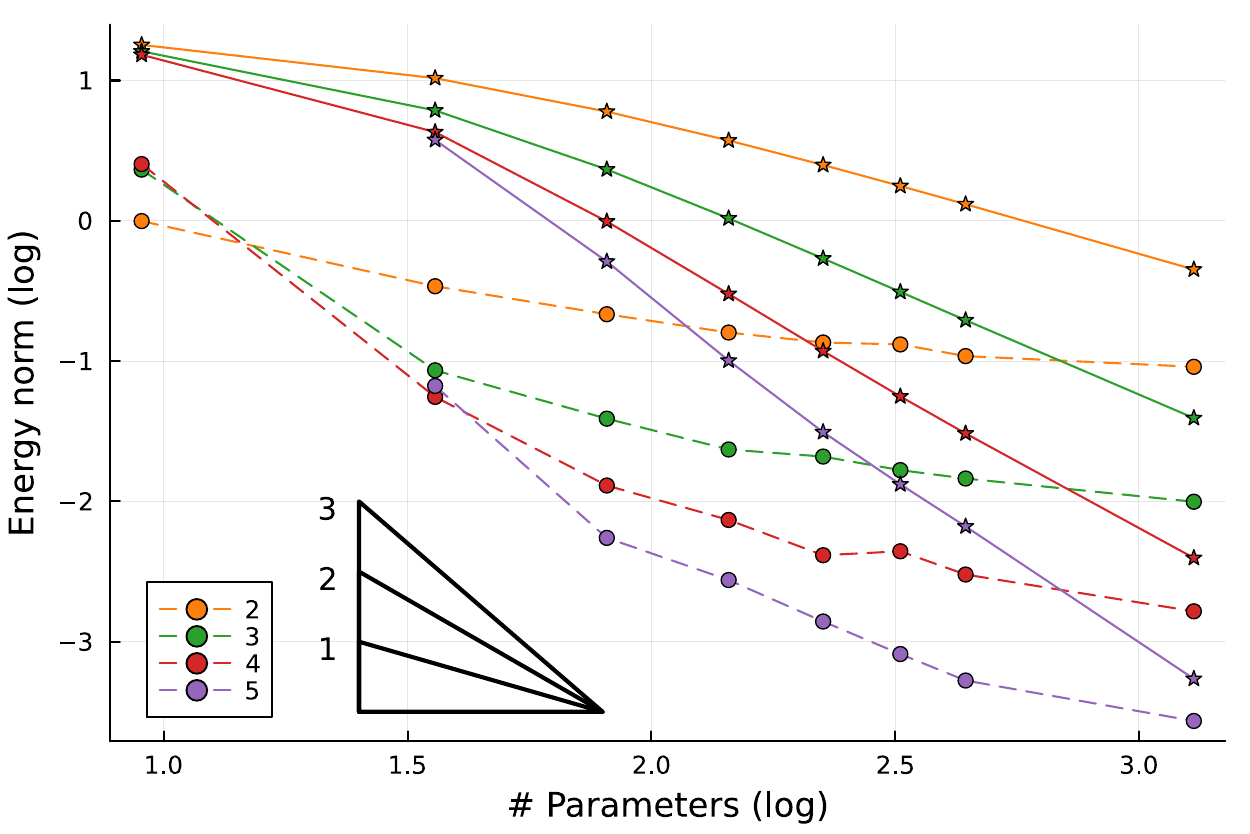}
    }%
    \caption{Convergence plots for the second Poisson problem in 2D with $1$, $4$ or $9$ B-splines patches of degrees $1$ to $5$, in the energy norm. Solid lines represent uniform meshes while coloured circles correspond to adapted meshes.}
    \label{fig:2-poisson-peak-convergence}
\end{figure}

\fig{2-poisson-peak-convergence} shows the convergence of the energy norm of the error with respect to the number of linear \acp{dof}, comparing uniform and adapted meshes across various spline degrees and patch configurations.

In this example, adaptivity yields substantial gains even in the single-patch case, with error reductions of one to two orders of magnitude. The cases with $4$ or $9$ patches clearly demonstrate the core benefit of adaptivity: it is most effective at intermediate resolutions. With too few \acp{dof}, the mesh lacks flexibility; with too many, the uniform mesh already performs well. In between, adaptivity significantly improves accuracy by reallocating resolution to the most critical regions.

\section{Conclusion}
\label{sect:conclusion}
In this work, we design a space of overlapping tensor-product free-knot B-spline patches. By treating knot positions as nonlinear parameters, this construction enables geometric adaptivity while preserving smoothness and avoiding the rigidity of global tensor-product grids. We verify that the associated discrete energy functional satisfies the structural assumptions required by the convergence analysis of our companion paper \cite{companion}, ensuring both local and global convergence of the hybrid optimisation algorithm.

We demonstrate the effectiveness of free-knot B-spline spaces through a suite of numerical experiments. The results show that our approach achieves markedly improved accuracy compared to standard \ac{iga}, using significantly fewer degrees of freedom. The ability to resolve localised solution features with high precision underscores the strength of combining geometric adaptivity with energy-based optimisation. Although our experiments focus on low-dimensional problems, the framework naturally extends to higher dimensions and offers promising scalability with parallel implementations.

While this work focuses on B-spline spaces, other adaptive or free-knot piecewise polynomial spaces, such as Bernstein or Hermite splines, could also be incorporated within the same framework. We expect that the numerical analysis required to verify the abstract convergence assumptions follow similar lines. More broadly, our approach could apply to other nonlinear approximation manifolds, including spaces based on wavelets or compactly supported radial basis functions, as well as sparse grid constructions that retain only selected tensor-product blocks in hierarchical decompositions \cite{bungartz2004sparse}.

Beyond exploring these approximation spaces, we are also interested in developing more efficient optimisation strategies, such as incorporating second-order information or problem-adapted preconditioners, to improve convergence rates and robustness in high-dimensional problems. We also plan to investigate extending our framework to more complex geometries using advanced unfitted techniques \cite{dePrenter2023}. Additional code development will also be required to design efficient implementations of the proposed nonlinear methods and tailored nonlinear solvers to tackle more challenging problems, particularly in high dimensions. One possible direction is to develop matrix-free multigrid solvers, which would enable the use of automatic differentiation tools to compute gradients and Hessians without explicitly forming matrices. We leave these endeavours for future work.

\section*{Acknowledgements}
\label{sect:acknowledgements}
This research was partially funded by the Australian Government through the Australian Research Council (project DP220103160). A. Magueresse gratefully acknowledges the Monash Graduate Scholarship from Monash University.

\clearpage

\appendix

\section{Proofs}
\label{sect:proofs}
\subsection{Common tools for B-splines}
\label{app:toolsbsplines}

We start by recalling some properties of divided differences. Given $r \geq 0$, $y_{1}, \ldots, y_{r+1} \in \RR$ and $f: \RR \to \RR$ sufficiently smooth, the finite difference operator is defined inductively via
\begin{equation}
    \label{eq:divided-difference}
    [y_{1}, \ldots, y_{r+1}] f = \frac{[y_{2}, \ldots, y_{r+1}] f - [y_{1}, \ldots, y_{r}] f}{y_{r+1} - y_{1}}
\end{equation}
provided $y_{1} \neq y_{r}$, and
$$[y_{1}, \ldots, y_{r+1}] f = \frac{D^{r} f(y_{1})}{r!}$$
if $y_{1} = \ldots = y_{r}$, where $D^{r} f$ denotes the $r$-th derivative of $f$.

We now list a two common tools for the proofs of \lem{boundedness} and \lem{holder}. First, it is known \cite[Theorem 4.23]{schumaker2007spline} that $\norm{N_{p}(\mb{\tau})}_{L^{1}(\RR)} = 1 / (p + 1)$ for all $\mb{\tau} \in \md{K}_{p+2}$, which implies $\norm{B_{p}(\mb{\tau})}_{L^{1}(\RR)} = \abs{\mb{\tau}} / (p + 1)$. Since $B_{p}(\mb{\tau})$ takes values in $\cc{0}{1}$, we reach $\norm{B_{p}(\mb{\tau})}_{L^{2}(\RR)}^{2} \leq \abs{\mb{\tau}} / (p + 1)$. Therefore,
\begin{equation}
    \label{eq:bound-bspline-norm}
    C_{p}^{2} \doteq \sup_{\mb{\tau} \in \md{K}_{p+2}} \frac{\norm{B_{p}(\mb{\tau})}_{L^{2}(\RR)}^{2}}{\abs{\mb{\tau}}} \leq \frac{1}{p + 1}.
\end{equation}

Second, note that $\remBeg$ and $\remEnd$ commute. For all $\alpha, \beta \geq 0$, Let then $\remBeg^{\alpha} \remEnd^{\beta}$ denote $\alpha$ applications of $\remBeg$ and $\beta$ applications of $\remEnd$. For all $\mb{\tau} \in \md{K}_{p}$, it holds
\begin{equation}
    \label{eq:bound-knots}
    \abs{\mb{\tau} \remBeg^{\alpha} \remEnd^{\beta}} \geq (p - 1 - \alpha - \beta) h(\mb{\tau}).
\end{equation}
Besides, $\abs{\mb{\tau} \oplus x} \geq \abs{\mb{\tau}}$ for all $x \in \RR$.

\subsection{Proof of \texorpdfstring{\lem{BsplineDerivatives}}{Lemma~\ref{lem:BsplineDerivatives}}}
\label{app:BsplineDerivatives}

\begin{proof}
    Assume $p = 0$ and let $\mb{\tau} = (a, b) \in \md{K}_{2}$. Let $\phi \in C^{\infty}_{c}(\RR)$ and compute
    $$-\int_{\RR} B_{0}(\mb{\tau})(x) \phi'(x) \dd{x} = -\int_{a}^{b} \phi'(x) \dd{x} = \phi(a) - \phi(b).$$
    By definition of the distributional derivative, this shows $\partial_{x} B_{0}(\mb{\tau}) = \delta_{a} - \delta_{b}$, where $\delta_{z}$ is the Dirac delta centred at $z \in \RR$.

    For $p = 1$, $\mb{\tau} = (a, b, c) \in \md{K}_{3}$, and $\phi \in C^{\infty}_{c}(\RR)$, we compute
    \begin{align*}
        -\int_{\RR} B_{1}(\mb{\tau})(x) \phi'(x) \dd{x} & = -\int_{a}^{b} \frac{x - a}{b - a} \phi'(x) \dd{x} - \int_{b}^{c} \frac{c - x}{c - b} \phi'(x) \dd{x}           \\
                                                        & = -\phi(b) + \int_{a}^{b} \frac{1}{b - a} \phi(x) \dd{x} + \phi(b) - \int_{b}^{c} \frac{1}{c - b} \phi(x) \dd{x} \\
                                                        & = -\int_{\RR} (N_{0}(\mb{\tau} \remBeg) - N_{0}(\mb{\tau} \remEnd)) \phi(x) \dd{x}
    \end{align*}
    As above, we conclude that the weak derivative of $B_{1}(\mb{\tau})$ is $\partial_{x} B_{1}(\mb{\tau}) = -1 \cdot (N_{0}(\mb{\tau} \remBeg) - N_{0}(\mb{\tau} \remEnd))$.

    Assume now that $p \geq 2$ and let $\mb{\tau} \in \md{K}_{p+2}$. Let $x \in \RR$ and apply \eqref{eq:divided-difference} to rewrite $B_{p}(\mb{\tau})(x)$ as
    $$B_{p}(\mb{\tau})(x) = [\mb{\tau} \remBeg] (\bullet - x)_{+}^{p} - [\mb{\tau} \remEnd] (\bullet - x)_{+}^{p}.$$
    Note that $x \mapsto (y - x)_{+}^{p}$ is differentiable for all $y \in \RR$ and $\partial_{x} (y - x)_{+}^{p} = -p (y - x)_{+}^{p - 1}$. The partial derivative with respect to $x$ commutes with the finite difference operator and
    \begin{align*}
        \partial_{x} B_{p}(\mb{\tau})(x) & = [\mb{\tau} \remBeg] \partial_{x} (\bullet - x)_{+}^{p} - [\mb{\tau} \remEnd] \partial_{x} (\bullet - x)_{+}^{p} \\
                                         & = - p [\mb{\tau} \remBeg] (\bullet - x)_{+}^{p-1} + p [\mb{\tau} \remEnd] (\bullet - x)_{+}^{p-1}                 \\
                                         & = - p (N_{p-1}(\mb{\tau} \remBeg)(x) - N_{p-1}(\mb{\tau} \remEnd)(x)).
    \end{align*}

    We turn to the parametric derivatives of the B-spline functor. For $p = 0$, given $\mb{\tau} = (a, b) \in \md{K}_{2}$, $\epsilon \in \oo{0}{b - a}$ so that $(a + \epsilon, b) \in \md{K}_{2}$, and $\phi \in H^{1}(\RR)$, we compute
    $$\abs{\inner{\phi}{\frac{B_{0}(a + \epsilon, b) - B_{0}(a, b)}{\epsilon} + \delta_{a}}_{H^{-1}(\RR)}} = \abs{\phi(a) - \frac{1}{\epsilon} \int_{a}^{a + \epsilon} \phi(t) \dd{t}} \leq \frac{1}{\epsilon} \int_{a}^{a + \epsilon} \abs{\phi(t) - \phi(a)} \dd{t},$$
    which tends to $0$ as $\epsilon \to 0$ since $\phi$ is continuous at $a$. This shows $B_{0}$ is differentiable with respect to its first input as a map from $\md{K}_{2}$ to $H^{-1}(\RR)$, and $\partial_{1} B_{0}(\mb{\tau}) = -\delta_{\mb{\tau}_{1}}$. A similar reasoning applies to the differentiability with respect to the second variable and yields $\partial_{2} B_{0}(\mb{\tau}) = \delta_{\mb{\tau}_{2}}$.

    We now assume $p \geq 1$. Let $i \in \rg{1}{p+2}$ and $\epsilon \in \RR$ be such that $\mb{\tau}_{i-1} < \mb{\tau}_{i} + \epsilon < \mb{\tau}_{i+1}$, with the convention $\mb{\tau}_{0} = -\infty$ and $\mb{\tau}_{p+3} = + \infty$. Consider then the knot vector $\mb{\tau}_{\epsilon} = \mb{\tau} + \epsilon \mb{\delta}_{i}$, where $(\mb{\delta}_{i})_{j} = \delta_{ij}$. The choice of $\epsilon$ ensures $\mb{\tau}_{\epsilon} \in \md{K}_{p+2}$. Let $x \in \RR$ and apply \eqref{eq:divided-difference} to rewrite
    \begin{align*}
        \frac{1}{\epsilon} (B_{p}(\mb{\tau}_{\epsilon})(x) - B_{p}(\mb{\tau})(x)) & = \frac{1}{\epsilon} ([\mb{\tau}_{\epsilon} \remBeg] (\bullet - x)_{+}^{p} - [\mb{\tau} \remBeg] (\bullet - x)_{+}^{p}) - \frac{1}{\epsilon} ([\mb{\tau}_{\epsilon} \remEnd] (\bullet - x)_{+}^{p} - [\mb{\tau} \remEnd] (\bullet - x)_{+}^{p}) \\
                                                                                  & = [\mb{\tau} \oplus (\mb{\tau}_{i} + \epsilon) \remBeg] (\bullet - x)_{+}^{p} \delta_{i \neq 1} - [\mb{\tau} \oplus (\mb{\tau}_{i} + \epsilon) \remEnd] (\bullet - x)_{+}^{p} \delta_{i \neq p + 2}.
    \end{align*}
    The restrictions on $i$ come from the fact that if $i$ is no longer in the abscissa of the divided difference operator, then the corresponding term does not depend on $\mb{\tau}_{i}$. These restrictions allow $\remBeg$ and $\remEnd$ to commute with $\oplus$: $(\mb{\tau} \remBeg) \oplus (\mb{\tau}_{i} + \epsilon) = (\mb{\tau} \oplus (\mb{\tau}_{i} + \epsilon)) \remBeg$, and similarly with $\remEnd$. Since $x \mapsto (y - x)_{+}^{p}$ is continuous for all $y \in \RR$, the divided difference operator is continuous with respect to its abscissa. Taking the limit as $\epsilon \to 0$, we infer that $B_{p}(\mb{\tau})(x)$ is differentiable with respect to $\mb{\tau}_{i}$ and
    \begin{align}
        \label{eq:derivative-BSpline}
        \partial_{i} B_{p}(\mb{\tau})(x) & = [\mb{\tau} \oplus \mb{\tau}_{i} \remBeg] (\bullet - x)_{+}^{p} \delta_{i \neq 1} - [\mb{\tau} \oplus \mb{\tau}_{i} \remEnd] (\bullet - x)_{+}^{p} \delta_{i \neq p + 2} \nonumber \\
                                         & = N_{p}(\mb{\tau} \oplus \mb{\tau}_{i} \remBeg)(x) \delta_{i \neq 1} - N_{p}(\mb{\tau} \oplus \mb{\tau}_{i} \remEnd)(x) \delta_{i \neq p + 2}.
    \end{align}
    This shows the pointwise differentiability of $B_{p}$. To show that $B_{p}$ is differentiable as a map from $\md{K}_{p+2}$ to $H^{p-1}(\RR)$, we need to show that
    $$\norm{\partial_{x}^{k} \left(\frac{B_{p}(\mb{\tau} + \epsilon \mb{\delta}_{i}) - B_{p}(\mb{\tau})}{\epsilon} - \partial_{i} B_{p}(\mb{\tau})\right)}_{L^{2}(\RR)}$$
    goes to $0$ as $\epsilon \to 0$ for all $k \in \rg{0}{p-1}$. From the expression of $\partial_{i} B_{p}$ above, we infer that $\partial_{i} B_{p}$ belongs to $H^{p-1}(\RR)$ owing to the double multiplicity of $\mb{\tau}_{i}$. Observe then that for all $\epsilon > 0$, the integrand of the norm above belongs to $H^{p-1}(\RR)$ as a linear combination of elements of $H^{p-1}(\RR)$. The integrand also has compact support (say $\cc{\mb{\tau}_{1} - 1}{\mb{\tau}_{p+2} + 1}$ when $\abs{\epsilon} \leq 1$) and tends to $0$ pointwise by definition of $\partial_{i} B_{p}$. The dominated convergence theorem ensures that the norms above tend to $0$ as $\epsilon \to 0$. This in turn shows $B_{p}$ is differentiable in the $H^{p-1}(\RR)$-norm.

    Finally, the formula \eqref{eq:derivative-BSpline} still holds for $p = 0$ since, by the properties of divided differences, $N_{0}(a, a) = H'(a) = \delta_{a}$ for all $a \in \RR$, where $H$ is the Heaviside function.
\end{proof}

\subsection{Proof of \texorpdfstring{\lem{derivativeLinearForm}}{Lemma~\ref{lem:derivativeLinearForm}}}
\label{app:derivativeLinearForm}

\begin{proof}
    Let $p \geq 0$ and $f \in H^{-(p-1)}(I)$. Observe that $\partial_{i} B_{p}(\mb{\tau}) \in H^{p-1}(I)$ for all $\mb{\tau} \in \md{K}_{p+2}$, so $\ell_{f}(B_{p}(\mb{\tau}))$ and $\ell_{f}(\partial_{i} B_{p}(\mb{\tau}))$ make sense. Let $\epsilon \in \RR$ such that $\mb{\tau}_{\epsilon} \doteq \mb{\tau} + \epsilon \mb{\delta}_{i} \in \md{K}_{p+2}$. We obtain the bound
    $$\abs{\frac{\ell_{f}(B_{p}(\mb{\tau}_{\epsilon})) - \ell_{f}(B_{p}(\mb{\tau}))}{\epsilon} - \ell_{f}(\partial_{i} B_{p}(\mb{\tau}))} \leq \norm{f}_{H^{-(p-1)}(I)} \norm{\frac{B_{p}(\mb{\tau}_{\epsilon}) - B_{p}(\mb{\tau})}{\epsilon} - \partial_{i} B_{p}(\mb{\tau})}_{H^{p-1}(I)}.$$
    By \lem{BsplineDerivatives}, the differentiability of $B_{p}$ with respect to the $i$-th knot in the $H^{p-1}(I)$ topology implies the right hand-side goes to $0$ as $\epsilon \to 0$. This shows $\ell_{f} \circ B_{p}$ is differentiable with respect to the $i$-th knot and $\partial_{i} \ell_{f}(B_{p}(\mb{\tau})) = \ell_{f}(\partial_{i} B_{p}(\mb{\tau}))$.
\end{proof}

\subsection{Proof of \texorpdfstring{\lem{derivativeBilinearForm}}{Lemma~\ref{lem:derivativeBilinearForm}}}
\label{app:derivativeBilinearForm}

\begin{proof}
    Observe that $\partial_{x}^{k} B_{p}(\mb{\tau}) \in H^{p-k}(\RR)$ and $\partial_{i} \partial_{x}^{k} B_{p}(\mb{\tau}) \in H^{p-k-1}(\RR)$ for all $\mb{\tau} \in \md{K}_{p+2}$ and $\partial_{x}^{k} B_{q}(\mb{\sigma}) \in H^{q-k}(\RR)$ for all $\mb{\sigma} \in \md{K}_{q+2}$.

        {\bfseries Case 1: $r \geq 1$; $p \geq k + 2$, $q \geq k + 1$.}
    By a classical Sobolev embedding theorem, \cite[Theorem 4.39]{adams2003sobolev}, it follows that $(\partial_{x}^{k} B_{p}(\mb{\tau})) (\partial_{x}^{k} B_{q}(\mb{\tau})) \in H^{s}(I)$ provided $s \doteq \min(p-k, q-k) > 1/2$; that is, $p \geq k+1$ and $q \geq k + 1$. Similarly $(\partial_{i} \partial_{x}^{k} B_{p}(\mb{\tau})) (\partial_{x}^{k} B_{q}(\mb{\sigma})) \in H^{r}(I)$ provided $r \doteq \min(p-k-1, q-k) > 1/2$; that is, $p \geq k + 2$ and $q \geq k + 1$. Given $f \in H^{-r}(I)$, these properties ensure $a_{k, f}(B_{p}(\mb{\tau}), B_{q}(\mb{\sigma}))$ and $a_{k, f}(\partial_{i} B_{p}(\mb{\tau}), B_{q}(\mb{\sigma}))$ make sense.

    Let $\epsilon \in \RR$ such that $\mb{\tau}_{\epsilon} \doteq \mb{\tau} + \epsilon \mb{\delta}_{i} \in \md{K}_{p+2}$. We bound
    \begin{align*}
         & \abs{\frac{a_{k, f}(B_{p}(\mb{\tau}_{\epsilon}), B_{q}(\mb{\sigma})) - a_{k, f}(B_{p}(\mb{\tau}), B_{q}(\mb{\sigma}))}{\epsilon} - a_{k, f}(\partial_{i} B_{p}(\mb{\tau}), B_{q}(\mb{\sigma}))}                                        \\
         & \leq \norm{f}_{H^{-r}(I)} \norm{\partial_{x}^{k} \left(\frac{B_{p}(\mb{\tau}_{\epsilon}) - B_{p}(\mb{\tau})}{\epsilon} - \partial_{i} B_{p}(\mb{\tau})\right) \partial_{x}^{k} B_{q}(\mb{\sigma})}_{H^{r}(I)}                          \\
         & \leq K \norm{f}_{H^{-r}(I)} \norm{\partial_{x}^{k} \left(\frac{B_{p}(\mb{\tau}_{\epsilon}) - B_{p}(\mb{\tau})}{\epsilon} - \partial_{i} B_{p}(\mb{\tau})\right)}_{H^{r}(\RR)} \norm{\partial_{x}^{k} B_{q}(\mb{\sigma})}_{H^{r}(\RR)}, \\
         & \leq K \norm{f}_{H^{-r}(I)} \norm{\frac{B_{p}(\mb{\tau}_{\epsilon}) - B_{p}(\mb{\tau})}{\epsilon} - \partial_{i} B_{p}(\mb{\tau})}_{H^{p-1}(\RR)} \norm{\partial_{x}^{k} B_{q}(\mb{\sigma})}_{H^{q}(\RR)},
    \end{align*}
    where the third inequality follows from \cite[Theorem 4.39]{adams2003sobolev}, and the constant $K$ depends on $p$, $q$ and $I$. The fourth inequality follows from the definition of $r$ and the inequality $\norm{\partial_{x}^{k} u}_{H^{p}(I)} \leq \norm{u}_{H^{p+k}(I)}$ for all $u \in H^{p+k}(I)$. The final upper bound tends to $0$ by the differentiability of $B_{p}(\mb{\tau})$ with respect to the $i$-th knot in the $H^{p-1}(I)$ topology (\lem{BsplineDerivatives}).

        {\bfseries Case 2: $r = 0$; $p \geq k + 1$, $q = k$ or $p = k+1$, $q \geq k$}
    Observe that $\partial_{x}^{k} B_{q}(\mb{\sigma})$ is piecewise continuous and bounded pointwise. Since $\partial_{x}^{k} B_{p}(\mb{\tau}) \in L^{2}(\RR)$, we infer that $(\partial_{x}^{k} B_{p}(\mb{\tau})) (\partial_{x}^{k} B_{q}(\mb{\tau})) \in L^{2}(\RR)$. Similarly, $\partial_{i} \partial_{x}^{k} B_{p}(\mb{\tau}) \in L^{2}(\RR)$ so $(\partial_{i} \partial_{x}^{k} B_{p}(\mb{\tau})) (\partial_{x}^{k} B_{q}(\mb{\tau})) \in L^{2}(\RR)$. Provided $f \in L^{2}(I)$, this ensures $a_{k, f}(B_{p}(\mb{\tau}), B_{q}(\mb{\sigma}))$ and $a_{k, f}(\partial_{i} B_{p}(\mb{\tau}), B_{q}(\mb{\sigma}))$ make sense. Here again, we bound
    \begin{align*}
         & \abs{\frac{a_{k, f}(B_{p}(\mb{\tau}_{\epsilon}), B_{q}(\mb{\sigma})) - a_{k, f}(B_{p}(\mb{\tau}), B_{q}(\mb{\sigma}))}{\epsilon} - a_{k, f}(\partial_{i} B_{p}(\mb{\tau}), B_{q}(\mb{\sigma}))}              \\
         & \leq \norm{f}_{L^{2}(I)} \norm{\partial_{x}^{k} \left(\frac{B_{p}(\mb{\tau}_{\epsilon}) - B_{p}(\mb{\tau})}{\epsilon} - \partial_{i} B_{p}(\mb{\tau})\right) \partial_{x}^{k} B_{q}(\mb{\sigma})}_{L^{2}(I)} \\
         & \leq \norm{f}_{L^{2}(I)} \norm{\frac{B_{p}(\mb{\tau}_{\epsilon}) - B_{p}(\mb{\tau})}{\epsilon} - \partial_{i} B_{p}(\mb{\tau})}_{H^{p-1}(\RR)} \norm{\partial_{x}^{k} B_{q}(\mb{\sigma})}_{L^{\infty}(\RR)}.
    \end{align*}
    The last inequality comes from $B_{q}(\mb{\sigma})$ being bounded and $k \leq p - 1$. The right hand-side tends to $0$ as $\epsilon \to 0$ by the differentiability of $B_{p}(\mb{\tau})$ in the $H^{p-1}(I)$ topology (\lem{BsplineDerivatives}).

        {\bfseries Case 3: $r = -1$; $p = k$, $q \geq k + 1$.}
    Observe that $\partial_{x}^{k} B_{p}(\mb{\tau})$ is piecewise constant and bounded, so it is enough to consider the case where $u$ is the indicator function of the segment $\oo{c}{d}$, for any $c < d \in \RR$. In that case,
    $$a_{k, f}(u, B_{q}(\mb{\sigma})) = \int_{I \cap \oo{c}{d}} f \partial_{x}^{k} B_{q}(\mb{\sigma}) \dd{x}.$$
    Owing to the assumption $q \geq k + 1$, we have $\partial_{x}^{k} B_{q}(\mb{\sigma}) \in C^{0}(\RR)$. Similarly, in that case $r = -1$ so $f \in H^{1}(I) \subset C^{0}(I)$. This implies $f \partial_{x} B_{q}(\mb{\sigma}) \in C^{0}(I)$. Therefore $a_{k, f}(u, B_{q}(\mb{\sigma}))$ is differentiable with respect to $c$ and $d$. By the Leibniz integral rule, it follows that
    $$\partial_{c} a_{k, f}(u, B_{q}(\mb{\sigma})) = -f(c) \partial_{x}^{k} B_{q}(\mb{\sigma})(c) \delta_{c \in \overline{I}} = \int_{I} f (\partial_{c} u) (\partial_{x}^{k} B_{q}(\mb{\sigma})) \dd{x},$$
    and similarly for the derivative with respect to $d$. The last integral is understood as a duality pairing between $\partial_{c} u$ and $f \partial_{x}^{k} B_{q}(\mb{\sigma})$. By linearity, we infer that $a_{k, f}(B_{p}(\mb{\tau}), B_{q}(\mb{\sigma}))$ is differentiable with respect to the $i$-th knot and $\partial_{i} a_{k, f}(B_{p}(\mb{\tau}), B_{q}(\mb{\sigma})) = a_{k, f}(\partial_{i} B_{p}(\mb{\tau}), B_{q}(\mb{\sigma}))$.
\end{proof}

\subsection{Proof of \texorpdfstring{\lem{boundedness}}{Lemma~\ref{lem:boundedness}}}
\label{app:boundedness}

\begin{proof}
    {\bfseries Part 1: Bound for $B_{p}$.}
    Let $p \geq 0$ and $\mb{\tau} \in \md{K}_{p+2}$. The bound for the $L^{2}$ norm of a B-spline \eqref{eq:bound-bspline-norm} readily provides
    $$\norm{B_{p}(\mb{\tau})}_{L^{2}(\RR)}^{2} \leq C_{p}^{2} \abs{\mb{\tau}}.$$

    {\bfseries Part 2: Bound for $\partial_{x} B_{p}$.}
    For all $p \geq 1$, the recursion for $\partial_{x} B_{p}$ and the non-negativity of $N_{p}$ yield
    \begin{align*}
        \norm{\partial_{x} B_{p}(\mb{\tau})}_{L^{2}(\RR)}^{2} & \leq p^{2} (\norm{N_{p-1}(\mb{\tau} \remBeg)}_{L^{2}(\RR)}^{2} + \norm{N_{p-1}(\mb{\tau} \remEnd)}_{L^{2}(\RR)}^{2}) \\
                                                              & \leq p^{2} C_{p-1}^{2} (\abs{\mb{\tau} \remBeg}^{-1} + \abs{\mb{\tau} \remEnd}^{-1})                                 \\
                                                              & \leq 2 p C_{p-1}^{2} h(\mb{\tau})^{-1},
    \end{align*}
    where we used \eqref{eq:bound-knots} in the last inequality.

        {\bfseries Part 3: Bound for $\partial_{i} B_{p}$.}
    If $p \geq 1$, the recursion for the derivative of $B_{p}(\mb{\tau})$ with respect to the $i$-th knot provides
    \begin{align*}
        \norm{\partial_{i} B_{p}(\mb{\tau})}_{L^{2}(\RR)}^{2} & \leq \norm{N_{p}(\mb{\tau} \oplus \mb{\tau}_{i} \remBeg)}_{L^{2}(\RR)}^{2} \delta_{i \neq 1} + \norm{N_{p}(\mb{\tau} \oplus \mb{\tau}_{i} \remEnd)}_{L^{2}(\RR)}^{2} \delta_{i \neq p + 2} \\
                                                              & \leq C_{p}^{2} (\abs{\mb{\tau} \oplus \mb{\tau}_{i} \remBeg}^{-1} + \abs{\mb{\tau} \oplus \mb{\tau}_{i} \remEnd}^{-1})                                                                     \\
                                                              & \leq \frac{2 C_{p}^{2}}{p} h(\mb{\tau})^{-1},
    \end{align*}
    where we used the non-negativity of $N_{p}$ in the first inequality, the bound on the $L^{2}$ norm of $N_{p}$ \eqref{eq:bound-bspline-norm} for the second inequality, and the bound on the minimum mesh size \eqref{eq:bound-knots} for the third inequality, noting that duplicating a knot does not change the knot span.

        {\bfseries Part 4: Bound for $\partial_{i} \partial_{x} B_{p}$.}
    For all $p \geq 2$ and $i \in \rg{1}{p+2}$, the recursions for $\partial_{i} B_{p}$ and $\partial_{x} B_{p}$ provide
    \begin{align*}
        \partial_{i} \partial_{x} B_{p}(\mb{\tau}) & = \frac{-p}{\abs{\mb{\tau} \oplus \mb{\tau}_{i} \remBeg}} (N_{p-1}(\mb{\tau} \oplus \mb{\tau}_{i} \remBeg \remBeg) - N_{p-1}(\mb{\tau} \oplus \mb{\tau}_{i} \remBeg \remEnd)) \delta_{i \neq 1}   \\
                                                   & + \frac{p}{\abs{\mb{\tau} \oplus \mb{\tau}_{i} \remEnd}} (N_{p-1}(\mb{\tau} \oplus \mb{\tau}_{i} \remEnd \remBeg) - N_{p-1}(\mb{\tau} \oplus \mb{\tau}_{i} \remEnd \remEnd)) \delta_{i \neq p+2}.
    \end{align*}
    Using the triangle inequality and the non-negativity of $N_{p}$, we obtain
    \begin{align*}
        \norm{\partial_{i} \partial_{x} B_{p}(\mb{\tau})}_{L^{2}(\RR)} & \leq \frac{p}{\abs{\mb{\tau} \oplus \mb{\tau}_{i} \remBeg}} (\norm{N_{p-1}(\mb{\tau} \oplus \mb{\tau}_{i} \remBeg \remBeg)}_{L^{2}(\RR)}^{2} + \norm{N_{p-1}(\mb{\tau} \oplus \mb{\tau}_{i} \remBeg \remEnd)}_{L^{2}(\RR)}^{2})^{1/2} \delta_{i \neq 1} \\
                                                                       & + \frac{p}{\abs{\mb{\tau} \oplus \mb{\tau}_{i} \remEnd}} (\norm{N_{p-1}(\mb{\tau} \oplus \mb{\tau}_{i} \remEnd \remBeg)}_{L^{2}(\RR)}^{2} + \norm{N_{p-1}(\mb{\tau} \oplus \mb{\tau}_{i} \remEnd \remEnd)}_{L^{2}(\RR)}^{2})^{1/2} \delta_{i \neq p+2}  \\
                                                                       & \leq \frac{p}{\abs{\mb{\tau} \oplus \mb{\tau}_{i} \remBeg}} C_{p-1} (\abs{\mb{\tau} \oplus \mb{\tau}_{i} \remBeg \remBeg}^{-1} + \abs{\mb{\tau} \oplus \mb{\tau}_{i} \remBeg \remEnd}^{-1})^{1/2} \delta_{i \neq 1}                                     \\
                                                                       & + \frac{p}{\abs{\mb{\tau} \oplus \mb{\tau}_{i} \remEnd}} C_{p-1} (\abs{\mb{\tau} \oplus \mb{\tau}_{i} \remEnd \remBeg}^{-1} + \abs{\mb{\tau} \oplus \mb{\tau}_{i} \remEnd \remEnd}^{-1})^{1/2} \delta_{i \neq p+2}                                      \\
                                                                       & \leq \frac{2 \sqrt{2} C_{p-1}}{p - 1} h(\mb{\tau})^{-3/2}.
    \end{align*}
\end{proof}

\subsection{Proof of \texorpdfstring{\lem{holder}}{Lemma~\ref{lem:holder}}}
\label{app:holder}

\begin{proof}
    Let $p \geq 0$ and $\mb{\sigma}, \mb{\tau} \in \md{K}_{p+2}$. It is always possible to connect $\mb{\sigma}$ and $\mb{\tau}$ by changing a single knot at a time while preserving the minimum mesh size $h \doteq h(\mb{\sigma}, \mb{\tau})$ for all the intermediate knot vectors. For the four types of inequalities, we start by perturbating a single knot and we conclude using the triangle and Hölder or Cauchy-Schwarz inequalities. By symmetry in $\mb{\sigma}$ and $\mb{\tau}$, we can also restrict our attention to positive perturbations.

    Let $k \in \rg{1}{p+2}$ and $\epsilon > 0$. We consider the perturbation of the $k$-th knot by $\epsilon$, which we write $\mb{\tau} + \epsilon \mb{\delta}_{k}$ such that $(\mb{\tau} + \epsilon \mb{\delta}_{k})_{i} = \mb{\tau}_{i} + \epsilon \delta_{i = k}$. From the remark above, we can assume that $h(\mb{\tau} + \epsilon \mb{\delta}_{k}) \geq h_{\min}$ so in particular $\epsilon$ must be such that $\mb{\tau}_{k-1} + h_{\min} \leq \mb{\tau}_{k} + \epsilon \leq \mb{\tau}_{k+1} - h_{\min}$, with conventions $\mb{\tau}_{0} = -\infty$ and $\mb{\tau}_{p+3} = \infty$. For all $x \in \RR$, we define $f_{x}: y \mapsto (y - x)_{+}^{p}$. For conciseness, we write
    $$\Delta_{k} B_{p}(\mb{\tau}, \epsilon) \doteq B_{p}(\mb{\tau} + \epsilon \mb{\delta}_{k}) - B_{p}(\mb{\tau}).$$
    We also introduce $\mb{\tau}_{\epsilon} \doteq \mb{\tau} \oplus (\mb{\tau}_{k} + \epsilon) \in \md{K}_{p+3}$ for convenience.

    We derive general bounds for arbitrary degree. As we will see, these bounds sometimes fail for low orders and these cases need special treatment. These cases are identified within the proof, and the corresponding arguments are given in \cite{thesis}.

    {\bfseries Part 1: Bound for $B_{p}$.}
    The formula for the difference of two divided differences provides
    \begin{align*}
        \Delta_{k} B_{p}(\mb{\tau}, \epsilon)(x) & = \epsilon [(\mb{\tau} \remEnd) \oplus (\mb{\tau}_{k} + \epsilon)] f_{x} \delta_{k \neq p + 2} - \epsilon [(\mb{\tau} \remBeg) \oplus (\mb{\tau}_{k} + \epsilon)] f_{x} \delta_{k \neq 1}, \\
                                                 & = \epsilon [(\mb{\tau} \oplus (\mb{\tau}_{k} + \epsilon)) \remEnd] f_{x} \delta_{k \neq p + 2} - \epsilon [(\mb{\tau} \oplus (\mb{\tau}_{k} + \epsilon)) \remBeg] f_{x} \delta_{k \neq 1},
    \end{align*}
    Note that $\remBeg$ and $\remEnd$ commute with $\oplus$ in both cases due to the restriction on $k$, and by design of $\epsilon$. From the definition of $N_{p}$ and the positivity of $\epsilon$, we rewrite
    \begin{equation}
        \label{eq:decomp-spatial}
        \Delta_{k} B_{p}(\mb{\tau}, \epsilon) = \epsilon (N_{p}(\mb{\tau}_{\epsilon} \remEnd) \delta_{k \neq p + 2} - N_{p}(\mb{\tau}_{\epsilon} \remBeg) \delta_{k \neq 1}).
    \end{equation}
    Using the non-negativitiy of N-splines to ignore the inner product between $N_{p}(\mb{\tau}_{\epsilon} \remBeg)$ and $N_{p}(\mb{\tau}_{\epsilon} \remEnd)$ in the expansion of the squared norm of $\Delta_{k} B_{p}(\mb{\tau}, \epsilon)$, and the bound for the $L^{2}$ norm of a B-spline \eqref{eq:bound-bspline-norm}, we then find
    \begin{align*}
        \norm{\Delta_{k} B_{p}(\mb{\tau}, \epsilon)}_{L^{2}(\RR)}^{2} & \leq \epsilon^{2} (\norm{N_{p}(\mb{\tau}_{\epsilon} \remEnd)}_{L^{2}(\RR)}^{2} \delta_{k \neq p + 2} + \norm{N_{p}(\mb{\tau}_{\epsilon} \remBeg)}_{L^{2}(\RR)}^{2} \delta_{k \neq 1}) \\
                                                                      & \leq \epsilon^{2} C_{p}^{2} (\abs{\mb{\tau}_{\epsilon} \remEnd}^{-1} \delta_{k \neq p+2} + \abs{\mb{\tau}_{\epsilon} \remBeg}^{-1} \delta_{k \neq 1}).
    \end{align*}
    We separate two cases depending on $p$.

        {\bfseries $\bullet$ Case 1 ($p = 0$).}
    If $p = 0$, then a simple examination of the cases $k = 1$ and $k = 2$ shows that the right-hand side is equal to $\epsilon C_{0}^{2}$. Using the triangle inequality to sum over the two perturbations and Hölder inequality with $p = 4$ and $q = 4/3$, we find
    $$\norm{B_{0}(\mb{\sigma}) - B_{0}(\mb{\tau})}_{L^{2}(\RR)} \leq C_{0} (\abs{\mb{\sigma}_{1} - \mb{\tau}_{1}}^{1/2} + \abs{\mb{\sigma}_{2} - \mb{\tau}_{2}}^{1/2}) \leq \norm{(1, 1)}_{4/3} C_{0} \norm{\mb{\sigma} - \mb{\tau}}_{2}^{1/2}$$
    for all $\mb{\sigma}, \mb{\tau} \in \md{K}_{2}$. The constant involving a $4/3$-norm is smaller than $2$.

        {\bfseries $\bullet$ Case 2 ($p \geq 1$).}
    If $p \geq 1$, the lower bound for the span of edited knots \eqref{eq:bound-knots} yields
    $$\norm{\Delta_{k} B_{p}(\mb{\tau}, \epsilon)}_{L^{2}(\RR)}^{2} \leq \frac{\delta_{k \neq 1} + \delta_{k \neq p+2}}{p} C_{p}^{2} h^{-1} \epsilon^{2}.$$
    Using the triangle inequality again to sum over the $p+2$ perturbations and the Cauchy-Schwarz inequality, we finally reach
    $$\norm{B_{p}(\mb{\sigma}) - B_{p}(\mb{\tau})}_{L^{2}(\RR)} \leq \frac{1}{\sqrt{p}} C_{p} h^{-1/2} \sum_{k = 1}^{p + 2} (\delta_{k \neq 1} + \delta_{k \neq p+2})^{1/2} \abs{\mb{\sigma}_{k} - \mb{\tau}_{k}} \leq \sqrt{\frac{2 (p + 1)}{p}} C_{p} h^{-1/2} \norm{\mb{\sigma} - \mb{\tau}}_{2},$$
    for all $\mb{\sigma}, \mb{\tau} \in \md{K}_{p+2}$..

        {\bfseries Part 2: Bound for $\partial_{x} B_{p}$.}
    Assume $p \geq 1$. Taking the spatial derivative of \eqref{eq:decomp-spatial} and using the expression of the spatial derivative of a B-spline, we find
    \begin{align*}
        \partial_{x} \Delta_{k} B_{p}(\mb{\tau}, \epsilon) & = -\frac{p \epsilon}{\abs{\mb{\tau}_{\epsilon} \remEnd}} \left(N_{p-1}(\mb{\tau}_{\epsilon} \remEnd \remBeg) - N_{p-1}(\mb{\tau}_{\epsilon} \remEnd \remEnd)\right) \delta_{k \neq p+2} \\
                                                           & + \frac{p \epsilon}{\abs{\mb{\tau}_{\epsilon} \remBeg}} \left(N_{p-1}(\mb{\tau}_{\epsilon} \remBeg \remBeg) - N_{p-1}(\mb{\tau}_{\epsilon} \remBeg \remEnd)\right) \delta_{k \neq 1}.
    \end{align*}
    We bound the $L^{2}$ norm of $\partial_{x} \Delta_{k} B_{p}(\mb{\tau}, \epsilon)$ using the triangle inequality to split these two terms. Using the non-negativitiy of $N_{p}$ and the bound for the $L^{2}$ norm of a B-spline \eqref{eq:bound-bspline-norm}, we reach
    \begin{align*}
        \norm{\partial_{x} \Delta_{k} B_{p}(\mb{\tau}, \epsilon)}_{L^{2}(\RR)} & \leq \frac{p \epsilon}{\abs{\mb{\tau}_{\epsilon} \remEnd}} C_{p-1} \left(\abs{\mb{\tau}_{\epsilon} \remEnd \remBeg}^{-1} + \abs{\mb{\tau}_{\epsilon} \remEnd \remEnd}^{-1}\right)^{1/2} \delta_{k \neq p + 2} \\
                                                                               & + \frac{p \epsilon}{\abs{\mb{\tau}_{\epsilon} \remBeg}} C_{p-1} \left(\abs{\mb{\tau}_{\epsilon} \remBeg \remBeg}^{-1} + \abs{\mb{\tau}_{\epsilon} \remBeg \remEnd}^{-1}\right)^{1/2} \delta_{k \neq 1}.
    \end{align*}
    Here again, we separate two cases depending on the value of $p$.

        {\bfseries $\bullet$ Case 1 ($p = 1$).}
    If $p = 1$ and $k = 1$, the right-hand side simplifies to
    \begin{align*}
        \norm{\partial_{x} \Delta_{1} B_{1}(\mb{\tau}, \epsilon)}_{L^{2}(\RR)} & \leq \frac{C_{0} \epsilon}{\mb{\tau}_{2} - \mb{\tau}_{1}} \left(\frac{1}{\mb{\tau}_{2} - \mb{\tau}_{1} - \epsilon} + \frac{1}{\epsilon}\right)^{1/2} \\
                                                                               & = \frac{C_{0} \epsilon^{1/2}}{(\mb{\tau}_{2} - \mb{\tau}_{1})^{1/2} (\mb{\tau}_{2} - \mb{\tau}_{1} - \epsilon)^{1/2}}                                \\
                                                                               & \leq C_{0} h^{-1} \epsilon^{1/2}.
    \end{align*}
    The same bound is obtained for the perturbation with respect to the two other knots, up to a factor $2$ for the second knot. Using the Hölder inequality, we conclude
    $$\norm{\partial_{x} B_{1}(\mb{\sigma}) - \partial_{x} B_{1}(\mb{\tau})}_{L^{2}(\RR)} \leq \norm{(1, 2, 1)}_{4/3} C_{0} h^{-1} \norm{\mb{\sigma} - \mb{\tau}}_{2}^{1/2}$$
    for all $\mb{\sigma}, \mb{\tau} \in \md{K}_{3}$. The constant involving a $4/3$-norm is smaller than $4$.

        {\bfseries $\bullet$ Case 2 ($p \geq 2$).}
    If $p \geq 2$, then using the lower bound on the knot span of edited knots \eqref{eq:bound-knots}, we infer
    $$\norm{\partial_{x} \Delta_{k} B_{p}(\mb{\tau}, \epsilon)}_{L^{2}(\RR)} \leq (\delta_{k \neq 1} + \delta_{k \neq p + 2}) \sqrt{\frac{2}{p - 1}} C_{p-1} h^{-3/2} \epsilon.$$
    Again, the triangle and Cauchy-Schwarz inequalities provide
    $$\norm{\partial_{x} B_{p}(\mb{\sigma}) - \partial_{x} B_{p}(\mb{\tau})}_{L^{2}(\RR)} \leq 2 \sqrt{\frac{p + 1}{p - 1}} C_{p-1} h^{-3/2} \norm{\mb{\sigma} - \mb{\tau}}_{2}$$
    for all $\mb{\sigma}, \mb{\tau} \in \md{K}_{p+2}$.

        {\bfseries Part 3: Bound for $\partial_{i} B_{p}$.}
    Assume $p \geq 1$ and let $i \in \{1, \ldots, p+2\}$. Expressing $B_{p}$ via divided differences, we decompose
    \begin{align*}
        \partial_{i} \Delta_{k} B_{p}(\mb{\tau}, \epsilon) & = ([(\mb{\tau} + \epsilon \mb{\delta}_{k}) \remBeg \oplus (\mb{\tau}_{i} + \epsilon \delta_{i = k})] f_{x} - [\mb{\tau} \remBeg \oplus \mb{\tau}_{i}] f_{x}) \delta_{i \neq 1} \delta_{k \neq 1}      \\
                                                           & + ([(\mb{\tau} + \epsilon \mb{\delta}_{k}) \remEnd \oplus (\mb{\tau}_{i} + \epsilon \delta_{i = k})] f_{x} - [\mb{\tau} \remEnd \oplus \mb{\tau}_{i}] f_{x}) \delta_{i \neq p+2} \delta_{k \neq p+2}.
    \end{align*}
    Note that $\remBeg$ and $\remEnd$ commute with $\oplus$ in both cases due to the restriction on $k$ and $i$, and by design of $\epsilon$. The case $k = i$ requires special care as $\epsilon$ appears twice in $(\mb{\tau} + \epsilon \mb{\delta}_{k}) \oplus (\mb{\tau}_{i} + \epsilon \delta_{i = k})$. To deal with one $\epsilon$ at a time, we consider the intermediate knot vector $(\mb{\tau} + \epsilon \mb{\delta}_{k}) \oplus \mb{\tau}_{i}$ and rewrite
    \begin{align*}
        \partial_{i} \Delta_{k} B_{p}(\mb{\tau}, \epsilon) & = (\epsilon [\mb{\tau}_{\epsilon \epsilon k} \remBeg] f_{x} \delta_{i = k}  + \epsilon [\mb{\tau}_{\epsilon k i} \remBeg] f_{x}) \delta_{i \neq 1} \delta_{k \neq 1}     \\
                                                           & - (\epsilon [\mb{\tau}_{\epsilon \epsilon k} \remEnd] f_{x} \delta_{i = k} + \epsilon [\mb{\tau}_{\epsilon k i} \remEnd] f_{x}) \delta_{i \neq p+2} \delta_{k \neq p+2},
    \end{align*}
    where $\mb{\tau}_{\epsilon k i} \doteq \mb{\tau} \oplus (\mb{\tau}_{k} + \epsilon) \oplus \mb{\tau}_{i}$ and $\mb{\tau}_{\epsilon \epsilon k} \doteq \mb{\tau} \oplus (\mb{\tau}_{k} + \epsilon) \oplus (\mb{\tau}_{k} + \epsilon)$. We can now rewrite this expression in terms of N-splines:
    \begin{align}
        \label{eq:decomp-param}
        \partial_{i} \Delta_{k} B_{p}(\mb{\tau}, \epsilon) & = \frac{\epsilon}{\abs{\mb{\tau}_{\epsilon \epsilon k} \remBeg}} (N_{p}(\mb{\tau}_{\epsilon \epsilon k} \remBeg \remBeg) - N_{p}(\mb{\tau}_{\epsilon \epsilon k} \remBeg \remEnd)) \delta_{i = k} \delta_{k \neq 1}                 \\
                                                           & - \frac{\epsilon}{\abs{\mb{\tau}_{\epsilon \epsilon k} \remEnd}} (N_{p}(\mb{\tau}_{\epsilon \epsilon k} \remEnd \remBeg) - N_{p}(\mb{\tau}_{\epsilon \epsilon k} \remEnd \remEnd)) \delta_{i = k} \delta_{k \neq p+2} \nonumber     \\
                                                           & + \frac{\epsilon}{\abs{\mb{\tau}_{\epsilon k i} \remBeg}} (N_{p}(\mb{\tau}_{\epsilon k i} \remBeg \remBeg) - N_{p}(\mb{\tau}_{\epsilon k i} \remBeg \remEnd)) \delta_{i \neq 1} \delta_{k \neq 1}                         \nonumber \\
                                                           & - \frac{\epsilon}{\abs{\mb{\tau}_{\epsilon k i} \remEnd}} (N_{p}(\mb{\tau}_{\epsilon k i} \remEnd \remBeg) - N_{p}(\mb{\tau}_{\epsilon k i} \remEnd \remEnd)) \delta_{i \neq p+2} \delta_{k \neq p+2}. \nonumber
    \end{align}
    The case $p = 1$ requires a separate treatment. Assuming $p \geq 2$, using the lower bound on the knot spans provided by \eqref{eq:bound-knots} and applying the same method as before, we find
    $$\norm{\partial_{i} \Delta_{k} B_{p}(\mb{\tau}, \epsilon)}_{L^{2}(\RR)} \leq \frac{\sqrt{2} C_{p} D_{k, i}}{p \sqrt{p - 1}} h^{-3/2} \epsilon,$$
    where $D_{k, i} = (1 + \delta_{i = k}) (\delta_{i \neq 1} \delta_{k \neq 1} + \delta_{i \neq p+2} \delta_{k \neq p+2})$. Using the Cauchy-Schwarz inequality, we find
    $$\norm{\partial_{i} B_{p}(\mb{\sigma}) - \partial_{i} B_{p}(\mb{\tau})}_{L^{2}(\RR)} \leq \frac{\sqrt{2} C_{p} D_{i}}{p \sqrt{p - 1}} h^{-3/2} \norm{\mb{\sigma} - \mb{\tau}}_{2},$$
    where $D_{i}^{2} = \sum_{k = 1}^{p + 2} D_{k, i}^{2}$. We compute $D_{i}^{2} = p + 4$ if $i \in \{1, p + 2\}$ and $D_{i}^{2} = 4 p + 14$ otherwise. Altogether, $D_{i}$ is bounded by $\sqrt{4 p + 14}$.

        {\bfseries Part 4: Bound for $\partial_{i} \partial_{x} B_{p}$.}
    Assuming $p \geq 2$ and taking the spatial derivative of \eqref{eq:decomp-param}, we obtain

    \resizebox{0.8\linewidth}{!}{
        \begin{minipage}{\linewidth}
            \begin{align}
                \label{eq:decomp-spatial-param}
                \partial_{i} \partial_{x} \Delta_{k} B_{p}(\mb{\tau}, \epsilon) & = - \frac{p \epsilon}{\abs{\mb{\tau}_{\epsilon \epsilon k} \remBeg} \abs{\mb{\tau}_{\epsilon \epsilon k} \remBeg \remBeg}} (N_{p-1}(\mb{\tau}_{\epsilon \epsilon k} \remBeg \remBeg \remBeg) - N_{p-1}(\mb{\tau}_{\epsilon \epsilon k} \remBeg \remBeg \remEnd)) \delta_{i = k} \delta_{k \neq 1}               \\
                                                                                & + \frac{p \epsilon}{\abs{\mb{\tau}_{\epsilon \epsilon k} \remBeg} \abs{\mb{\tau}_{\epsilon \epsilon k} \remBeg \remEnd}} (N_{p - 1}(\mb{\tau}_{\epsilon \epsilon k} \remBeg \remEnd \remBeg) - N_{p - 1}(\mb{\tau}_{\epsilon \epsilon k} \remBeg \remEnd \remEnd)) \delta_{i = k} \delta_{k \neq 1} \nonumber   \\
                                                                                & + \frac{p \epsilon}{\abs{\mb{\tau}_{\epsilon \epsilon k} \remEnd} \abs{\mb{\tau}_{\epsilon \epsilon k} \remEnd \remBeg}} (N_{p - 1}(\mb{\tau}_{\epsilon \epsilon k} \remEnd \remBeg \remBeg) - N_{p - 1}(\mb{\tau}_{\epsilon \epsilon k} \remEnd \remBeg \remEnd)) \delta_{i = k} \delta_{k \neq p+2} \nonumber \\
                                                                                & - \frac{p \epsilon}{\abs{\mb{\tau}_{\epsilon \epsilon k} \remEnd} \abs{\mb{\tau}_{\epsilon \epsilon k} \remEnd \remEnd}} (N_{p - 1}(\mb{\tau}_{\epsilon \epsilon k} \remEnd \remEnd \remBeg) - N_{p - 1}(\mb{\tau}_{\epsilon \epsilon k} \remEnd \remEnd \remEnd)) \delta_{i = k} \delta_{k \neq p+2} \nonumber \\
                                                                                & - \frac{p \epsilon}{\abs{\mb{\tau}_{\epsilon k i} \remBeg} \abs{\mb{\tau}_{\epsilon k i} \remBeg \remBeg}} (N_{p - 1}(\mb{\tau}_{\epsilon k i} \remBeg \remBeg \remBeg) - N_{p - 1}(\mb{\tau}_{\epsilon k i} \remBeg \remBeg \remEnd)) \delta_{i \neq 1} \delta_{k \neq 1} \nonumber                            \\
                                                                                & + \frac{p \epsilon}{\abs{\mb{\tau}_{\epsilon k i} \remBeg} \abs{\mb{\tau}_{\epsilon k i} \remBeg \remEnd}} (N_{p - 1}(\mb{\tau}_{\epsilon k i} \remBeg \remEnd \remBeg) - N_{p - 1}(\mb{\tau}_{\epsilon k i} \remBeg \remEnd \remEnd)) \delta_{i \neq 1} \delta_{k \neq 1} \nonumber                            \\
                                                                                & + \frac{p \epsilon}{\abs{\mb{\tau}_{\epsilon k i} \remEnd} \abs{\mb{\tau}_{\epsilon k i} \remEnd \remBeg}} (N_{p - 1}(\mb{\tau}_{\epsilon k i} \remEnd \remBeg \remBeg) - N_{p - 1}(\mb{\tau}_{\epsilon k i} \remEnd \remBeg \remEnd)) \delta_{i \neq p+2} \delta_{k \neq p+2} \nonumber                        \\
                                                                                & - \frac{p \epsilon}{\abs{\mb{\tau}_{\epsilon k i} \remEnd} \abs{\mb{\tau}_{\epsilon k i} \remEnd \remEnd}} (N_{p - 1}(\mb{\tau}_{\epsilon k i} \remEnd \remEnd \remBeg) - N_{p - 1}(\mb{\tau}_{\epsilon k i} \remEnd \remEnd \remEnd)) \delta_{i \neq p+2} \delta_{k \neq p+2}. \nonumber
            \end{align}
        \end{minipage}
    }

    The case $p = 2$ requires a separate treatment. When $p \geq 3$, using the non-negativity of $N_{p}$, proceeding as before with \eqref{eq:bound-bspline-norm} and \eqref{eq:bound-knots}, we find
    $$\norm{\partial_{i} \partial_{x} \Delta_{k} B_{p}(\mb{\tau}, \epsilon)}_{L^{2}(\RR)} \leq \frac{2 \sqrt{2} C_{p-1} D_{k, i}}{(p - 1) \sqrt{p - 2}} h^{-5/2} \epsilon,$$
    where $D_{k, i}$ is the same constant as in part 3. We thus infer
    $$\norm{\partial_{i} \partial_{x} B_{p}(\mb{\sigma}) - \partial_{i} \partial_{x} B_{p}(\mb{\tau})}_{L^{2}(\RR)} \leq \frac{2 \sqrt{2} C_{p-1} D_{i}}{(p - 1) \sqrt{p - 2}} h^{-5/2} \norm{\mb{\sigma} - \mb{\tau}}_{2},$$
    where $D_{i}$ is the same constant as above.
\end{proof}

\printbibliography

\end{document}